\documentclass[a4paper,11pt]{amsart}
\usepackage[body=16cm,top=4cm,bottom=4cm]{geometry} 
\usepackage[english,frenchb]{babel} 
\usepackage[utf8]{inputenc} 
\usepackage[T1]{fontenc} 
\usepackage{amsmath,amsfonts,amssymb,amscd,stmaryrd,bbm,amsthm,enumerate,mathrsfs,bbm,amsaddr}
\usepackage{enumerate} 
\usepackage[colorlinks=true,linkcolor=MidnightBlue,citecolor=OliveGreen]{hyperref}
\usepackage[dvipsnames]{xcolor} 

\usepackage[titletoc]{appendix}

\usepackage{pgf,tikz}
\usetikzlibrary{shapes,snakes,patterns,arrows,automata,positioning,patterns}

\newcommand{\cC}{\mathcal{C}}

\newcommand{\cF}{\mathcal{F}}

\newcommand{\BB}{\mathbbm{B}}

\newcommand{\EE}{\mathbbm{E}}

\newcommand{\PP}{\mathbbm{P}}

\newcommand{\RR}{\mathbbm{R}}

\newcommand{\bB}{\mathbf{B}}

\newcommand{\bX}{\mathbf{X}}

\newcommand{\eps}{\varepsilon}
\newcommand{\dd}{\mathrm{d}}

\theoremstyle{plain}
\newtheorem{fr-theorem}{Th\'eor\`eme}[section]
\newtheorem{fr-corollary}[fr-theorem]{Corollaire}
\newtheorem{fr-lemma}[fr-theorem]{Lemme}
\newtheorem{fr-proposition}[fr-theorem]{Proposition}

\newtheorem{theorem}{Theorem}[section]
\newtheorem{corollary}[theorem]{Corollary}
\newtheorem{lemma}[theorem]{Lemma}
\newtheorem{proposition}[theorem]{Proposition}

\theoremstyle{definition}
\newtheorem{fr-remark}[fr-theorem]{Remarque}
\newtheorem{fr-definition}[fr-theorem]{Définition}
\newtheorem{fr-notation}[fr-theorem]{Notation}

\newtheorem{remark}[theorem]{Remark}
\newtheorem{definition}[theorem]{Definition}

\newcommand{\assign}{:=}

\newcommand{\nobracket}{}
\newcommand{\nocomma}{}
\newcommand{\noplus}{}
\newcommand{\nosymbol}{}
\newcommand{\tmem}[1]{{\em #1\/}}

\newcommand{\tmop}[1]{\ensuremath{\operatorname{#1}}}
\newcommand{\tmscript}[1]{\text{\scriptsize{$#1$}}}

\newcommand{\tmtextit}[1]{{\itshape{#1}}}

\renewcommand{\div}{\mathrm{div}\,}

\begin{document}
\selectlanguage{english} 
\author{R\'emi Catellier} 
\address{Universit\'e de Rennes 1 - IRMAR - Centre Henri Lebesgue}
\date{\today}
\title{Rough linear transport equation with an irregular drift}

\begin{abstract}
 We study the linear transport equation
  \[ \frac{\partial}{\partial t} u ( t,x ) +b ( t,x ) \cdot \nabla u ( t,x ) +
     \nabla u ( t,x ) \cdot \frac{\partial}{\partial t} X ( t ) =0,
     \hspace{2em} u ( 0,x ) =u_{0} ( x ) \]
  where $b$ is a vectorfield of limited regularity and $X$ a vector-valued
  H{\"o}lder continuous driving term. Using the theory of controlled rough
  paths we give a meaning to the weak formulation of the PDE and solve that
  equation for smooth vectorfields $b$. In the case of the fractional Brownian
  motion a phenomenon of regularization by noise is displayed.
\end{abstract}

\maketitle

\tableofcontents

\section{Introduction}
In this work we study the driven linear transport equation
\begin{equation}
  \frac{\partial}{\partial t} u ( t,x ) +b ( t,x ) \cdot \nabla u ( t,x ) +
  \nabla u ( t,x ) \cdot \frac{\partial}{\partial t} X ( t ) =0, \hspace{2em}
  u ( 0,x ) =u_{0} ( x ) \label{eq:ur-transport}
\end{equation}
where $b:\mathbbm{R}^{d} \rightarrow \mathbbm{R}^{d}$ is a sufficiently smooth
vectorfield and $X: [0,T] \rightarrow \mathbbm{R}^{d}$ a vector-valued driving
term (which we always take such that $X_{0} =0$). We consider solutions $u$
and initial conditions $u_{0}$ which are only bounded functions of space so
that the transport equation has to be understood in the weak sense with
respect to the spatial variable.

When $X=0$, Di Perna and Lions~{\cite{diperna_ordinary_1989}} showed that
when $b \in L^{1} ([0,T];W^{1,1}_{\tmop{loc}} ( \mathbbm{R}^{d} ) )$ with
linear growth condition and $\tmop{div}  b \in L^{1} ([0,T] \times
\mathbbm{R}^{d} )$, a unique $L^{\infty}$ weak solution exists. Furthermore a
wide range of results follows this one, a partial survey can be found
in~{\cite{ambrosio_existence_2008}}.

Under weaker condition on regularity of the vectorfield $b$, the equation is
known to be ill-posed. Despite of that, Flandoli, Gubinelli and
Priola~{\cite{flandoli_well-posedness_2010,flandoli_remarks_2013}} showed that
when $b \in C ([0,T];\mathcal{C}^{\alpha}_{b} (\mathbbm{R}^{d} ))$ and
$\tmop{div}  b \in L^{p} ([0,T] \times \mathbbm{R}^{d} )$ for $p \geqslant 2$,
adding a Brownian perturbation $X$ allows to maintain uniqueness of
$L^{\infty}$ solutions which are strong in the probabilistic sense. In that
case the driving term in eq.~{\eqref{eq:ur-transport}} has to be understood as
a Stratonovitch integral against the Brownian motion $X$. Their result is
based on the regularization effect of the Brownian perturbation on the flow of
the characteristic equation
\begin{equation}
  \Phi_{t} ( x ) =x+ \int_{0}^{t} b_{q} ( \Phi_{q} ( x ) ) \dd q+X_{t} ,
  \hspace{2em} x \in \mathbbm{R}^{d} ,t \geqslant 0. \label{eq:caracteristics}
\end{equation}
The key estimates on this regularization effect depends quite heavily on
stochastic calculus techniques and so breaks down easily if we try to apply
the same approach to more general perturbations $X$, for example a fractional
Brownian motion of given Hurst parameter. Note that in a recent paper, Beck,
Flandoli Gubinelli and Maurelli {\cite{beck_stochastic_2014}} proved directly
regularization properties of the Brownian motion on the transport equation
with more general vector fields without relying on the flow of caracteristics,
but they still have to rely on stochastic calculus tools, in particular the
It{\^o} formula for Brownian semi--martingales.

While stochastic calculus is not available for the fractional Brownian motion
(fBm), in a recent paper, Catellier and Gubinelli
{\cite{catellier_averaging_2012}} proved that, at the level of the
characteristic equation~{\eqref{eq:caracteristics}}, the same phenomenon of
regularization by noise appears for arbitrary Hurst parameter $H \in (0,1)$ of
the fractional Brownian motion $X$. In particular, the phenomenon of
regularization gets stronger as $H$ gets smaller.

In the perspective of this last result it is then interesting to investigate
the regularization by noise phenomenon at the level of the transport
equation~{\eqref{eq:ur-transport}}. In order to do so we need first to give an
appropriate meaning to the transport equation with non-differentiable driver
$X$.

In a more general setting, Lions, Perthame and Souganidis
{\cite{lions_scalar_2014}} use the entropy solutions and the kinetic
formulation of scalar conservation laws to overcome the difficulty given by
the rough drivers. On that setting the authors use a one-dimensional irregular
path only. For a multidimensional noise, Caruana, Friz and Oberhauser
{\cite{caruana_rough_2011}} use Lyons' theory of rough
paths~{\cite{lyons_differential_1998}} and the notion of viscosity solutions.
These two approaches are based on a common idea. The irregular signal $X$ is
approximaded by a family $( X^{\varepsilon} )_{\varepsilon >0}$ of smooth
functions, and a solution $u^{\varepsilon}$ of the approximate equation
constructed. Then, using suitable a priori bounds, the authors show that the
solutions converge to a function $u$ which is then \tmtextit{defined} to be
the solution of the equation. In these approaches the equation is replaced by
a limiting procedure and no attempts are made to investigate the equation
satisfied by the limiting object.

This way of proceeding is not very useful in the context of the regularization
by noise phenomenon we would like to study. Indeed, as soon as $X$ is replaced
by a more regular signal $X^{\varepsilon}$ the regularization phenomenon is
lost and there is no hope that the approximate equations have unique
solutions. So while the existence problem is easier, the uniqueness gets out
of reach. In the limit where the regularization is removed one expects to
regain uniqueness but then it is the meaning of the equation which is not
clear.

In order to have a well--defined setting in which to discuss the existence,
uniqueness and regularization effect for transport equations we will follow
the recent work of Gubinelli, Tindel and Toricella
{\cite{gubinelli_controlled_2014}} where they use the theory of controlled
rough paths introduced by Gubinelli in {\cite{gubinelli_controlling_2004}} to
define controlled viscosity solutions of fully non--linear PDEs with driving
signals given by a general (step--2) rough path, thus filling the conceptual
gap left behind in the approach of Caruana, Friz and Oberhauser
{\cite{caruana_rough_2011}}. Their approach shows the versatility of the
controlled approach to deal with various problems in rough PDE theory.

In order to do so we need a geometric rough path $\mathbf{X}$ and we will
define a certain class of solutions which are controlled, in a weak sense, by
the rough path $\mathbf{X}$. \ Finally, we will show that this notion of
solution is an extension of the classical notion of solutions, which enjoys
uniqueness in a natural class of vectorfields $b$ and that the regularization
properties of $X$ can be used to extend the class of vectorfields for which
uniqueness holds, proving for the first time the regularization by noise
result for this class of rough PDEs.

Let us remark that Gubinelli and Jara {\cite{gubinelli_regularization_2012}}
have already used the controlled path approach in a more probabilistic setting
in order to prove regularization by noise results for the
Kardar--Parisi--Zhang equation. In a recent paper Diehl, Friz and Stannat \cite{diehl_stochastic_2014} come with quite similar techniques. Nevertheless they seems to be unable to apply it to regularization by noise.

Denoting with $u_{t} ( \varphi ) = \langle u(t, \cdot ), \varphi ( \cdot )
\rangle$ the pairing of $u$ with a smooth test function $\varphi$ (depending
only on the space variable) we can reformulate the PDE as the infinite set of
integral equations
\begin{equation}
  u_{t} ( \varphi ) =u_{0} ( \varphi ) + \int_{0}^{t} u_{s} ( \tmop{div}
  (b_{s} \varphi )) \dd s+ \int_{0}^{t} u_{s} ( \nabla \varphi ) \cdot
  \dd \mathbf{X}_{s} \label{eq:rough-transport}
\end{equation}
for all $t \geqslant 0$ and test functions $\varphi$ in a suitable class. The
last integral in the r.h.s. will be understood as a rough integral for the
\tmtextit{controlled path} $s \mapsto u_{s} ( \nabla \varphi )$ w.r.t. the
rough path $\mathbf{X}$.

In the controlled path theory, the idea is to ask the integrand to ``look
like'' the driving term, at least at a first order level. We will ask the
pairing of the solutions against test functions to have this property, and
this will lead to the following definition.

\begin{definition}[Definition \ref{definition_WCS} below]
  Let $X \in \mathcal{C}^{\gamma} ( [0,T] )$ and $b \in L^{\infty} ( [0,T];
  \tmop{Lin} ( \mathbbm{R}^{d} ) )$ with $\tmop{div}  b \in L^{\infty} ( [0,T]
  \times \mathbbm{R}^{d} )$. Here $\mathrm{Lin}$ denote the space of functions with linear growth, as in definition \ref{def:lin} below. A function $u \in L^{\infty} ([0,T] \times
  \mathbbm{R}^{d} )$ is a Weak controlled solution to equation
  {\eqref{eq:rough-transport}} with initial condition $u_{0} \in L^{\infty}
  (\mathbbm{R}^{d} )$ and driving term $X$ if for all $\varphi \in
  C^{\infty}_{c} (\mathbbm{R}^{d} )$ the pairing $u ( \varphi )$ is controlled
  by $X$ in the sense of Definition \ref{definition:controlled_path} and if
  furthermore Equation {\eqref{eq:rough-transport}} is fulfilled, where the
  term $\int_{0}^{t} u_{s} ( \nabla \varphi ) \cdot \dd X_{s}$ is
  understood as the controlled rough integral of $u ( \nabla \varphi )$
  against the rough path $\mathbf{X}$.
\end{definition}

As it will be shown, this definition is an extension of the classical notion
of weak solutions, in the case where $X$ is smooth. Besides we will show that
the same phenomenon of regularization by noise appears in the case of the
fractional Brownian motion, and we will mostly retrieve the results of
Flandoli, Gubinelli and Priola but for a larger class of noises. Moreover as
our theory is completely deterministic we will be able to handle random
vectorfield $b$. In fact, for the fractional Brownian motion define on $(\Omega,\cF, \PP)$, if we add to the last definition that the solutiond $u$ lie in $L^\infty(\Omega\times[0,T];\mathrm{Lin}(\RR^d))$, 
one have the following existence and uniqueness result. This is a melting pot of Theorems \ref{theorem:existence_SWCS} and \ref{theorem:unique} and Corollary \ref{theorem:strong_uniqueness}.

\begin{theorem}
 Let $(\Omega,\cF,\PP)$ a probability space, $H\in(1/3,1)$ and $B^H$ a $d$-dimensional fractional Brownian motion on that space. Let $\bB^H=(B^H,\BB^H)$ a lift of the fractional Brownian  motion into the geometric rough paths of order $\gamma\in(1/3,H)$ and $u_0\in L^\infty(\Omega\times\RR^d;\RR)$.
 	\begin{enumerate}
\item (Theorem \ref{theorem:existence_SWCS})
When $b\in L^\infty([0,T]\times \RR^d)$ and $\div b \in L^{\infty}([0,T]\times\RR^d)$ there exists a stochastic weak controlled solution.

\item (Corollary \ref{theorem:regu_uniqueness}) Let $\alpha+3/2>0$ and $\alpha>-1/2H$. When $b, \div b\in L^\infty(\Omega\times\RR^d)$ such that almost surely $k\to \hat b(k) (1+|k|)^\alpha \in L^1$ and $k\to \hat {\div b}(k) (1+|k|)^\alpha \in L^1$ there exists a unique stochastic weak controlled solution with initial condition $u_0$.

\item (Theorem \ref{theorem:unique}) When $u_0\in L^\infty (\RR^d)$ and $b, \div b \in \cC_b^{\alpha+1}(\RR^d)$ with $\alpha>\max(-1,-1/2H)$ there exists a unique stochastic weak controlled solution with initial condition $u_0$
\end{enumerate}
 \end{theorem}

Note that the point (2) in the last theorem allows us to consider random vector fields and random initial conditions, which is a huge improvement of the results of \cite{flandoli_well-posedness_2010}.

This paper is structured as follows. Section~\ref{section:preliminary} is devoted
to preliminary results: we recall some notation for the involved function
spaces and we give a short overview of the theory of controlled rough paths and then we recall some of the results of~{\cite{catellier_averaging_2012}} for
the regularization by non Brownian noise Section~\ref{section:existence}
is devoted to the definition of weak controlled solutions and to the proof of
their existence in general case. Finally in Section \ref{section:uniqueness},
thanks to a duality method we prove uniqueness for
equation~{\eqref{eq:rough-transport}}.
\subsubsection*{Acknowledgments}
The author benefited from deep discussion on that subject with his PhD advisor Massimiliano Gubinelli. Nicolas Perkowski proofread this work. The author is supported by the Center Henri Lebesgue. Most of this work have been done under the support of a PhD thesis at University Paris Dauphine.
\section{Preliminaries}\label{section:preliminary}
	\subsection{Notations and functional spaces}
  In all the following, the notation $D ^{n} f$ is for the $n$--th
  Differential of a function $f$. When $f:\mathbbm{R}^{d} \to\mathbbm{R}$,
  $\nabla f= ( \partial_{1} f, \ldots , \partial_{d} f )$ is the gradient, and
  $\nabla^{2} f= ( \partial_{i} \partial j f )_{1 \leqslant i,j \leqslant d}$
  is the Hessien of $f$. Furthermore when $f:\mathbbm{R}^{d} \rightarrow
  \mathbbm{R}^{d}$ we denote by $\tmop{Jac}  f= \det   ( D f )$ the Jacobien
  determinant of $f$.
  
    For a function $u$ of $[0,T]$, we define $\delta u_{s,t} =u_{t} -u_{s}$ the
  increment of $u$.
Whenever it make sense,  we denote with $u ( \varphi ) = \langle u ( \cdot ) , \varphi ( \cdot )
  \rangle$ the pairing of $u$ with a smooth test function $\varphi$.

Finally for $a,b \in \mathbbm{R}$ we write $a \lesssim b$ if there exists a
  constant $C>0$ independent of $a$ and $b$ such that $a \leqslant C b$. When
  $a \lesssim b$ and $b \lesssim a$ we write $a \sim b$. Furthermore, the
  notation $a \lesssim_{c} b$ specifies that the constant $C>0$ depends on
  $c$.

\begin{definition}
  Let $( E,d_{E} )$ a complete metric linear space and $( F, \| \|_{F} )$ a
  Banach space. For $n \in \mathbbm{N}$ and \ $0< \alpha \leqslant 1$ we
  define the space of $\alpha$-H{\"o}lder-continuous functions from $E$ to $F$
  by
  \[ \mathcal{C}^{n+ \alpha} =\mathcal{C}^{n+ \alpha} ( E,F ) = \left\{ f \in
     C^{n} ( E,F )  : \llbracket f \rrbracket_{n+ \alpha} =  \sup_{x \neq y}
     \frac{\| D^{n} f ( x ) -D^{n} f ( y ) \|_{F}}{d_{E} ( x,y )^{\alpha}} <+
     \infty \right\} . \]
  The quantity $\llbracket f \rrbracket_{n+ \alpha}$ is only a semi-norm for
  the space $\mathcal{C}^{n+ \alpha}$. For $x_{0} \in E$ the following
  quantity is a norm such that the space $\mathcal{C}^{\alpha}$ is complete
  \[ \| f \|_{x_{0} ,n+ \alpha ;F} = \llbracket f \rrbracket_{n+ \alpha} +
     \sum_{k=0}^{n} \| D^{k} f ( x_{0} ) \|_{F} . \]
  When $x_{0} =0$ we only write $\| f \|_{n+ \alpha}$.
To avoid confusion with the space of continuously
  differentiable functions, we will write $\cC^1(E;F) = \tmop{Lip} ( E;F )$. Furthermore
  When it is not specified, $F$ is always assumed to be the space
  $\mathbbm{R}^{d}$ and $| . |_{F} = | . |$ the usual Euclidean norm.
\end{definition}

\begin{definition}
  For $E$ and $F$ as before and $n< \alpha \leqslant n+1$, $f \in
  \mathcal{C}^{\alpha}_{b} ( E,F )$ if $f \in \mathcal{C}^{\alpha} (
  \mathbbm{R}^{d } )$ and if for all $k \in \{ 0, \ldots ,n \}$ $\| D^{k} f
  \|_{\infty} \assign \sup_{x \in E} | D^{k} f ( y ) | <+ \infty$. On this
  space we consider the norm
  \[ \| f \|_{\mathcal{C}^{\alpha}_{b}} = \llbracket f \rrbracket_{\alpha} +
     \sum_{k=0}^{n} \| D^{k} f \|_{\infty} . \]
When $E$ is bounded, the norms on $\cC^\alpha = \cC^\alpha_b$ and the norms are equivalents, we identify it in that case.
\end{definition}

\begin{remark}
  There is a way to extend the space $\mathcal{C}^{\alpha}_{b}$ to nonpositive
  value of $\alpha$. This is via the Besov spaces $B^{\alpha}_{\infty ,
  \infty}$, as it can be found in {\cite{bahouri_fourier_2011}} or
  {\cite{triebel_theory_2010}}. When $\alpha <0$, the results of part
  \ref{subsec:regularization} about the flow of the caracteristic equation are
  still true, see {\cite{catellier_averaging_2012}}. Nevertheless, the
  definition of the transport equation for such irregular vectorfields seems
  to be tricky in that case. The method of Chouk and Gubinelli
  {\cite{chouk_nonlinear_2013}} and {\cite{chouk_nonlinear_2014}} does not
  apply and another definition has to be found. We postpone the analysis of
  this situation to a further publication. 
\end{remark}

\begin{definition}
  Let $\nu \in ( 0,1 ]$, $(E,d_{E} )$ a complete metric space and $(F, \|
  \|_{F} )$ a Banach space. We define the space of $\nu$-H{\"o}lder-continuous
  functions from $E^{2}$ to $F$ by
  \[ \mathcal{C}^{\nu}_{2} (E^{2} ,F) = \left\{ f \in C(E^{2} ,F) :f(x,x)=0 
     \tmop{and} \| f \|_{\nu} =  \sup_{x \neq y} \frac{| f(x,y) |_{F}}{d_{E}
     (x,y)^{\nu}} <+ \infty \right\} . \]
\end{definition}

Unlike the case of the space $\mathcal{C}^{\alpha} ( E,F )$, $\| . \|_{\nu}$
is a norm on $\mathcal{C}^{\nu}_{2}$. Finally, we introduce some notations for
the usual $L^{p}$ spaces with image in Banach spaces.

\begin{definition}
  Let $p \geqslant 1$ and $F$ a Banach space and $T>0$. We define
  \[ L^{p} ( [0,T];F ) = \left\{ b : [ 0,T ] \rightarrow F : \| b \|_{p;F}
     \assign \left( \int_{0}^{T} \| b_{u} \|_{F}^{p} \dd u \right)^{1/p} <+
     \infty \right\} \]
  with the usual modification for $p=+ \infty$.
\end{definition}

In order to have existence of global weak solution for the transport equation
in the classical case, the vectorfield must have at most linear growth in the
space variable. In order to quantify that, let us define a space of function
with linear growth.

\begin{definition}\label{def:lin}
  Let $d \geqslant 1$, the space of functions with linear growth is defined as
  follows
  \[ \tmop{Lin} ( F ) = \left\{ b  \tmop{measurable}   \tmop{from} 
     \mathbbm{R}^{d}   \tmop{to}  \mathbbm{R}^{d} :  \| b \|_{\tmop{Lin}} =
     \left\| \frac{f ( . )}{1+ | . |} \right\|_{\infty} <+ \infty \right\} .
  \]
  Furthermore $( \tmop{Lin} ( F ) , \| \|_{\tmop{Lin}} )$ is a Banach space.
\end{definition}

The approach for the transport equation developed here mostly relies on the method of characteristics (see Appendix \ref{appendix:characteristics} for more details). In particular, we aim to consider vector field with linear growth in space, and we will need some a priori bound of the flows for associated to such vector fields thanks to the following differential equation
\begin{equation}\label{eq:differential}
\Phi_t(x) = x +\int_0^t b(r,\Phi_r(x)) \dd r + X_t
\end{equation}
 The method for the uniqueness developed in Section \ref{section:uniqueness} strongly relies on comparison between flows associated to the dynamics driven by different $X$. The two following lemmas gives the a priori bounds needed, the proof of those lemmas, and some additional material in the case of ODE with linear growth vector fields can be found in Appendix \ref{appendix:standard_flow}.
 
 \begin{lemma}
  \label{lemma:holder_norm_flow}Let $b \in L^{\infty} ( [0,T], \tmop{Lin} (
  \mathbbm{R}^{d} ) )$ such that the flow $\Phi$ of the equation \eqref{eq:differential}
   exists for all $t \in [ 0,T ]$. There exists a constant
  $K ( T, \| b \|_{\infty ; \tmop{Lin}} ) >0$ such that for all $x \in
  \mathbbm{R}^{d}$, $\Phi ( x ) \in \mathcal{C}^{\gamma} ( [0,T] )$ and
  \[ \| \Phi ( x ) \|_{\gamma} \leqslant K ( 1+ | x | ) ( 1+ \| X \|_{\gamma}
     ) . \]
\end{lemma}

\begin{remark}
  The constant $K$ can be chosen as $K ( T, \| b \|_{\infty ; \tmop{Lin}} )
  =K_{T}  g ( T \| b \|_{\infty ; \tmop{Lin}} )$ with $g ( x ) = ( ( x^{2} +x
  ) e^{x} +x+1 ) e^{x}$.
\end{remark}

\begin{lemma}
  \label{lemma:flow_convergence_X_smooth}Let $b \in L^{\infty} (
  [0,T];C^{1}_{b} ( \mathbbm{R}^{d } ) )$ and $X,Y \in \mathcal{C}^{\gamma}$.
  Then
  \[ \| \Phi^{X} ( x ) - \Phi^{Y} ( x ) \|_{\gamma , [ 0,T ]} \leqslant C (
     T, \| D b \|_{\infty} ) \| X-Y \|_{\gamma , [ 0,T ]} \]
  where $C$ is independent of $x$ and nondecreasing in $T$ and $\| D b
  \|_{\infty}$. Furthermore, when $b \in L^{\infty} ( [0,T];C^{2}_{b} (
  \mathbbm{R}^{d } ) )$, we also have
  \[ \| D \Phi^{X} ( x ) -D \Phi^{Y} ( x ) \|_{\tmop{Lip}} \leqslant C ( T, \|
     D b \|_{\infty} , \| D^{2}  b \|_{\infty} ) \| X-Y \|_{\gamma} . \]
\end{lemma}
	\subsection{Controlled rough path theory in a nutshell}

	Rough path theory is a way to describe the effects of irregular signals on
certain non-linear systems. It has been first developed by Lyons and his
coauthors, see for example
{\cite{lyons_differential_1998,lyons_differential_2007}} and the book by Friz
and Victoir~{\cite{friz_multidimensional_2010}}. In order to use this theory
to define integrals againts irregular signals we will use the notion of
{\tmem{controlled paths}} developed by Gubinelli in
{\cite{gubinelli_controlling_2004}}. An enjoyable exposition of this theory
can also be found in {\cite{friz_course_????}}. When the path is not of finite
variation, there is not enough informations to define an integral against its
(weak)-derivative. The theory of controlled rough paths overcomes this
problem and gives a general setting for the theory of integration against
irregular paths.
\subsubsection{Controlled integral against rough path}

One of the first quantities we would like to define is the integral of the
path against itself. The idea of rough path integration is to presuppose the
existence of a first order iterated integral, and to construct a theory of
integration related to that enhanced path (the path and its iterated
integral). This idea leads us to the following definition

\begin{definition}
  Let $1/3< \gamma \leqslant 1/2$. The pair $\mathbf{X}= ( X,\mathbbm{X}^{2}
  )$ \ is a rough path of order $\gamma$ if $X \in \mathcal{C}^{\gamma}
  ([0,T],\mathbbm{R}^{d} )$, $\mathbbm{X} \in \mathcal{C}^{2 \gamma}_{2}
  ([0,T]^{2} ;\mathcal{M}_{d} (\mathbbm{R}))$ and for $0 \leqslant s \leqslant
  u \leqslant t \leqslant T \text{}$
  \[ \mathbbm{X}_{s,t} -\mathbbm{X}_{s,u} -\mathbbm{X}_{u,t} = ( X_{u} -X_{s}
     ) \otimes ( X_{t} -X_{u} )_{} = ( ( X_{u}^{i} -X_{s}^{i} ) ( X_{t}^{j}
     -X_{u}^{j} ) )_{0 \leqslant i,j \leqslant d} . \]
  Furthermore, we define $\| \mathbf{X} \|_{\mathcal{R}^{\gamma}} = \| X
  \|_{\gamma} + \| \mathbbm{X} \|_{2 \gamma}$ and for two $\gamma$-rough paths
  $\mathbf{X}$ and $\mathbf{Y}$ we define
  \[ \| \mathbf{X}-\mathbf{Y} \|_{\mathcal{R}^{\gamma}}
     =d_{\mathcal{R}^{\gamma}} ( \mathbf{X},\mathbf{Y} ) = \| X-Y
     \|_{\gamma} + \| \mathbbm{X}-\mathbbm{Y} \|_{2 \gamma} . \]
\end{definition}

The term $\mathbbm{X}$ can be understood as the iterated integral of $X$
against itself. Formally we have
\[ \mathbbm{X}_{s,t} = \int_{s}^{t} ( X_{r} -X_{s} ) \otimes \dd X_{r} . \]
In fact, in the last equality, the left hand side is a definition for the
right hand side. On the other hand, when $X$ is a smooth path, for example $X
\in C^{1} ( [0,T] )$, we can always define a natural lift $\mathbf{X}$ to
$X$ by
\[ \mathbf{X}= ( X,\mathbbm{X} )   \tmop{where}  \mathbbm{X}_{s,t} =
   \int_{s}^{t} ( X_{r} -X_{s} ) \dot{X}_{r} \dd r. \]
In order to approximate irregular rough paths by smooth paths, we define the
space of geometric rough paths as the closure of smooth rough paths for the
rough path distance. This leads to the following definition

\begin{definition}
  Let $1/3< \gamma \leqslant 1/2$, a $\gamma$-rough path $\mathbf{X}$ is a
  geometric rough path, and we write $\mathbf{X} \in \mathcal{R}^{\gamma}$
  if there exists a sequence $X^{\varepsilon} \in C^{1} ([0,T];\mathbbm{R}^{d}
  )$ such that
  \[ \| \mathbf{X}-\mathbf{X}^{\varepsilon} \|_{\mathcal{R}^{\gamma}}
     \rightarrow_{\varepsilon \rightarrow 0} 0 \]
  where $\mathbf{X}^{\varepsilon}$ is the natural lift of $X^{\varepsilon}$
  as a $\gamma$-rough path.
\end{definition}

In all the following we will consider only geometric rough paths. A general
discussion about rough paths and geometric rough paths can be found in
{\cite{hairer_geometric_2012}}.

As in the stochastic calculus setting, where we can integrate progressively
measurable processes only, one has to give a structure to the paths we can
integrate. In fact, as the integral of $X$ against $\dd X$ is already
defined by the definition of the rough path $\mathbf{X}$, the idea is to
consider functions which locally up to the first order look like $X$. Such
functions are called controlled by $X$ and they are defined in the following
definition.

\begin{definition}
  \label{definition:controlled_path}Let $1/3< \gamma \leqslant 1/2$ and $X \in
  \mathcal{C}^{\gamma} ( [0,T] )$. A function $y \in \mathcal{C}^{\gamma} (
  [0,T] )$ is $\gamma$-controlled by $X$, and we write $y \in
  \mathcal{D}^{\gamma}_{X} ( [0,T] )$ if
  \[ y_{t} -y_{s} =y'_{s} ( X_{t} -X_{s} ) +y^{\#}_{s,t} \]
  with $y' \in \mathcal{C}^{\gamma}$ and $y^{\#} \in \mathcal{C}^{2
  \gamma}_{2}$. Furthermore, we define the controlled norm of $y$ by
  \[ \| y \|_{\mathcal{D}^{\gamma}_{X}} = \| y \|_{\gamma} + \| y' \|_{\gamma}
     + \| y^{\#} \|_{2 \gamma} . \]
  When there is no ambiguity, we will omit the $\gamma$ and say that $y$ is
  controlled by $X$.
\end{definition}

The space of controlled paths has a rich algebraic structure. In particular,
it is stable by products. Indeed the following estimate holds.

\begin{lemma}[Gubinelli {\cite{gubinelli_controlling_2004}}]
  \label{lemma:product_controlled_path}Let $a,b \in \mathcal{D}^{\gamma}_{X}$
  and $\tilde{a} , \tilde{b} \in \mathcal{D}^{\gamma}_{Y}$. Then $a b \in
  \mathcal{D}^{\gamma}_{X}$ and $\tilde{a}   \tilde{b} \in
  \mathcal{D}^{\gamma}_{Y}$ and
  \[ \|a b- \tilde{a}   \tilde{b} \|_{\mathcal{C}^{\gamma}} \leqslant \|a-
     \tilde{a} \|_{\gamma} (\|b\|_{\mathcal{D}^{\gamma}_{X}} + \| \tilde{b}
     \|_{\mathcal{D}^{\gamma}_{Y}} ) + \| b- \tilde{b} \|_{\gamma} (
     \|a\|_{\mathcal{D}^{\gamma}_{X}} + \| \tilde{a}
     \|_{\mathcal{D}^{\gamma}_{Y}} ) \]
  \begin{eqnarray*}
    \| ( a b )' - ( \tilde{a}   \tilde{b} )' \|_{\gamma} & \leqslant & ( \| a
    \|_{\mathcal{D}^{\gamma}_{X}} + \| \tilde{a} \|_{\mathcal{D}^{\gamma}_{Y}}
    + \| b \|_{\mathcal{D}^{\gamma}_{X}} +\| \tilde{b}
    \|_{\mathcal{D}^{\gamma}_{Y}} )\\
    &  & \times ( \| a- \tilde{a} \|_{\gamma} + \| a' - \tilde{a}'
    \|_{\gamma} +\|b- \tilde{b} \|_{\gamma} + \| b' - \tilde{b}' \|_{\gamma} )
  \end{eqnarray*}
  and
  \begin{eqnarray*}
    \| ( a b )^{\#} - ( \tilde{a}   \tilde{b} )^{\#} \|_{\mathcal{C}^{2
    \gamma}_2} & \leqslant & ( \| b \|_{\mathcal{D}^{\gamma}_{X}} + \| \tilde{b}
    \|_{\mathcal{D}^{\gamma}_{Y}} ) ( \| a- \tilde{a} \|_{\gamma} + \| a^{\#}
    - \tilde{a}^{\#} \|_{\gamma} )\\
    &  & + ( \| a \|_{\mathcal{D}^{\gamma}_{X}} + \| \tilde{a}
    \|_{\mathcal{D}^{\gamma}_{Y}} ) ( \| b- \tilde{b} \|_{\gamma} + \| b^{\#}
    - \tilde{b}^{\#} \|_{\gamma} )
  \end{eqnarray*}
\end{lemma}

Whenever a path is uniformly locally controlled by $X$, it is globally
controlled by $X$ as stated in the following lemma.

\begin{lemma}
  \label{remark:global_holder_norm}Suppose that $a: [0,T] \rightarrow
  \mathbbm{R}^{d}$ is uniformly locally controlled by $X$, that is there
  exists $a'$ and $a^{\#}$ such that for all $s,t \in [0,T]$
  \[ a_{t} -a_{s} =a'_{s} ( X_{t} -X_{s} ) +a^{\#}_{s,t} \]
  and there exists $\varepsilon >0$ such that for all $s,t \in [0,T]$ with $|
  s-t | \leqslant \varepsilon$ we have
  \[ \llbracket a \rrbracket_{\mathcal{D}^{\gamma}_{X, \varepsilon}} \assign
     \sup_{s \neq t, | t-s | < \varepsilon} | a_{t} -a_{s} | / | t-s
     |^{\gamma} + | a'_{t} -a'_{s} | / | t-s |^{\gamma} + | a^{\#}_{s,t} | / |
     t-s |^{2 \gamma} <+ \infty \]
  Then $a \in \mathcal{D}^{\gamma}_{X}$ and we have
  \[ \| a \|_{\mathcal{D}^{\gamma}_{X}} \lesssim ( T/ \varepsilon )^{1-
     \gamma} \llbracket a \rrbracket_{\mathcal{D}^{\gamma}_{X, \varepsilon}} (
     1+ \| X \|_{\gamma} ) + | a_{0} | + | a'_{0} | . \]
\end{lemma}

\begin{remark}
  This lemma also apply when $a$ is only locally H{\"o}lder continuous, we have
  \[ \llbracket a \rrbracket_{\gamma} \lesssim ( T/ \varepsilon ) \sup_{t \neq
     s, | t-s | \leqslant \varepsilon} | a_{t} -a_{s} | / | t-s |^{\gamma} .
  \]
\end{remark}

The definition of controlled paths and the definition of rough paths allow us
to construct a controlled rough integral as the limit of the Riemann sum. This
construction and the properties are a byproduct of the existence of the
{\tmem{sewing map}} (Proposition 1 in {\cite{gubinelli_controlling_2004}}).

\begin{theorem}[Controlled Rough Integral, Gubinelli
{\cite{gubinelli_controlling_2004}}]
  \label{th:controlled_integral}Let $1/3< \gamma \leqslant 1/2$, $\mathbf{X}
  \in \mathcal{R}^{\gamma} ([0,T])$ and let $a \in \mathcal{D}^{\gamma}_{X}
  ([0,T])$. For all $0 \leqslant s \leqslant t \leqslant T$, the following
  limit of Riemann sums exists
  \[ \int_{s}^{t} a_{r} \dd \mathbf{X}_{r} \assign
     \lim_{\tmscript{\begin{array}{c}
       \mathcal{P}  \tmop{partition}   \tmop{of}   [ s,t ]\\
       | \mathcal{P} | \rightarrow 0
     \end{array}}} \sum_{[ t_{i} ,t_{i+1} ] \in \mathcal{P}} ( a_{t_{i}}
     \delta X_{t_{i} ,t_{i+1}} +a'_{t_{i}} \mathbbm{X}_{t_{i} ,t_{i+1}} ) \]
  and does not depends on the partition. Furthermore,
  \[ \left| \int_{s}^{t} a_{r} \dd \mathbf{X}_{r} -a_{s} \delta X_{s,t}
     -a'_{s} \mathbbm{X}_{s,t} \right| \lesssim | t-s |^{3 \gamma} \| X
     \|_{\mathcal{R}^{\gamma}} \| a \|_{\mathcal{D}^{\gamma}_{X}} \]
  and the map from $\mathcal{D}^{\gamma}_{X}$ to $\mathcal{D}^{\gamma}_{X}$
  given by
  \[ a \rightarrow \int_{0}^{.} a_{r} \dd \mathbf{X}_{r} \]
  is linear and continuous and we have
  \[ \left\| \int_{0}^{.} a_{r} \dd \mathbf{X}_{r}
     \right\|_{\mathcal{D}^{\gamma}_{X}} \lesssim ( 1+ \| \mathbf{X}
     \|_{\mathcal{R}^{\gamma}} ) \| X \|_{D^{\gamma}_{X}} . \]
\end{theorem}

Let us also give the equivalent for a function $[0,T]^{2} \rightarrow F$ of
the classical result which says that when a continuous function $f$ from
$[0,T]$ to $F$ is such that there exists $\varepsilon >0$ with $| f_{t} -f_{s}
| \lesssim | t-s |^{1+ \varepsilon}$ then $f \equiv 0$.

\begin{proposition}[Gubinelli {\cite{gubinelli_controlling_2004}}]
  \label{th:sewing_map}Let $\mu >1$. Let $h \in \mathcal{C}^{\mu}_{2}$ such
  that for all $0 \leqslant s<u<t \leqslant T$
  \[ h_{s,t} -h_{u,t} -h_{s,u} =0 \]
  then $h \equiv 0$.
\end{proposition}
\subsubsection{Fractional Brownian motion as rough path}

Finally in order for this theory to be useful, we need to be able to lift to
the space of rough path the class of signals $X$ we will consider. The following
theorem gives a whole set of stochastic processes with such a property. It can
be found in {\cite{friz_course_????}} and {\cite{friz_multidimensional_2010}}.
Note that the first result for the lift of the fractional Brownian motion is
due to Coutin and Qian {\cite{coutin_stochastic_2002}}. Let us first remind the definition of a the fractional Brownian motion.

\begin{definition}\label{def:fbm}
 Let $(\Omega,\cF,\PP) $ a probability space, and $H\in(0,1)$. The fractional Brownian motion of Hurst parameter $H$ define is the only continuous centered Gaussian process $(B^H_t)_{t\in[0,T]}$ on $[0,T]$, starting at $0$ and with covariance
 \[ \EE [B^H_t B^H_s] = \frac12 \left(|t|^{2H}+|s|^{2H}- |t-s|^{2H}\right). \]
 A $d$-dimensional fractional Brownian motion of Hurst parameter $H$ is a stochastic process $B^H=(B^{H,1},\dots,B^{H,d})$ where the $B^{H,i}$ are i.i.d. standard fractional Brownian motion of Hurst parameter $H$.
 \end{definition}

The fractional Brownian motion is an extension of the Brownian motion, in the sense that when $H=\frac12$, the $B^H$ is a standard Brownian motion. Furthermore, the fractional Brownian motion have some nice properties, in particular it has stationary increment,  it is self-similar of order $H$, and for all $H>\eps>0$ it is almost surely $H-\eps$ Hölder continuous. Nevertheless, when $H\neq\frac12$ it is neither a semimartingale, nor a Markov process. Hopefully, one can use controlled rough path theory with the fractional Brownian motion.

\begin{theorem}
  Let $H \in ( 1/3,1/2 ]$ and $B^{H}$ a $d$-dimensional fractional Brownian motion of Hurst
  parameter $H$. For any $1/3< \gamma< H $, almost surely $B^{H}$ can be
  lifted as a $\gamma$-rough path $\mathbf{B}^{H} = (B^{H}
  ,\mathbbm{B}^{H} )$. Furthermore for every $+ \infty >p \geqslant 1$,
  \[ \mathbbm{E} [ \| \mathbf{B}^{H} \|^{p}_{\mathcal{R}^{\gamma} ( [0,T]
     )} ] <+ \infty \]
  and there exists a smooth measurable approximation $B^{H, \varepsilon}$ of
  $B^{H}$ such that $ \mathbf{B}^{H, \varepsilon} \rightarrow^{\mathcal{R}_\gamma}\mathbf{B}^{H}
  $ almost surely and for
  all $1 \leqslant p<+ \infty$
  \[ \mathbbm{E} [ \| \mathbf{B}^{H, \varepsilon} -\mathbf{B}^{H}
     \|_{\mathcal{R}^{\gamma} ( [0,T] )}^{p} ] \rightarrow 0, \]
  where $\mathbf{B}^{H, \varepsilon} = ( B^{H, \varepsilon} ,\mathbbm{B}^{H,
  \varepsilon} )$ with $\mathbbm{B}^{H, \varepsilon}_{s,t} = \int_{s}^{t} (
  B^{H, \varepsilon}_{r} -B^{H, \varepsilon}_{s} ) \dot{B}^{H,
  \varepsilon}_{r} \dd r$.
\end{theorem}

When $H=\frac12$ one can specify a bit this last theorem. Indeed, it that setting, $B=B^{1/2}$ is a Brownian motion, and both Ito and Stratonovitch calculus are available.  In particular, one can show \cite{friz_course_????, friz_multidimensional_2010} that for all $\gamma\in(1/3,1/2)$, the pair $\bB=(B,\BB^{strato})$ is almost surely a $\gamma$-geometric rough path, where $(\BB^{strato}_{s,t}) = (\BB^{strato,i,j}_{s,t})_{1\le i,j\le d}$ and
\[\BB^{strato,i,j}_{s,t} = \int_s^t (B^i_r-B^i_s) \circ \dd B_r^j\]
where the last integral is the Stratonavitch integral. This Stratonovitch Brownian rough path is exaclty the rough path of the last theorem. Furthermore, in that setting, the definition of the controlled rough integral and the definition of the Stratonovitch integral coincide almost surely. This result can be found in the article of Gubinelli \cite{gubinelli_controlling_2004} about rough integrals, and the formulation here comes from Friz and Hairer \cite{friz_course_????}.

\begin{proposition}\label{proposition:integral_strato_rough}
 Assume that almost surely $Y$ is $\gamma$-controlled by $B$. Then the rough integral of $Y$ against  $\bB=(B,\BB^{strato})$ exists. Moreover if $Y$ and $Y'$ are adapted, the quadratic covariation of $Y$ and $B$ exists and, almost surely,
 \[\int_0^T Y \dd \bB = \int_0^T Y\circ \dd B.\]
 \end{proposition}

	
	\subsection{Irregular paths and regularizations properties of fractional Brownian motion}\label{subsec:regularization}
	
	Classical solutions of the transport equation can be construct thank to the method of characteristics (see Appendix \ref{appendix:characteristics} ). The method strongly relies on the regularity of the flow of the characteristics ordinary differential equations associated to the transport equation. Namely when all the functions are regular enough, the unique solution of the Cauchy problem
	\[\partial_t u(t,x) + (b(t,x)+\dot{X}_t)\cdot \nabla u(t,x) =0,\quad u(0,x)=u_0(x)\]
	is $u(t,x)=u_0(\Phi^{-1}_t(x))$ where $\Phi$ is the flow of the equation
\[\Phi_t(x) = x + \int_0^t b(r,\Phi_r(x))\dd r + X_t.\]
In a recent article, Catellier and Gubinelli \cite{catellier_averaging_2012} have shown that, among others processes, the fractional Brownian motion have good regularization property for the last ordinary differential equation, and allows to have a regular flow for $b$ with really poor regularity. We give here the few needed results, and refer to \cite{catellier_averaging_2012} for more details and proofs.

\begin{definition}
  Let $X: [0,T] \rightarrow \mathbbm{R}^{d}$ and $\rho >0$. We say that the
  function $X$ is $\rho$-irregular if there exists $\gamma >1/2$ such that
  \[ \| \phi^{X} \|_{\mathcal{W}^{\rho , \gamma}_{T}} \assign \sup_{s \neq t
     \in [0,T]} \sup_{\xi \in \mathbbm{R}^{d}} ( 1+ | \xi | )^{\rho} | t-s
     |^{- \gamma} \left| \int_{s}^{t} e^{i \xi X_{q}} \dd q \right| <+
     \infty . \]
\end{definition}

This definition is not empty: almost every path of the fractional Brownian is
$\rho$-irregular. Furthermore, nondegenerate $\alpha$-stable L{\'e}vy
processes also have this property.

\begin{theorem}
  Let $B^{H}$ be a $d$-dimensional fractional Brownian motion of Hurst
  parameter $H \in ( 0,1 )$, then for all $\rho <1/2H$ there exists $\gamma>\frac12$ such that almost surely $B^{H}$
  is $\rho$-irregular. Furthermore there exists
  $\lambda >0$ such that
  \[ \mathbbm{E} [ \exp ( \lambda \| \phi^{B^{H}} \|_{\mathcal{W}^{\rho ,
     \gamma}_{T}}^{2} ) ] <+ \infty . \]
\end{theorem}

The regularization properties of $\rho$-irregular paths will occur in
Fourier--Lebesgue spaces. The oscillations of the function $\phi^{X}$ are
enough to regularize the differential system in that setting.

\begin{definition}
  The space of Fourier-Lebesgue function of order $\alpha >0$ is defined as
  \[ \mathcal{F}L^{\alpha} = \left\{ f \in \mathcal{S}' ( \mathbbm{R}^{d} ) |
     \| f \|_{\mathcal{F}L^{\alpha}} \assign \int_{\mathbbm{R}^{d}} \dd k |
     \hat{f} ( k ) | ( 1+ | k | )^{\alpha} <+ \infty \right\} \]
\end{definition}
and one of the main theorem of that article. 
\begin{remark}
  Note that for $\alpha >0$, $\mathcal{F}L^{\alpha} \subset
  \mathcal{C}^{\alpha}_{b}$ the space of bounded H{\"o}lder continuous
  functions of index $\alpha$.
\end{remark}

When the vector-field $b$ lies in a Fourier Lebesgue space and $X$ is
$\rho$-irregular, one has a regularization property for the perturbed
differential equation. When $b$ is regular enough, the flow has good spatial
regularity properties and furthermore there are good ways of approaching this flow by a
regularized one.

\begin{theorem}
  \label{theorem:rho_irregular_flow}Let $\rho >0$, let $X$ a $\rho$-irregular
  path. Let $\alpha >- \rho$ and $\alpha +3/2>0$. Let $b \in
  \mathcal{F}L^{\alpha +3/2}$. Then for all $x \in \mathbbm{R}^{d}$ there
  exists a unique solution $\Phi ( x )$ to the equation
  \[ \Phi_{t} ( x ) =x+ \int_{0}^{t} b ( \Phi_{q} ( x ) ) \dd q+X_{t} . \]
  Furthermore, $\theta_{t} ( x ) = \Phi_{t} ( x ) -X_{t}$ is Lipschitz
  continuous in time. The function $\theta$ is also locally Lipschitz in
  space, uniformly in time. Moreover, for any mollification $b^{\varepsilon}$
  of $b$ we write $\Phi^{\varepsilon}$ the flow of the approximate equation
  \[ \Phi^{\varepsilon}_{t} ( x ) =x+ \int_{0}^{t} b^{\varepsilon} (
     \Phi_{q}^{\varepsilon} ( x ) ) \dd q+X_{t} \]
  the following bound holds
  \[ \| \Phi^{\varepsilon} ( x ) - \Phi ( x ) \|_{\infty} \leqslant K ( | x |
     ) \| b-b^{\varepsilon} \|_{\mathcal{F}L^{\alpha +3/2}} \]
  where the constant $K$ is increasing in $| x |$ and independent of
  $\varepsilon$. The approximate flow $\Phi^{\varepsilon}$ is also
  differentiable in space and
  \[ \sup_{\varepsilon >0 } \sup_{t \in [0,T]} | D \Phi^{\varepsilon}_{t} ( x
     ) | \leqslant K ( | x | ) <+ \infty . \]
\end{theorem}

Even though the flow of perturbed differential equation does not have the same
regularization properties as $\rho$-irregular path, one can look at their
averaging properties. This is the purpose of the following Lemma, and it will
be really important in the following.

\begin{proposition}
  \label{prop:rho_irregular_flow}Let $f \in C^{1} ( \mathbbm{R}^{d} ) \cap
  \mathcal{F}L^{\alpha +3/2}$, $b$, $b^{\varepsilon}$ and $\Phi^{\varepsilon}$
  as in the previous theorem. For all $t \in [0,T]$ let us define the function
  \[ F_{t}^{\varepsilon} ( x ) = \int_{0}^{t} f ( \Phi^{\varepsilon}_{q} ( x
     ) ) \dd q. \]
  Then $F^{\varepsilon}_{t}$ is differentiable in space and there exists a
  constant $K ( | x | )$ at most with linear growth in $| x |$ and independent
  of $\varepsilon >0$ such that
  \[ \sup_{t \in [0,T]} | D F^{\varepsilon}_{t} ( x ) | \leqslant K ( | x | )
     \| f \|_{\mathcal{F}L^{\alpha +3/2}} . \]
\end{proposition}

In the case of the fractional Brownian motion, the previous results are not
optimal in term of the regularity of the vectorfields involved in the
equation. Indeed, using a Girsanov transform, one can show that the fractional
Brownian motion has regularization properties for ODEs when $b$ lies in a
bigger space of function/distribution. The price for that extension is the
loss of the global character of the averaging property for the fractional
Brownian motion, as the exceptional set here depends on the function.

\begin{theorem}
  \label{theorem:fBm_flow}Let $H \in ( 0,1 )$, $( \Omega
  ,\mathcal{F},\mathbbm{P} )$ a probability space and $B^{H}$ a
  $d$--dimensional fractional Brownian motion on $( \Omega
  ,\mathcal{F},\mathbbm{P} )$. Let $\alpha >-1/2H$ and $\alpha +1>0$. Let $b
  \in \mathcal{C}^{\alpha +1}_{b}$. Then for all $x \in \mathbbm{R}^{d}$,
  there exists $\mathcal{N}_{b,x} \subset \mathcal{F}$ such that $\mathbbm{P}
  ( \mathcal{N}_{b,x} ) =0$ such that for all $\omega \notin
  \mathcal{N}_{b,x}$ there exists a unique solution of the equation
  \[ \Phi_{t} ( \omega ,x ) =x+ \int_{0}^{t} b ( \Phi_{q} ( \nobracket
     \nobracket \omega ,x ) ) \dd q+B^{H}_{t} ( \omega ) . \]
  Furthermore for any mollification $b^{\varepsilon}$ of $b$ and for
  $\Phi^{\varepsilon}$ the flow of the approximate equation, for any $x \in
  \mathbbm{R}^{d }$
  \[ \mathbbm{E} [ \| \Phi^{\varepsilon} ( x ) - \Phi ( x ) \|_{\infty
     ,[0,T]}^{2} ] \rightarrow_{\varepsilon \rightarrow 0} 0 \]
  and there exists a constant $K ( | x | )$ at most with linear growth in $|
  x |$ such that
  \[ \sup_{\varepsilon >0 }  \mathbbm{E} [ \sup_{t \in [0,T]} | D
     \Phi^{\varepsilon}_{t} ( x ) |^{2} ] \leqslant K ( | x | )^{2} <+ \infty
     . \]
\end{theorem}

The same kind of regularization properties for the flow occur in the context.

\begin{proposition}
  \label{prop:fBm_flow}Let $f \in C^{1} ( \mathbbm{R}^{d} ) \cap
  \mathcal{C}_{b}^{\alpha +1}$ with $\alpha +1>0$ and $\alpha >-1/2H$, $b$,
  $b^{\varepsilon}$ and $\Phi^{\varepsilon}$ as in the previous Theorem.
  
  For all $t \in [0,T]$ let us define the function
  \[ F_{t}^{\varepsilon} ( x ) = \int_{0}^{t} f ( \Phi^{\varepsilon}_{q} ( x
     ) ) \dd q. \]
  Then $F^{\varepsilon}_{t}$ is differentiable in space and there exists a
  constant $K ( | x | )$ at most with linear growth in $| x |$ and independent
  of $\varepsilon >0$ such that
  \[ \mathbbm{E} [ ( \sup_{t \in [0,T]} | D F^{\varepsilon}_{t} ( x ) | )^{2}
     ] \leqslant K ( | x | )^{2} \| f \|_{\mathcal{C}^{\alpha}_{b}}^{2} . \]
\end{proposition}
	and one of the main theorem of that article. 
	
\section{Rough transport equation, existence}\label{section:existence}
In order to deal with the multiplicative perturbation of the classical
transport equation we need a new notion of solution. When $X$ is a Brownian
motion, one can use, as in Flandoli et al.
{\cite{flandoli_well-posedness_2010}} the Stratonovitch integral to deal with
the multiplicative term. As here, we intend to work with processes such as
fractional Brownian motion, which are neither a martingale nor a Markov
processes, there is no way to use classical stochastic calculus.

In order to replace the stochastic integral we will use controlled rough
paths, as presented in {\cite{gubinelli_controlling_2004}} but also in
{\cite{friz_course_????}}. The way we will require the solution to be weakly
controlled by the process $X$ is to be linked with the way Gubinelli et al. in {\cite{gubinelli_controlled_2014}} define controlled viscosity solution of non linear PDEs. In the following we will focus on rough paths $\mathbf{X} \in \mathcal{R}^{\gamma}$ with $1/3< \gamma \leqslant 1/2$, since in that case the controlled rough integrals are quite easy to define. When $1 \geqslant \gamma>1/2$, all the computations are easy as we can consider usual Young integrals.

We can now focus ourselves on the rough transport equation (RTE). Namely for
$b \in L ^{\infty} ( [0,T]; \tmop{Lin} ( \mathbbm{R}^{d} ) )$ and
$\mathbf{X} \in \mathcal{D}^{\gamma} ( [0,T] )$, we want to solve the
following Cauchy problem
\begin{equation}
  \left\{\begin{array}{l}
    \partial_{t} u_{t} ( x ) +b_{t} ( x ) . \nabla u_{t} ( x ) + \nabla u_{t}
    ( x ) . \dd \mathbf{X}_{t} =0\\
    u_{0} \in L^{\infty} ( \mathbbm{R}^{d} )
  \end{array}\right. \label{eq:RTE} .
\end{equation}
To deal with the term $\nabla u_{t} ( x ) . \dd \mathbf{X}_{t}$ which is
a priori ill-defined even in the weak sense, we need to introduce  new notions
of solution for equation {\eqref{eq:RTE}}. When $b$ is smooth enough, it is even possible to have strong solutions, when the last term of equation \eqref{eq:RTE} will be understood as a rough integral. Those strong solutions will be quite useful for proving the uniqueness of weak solutions thanks to a duality argument. This is why we define strong solutions for a more general equation, but with non-optimal hypothesis on the vector field $b$. It seems that $b$ could have linear growth (see \cite{diehl_stochastic_2014}).

\begin{definition}
  Let $b,c \in L^{\infty} ([0,T],L^{\infty} (\mathbbm{R}^{d} ))$, $\frac{1}{3}
  < \gamma \leqslant 1$ and $\mathbf{X}= ( X,\mathbbm{X} ) \in
  \mathcal{R}^{\gamma} ([0,T])$ a $\gamma$-rough path. A strong controlled
  solution with initial condition $u_{0}$ of the rough continuity equation
\begin{equation}\label{eq:rough_continuity_equation}
 \partial_{t} u_{t} ( x ) +b_{t}(x) . \nabla u_{t} ( x ) +c_{t} ( x ) u_{t} (
     x ) + \nabla u_{t} ( x ) . \dd \mathbf{X}_{t} =0 \nocomma ,
     \hspace{1em} u ( 0,x ) =u_{0} ( x ) , 
     \end{equation}
  is a function $u \in \mathcal{C}^{\gamma} ([0,T];C^{1} (\mathbbm{R}^{d} ))
  \cap L^{\infty} ([0,T] \times \mathbbm{R}^{d} )$ such that
  \begin{enumerate}
    \item For all $x \in \mathbbm{R}^{d}$, the function $t \rightarrow \nabla
    u_{t} ( x )$ is controlled by $X$.
    
    \item For all $x \in \mathbbm{R}^{d}$ and all $s \leqslant t \in [0,T]$,
    the following equation is satisfied
    \[ u_{t} ( x ) -u_{s} ( x ) + \int_{s}^{t} b_{q} ( x ) . \nabla u_{q} ( x
       ) +c_{q} ( x ) u_{q} ( x ) \dd q+ \int_{s}^{t} \nabla u_{q} ( x )
       \dd \mathbf{X}_{q} =0 \]
    where the last integral is the rough integral of $\nabla u_{.} ( x )$
    against $\dd \mathbf{X}$.
  \end{enumerate}
\end{definition}
When $u_{t} \in C^{2}_{b}$, and as $\bX$ is a geometric rough path,  we can
replace condition $(2)$ by the following one.
\begin{itemize}
  \item There exists a function $R^{u} ( x ) : [0,T]^{2} \rightarrow
  \mathbbm{R}$ such that $R^{u} ( x ) \in \mathcal{C}^{3 \gamma}_{2} ( [0,T]
  )$ and
  \[ u_{t} ( x ) -u_{s} ( x ) + \int_{s}^{t} ( b_{q} ( x ) . \nabla u_{q} ( x
     ) +c_{q} ( x ) u_{q} ( x ) ) \dd q+ \nabla u_{s} ( x ) .X_{s,t} +
     \frac{1}{2} \nabla^{2} u_{s} ( x ) X^{\otimes 2}_{s,t} +R_{s,t}^{u} ( x )
     =0. \]
\end{itemize}

As in the classical case, in order to have existence of strong solutions, the vector field $b$ has to be smooth. This is the meaning of the following theorem.

\begin{theorem}
  \label{theorem:strong_controlled_solutions}Let $b \in L^{\infty}
  ([0,T];C^{2}_{b} ( \mathbbm{R}^{d} ) )$, $c \in L^{\infty} ( [0,T];C^{2}_{b}
  ( \mathbbm{R}^{d } ) )$, $1/3< \gamma \leqslant 1/2$, $\mathbf{X} \in
  \mathcal{R}^{\gamma}$ and $\varphi_{0} \in C^{3}_{c} ( \mathbbm{R}^{d} )$.
  Let $\Phi$ the flow of the equation
  \[ \Phi_{t} ( x ) =x+ \int_{0}^{t} b_{q} ( \Phi_{q} ( x ) ) \dd q +X_{t} \]
  and $\Phi^{-1}$ its inverse. The function
  \[ ( t,x ) \rightarrow \psi_{t} ( x ) = \varphi_{0} ( \Phi^{-1}_{t} ( x ) )
     \exp \left( - \int_{0}^{t} c_{t-q} ( \Phi^{-1}_{q} ( x ) ) \dd q
     \right) \]
  is a strong controlled solution to equation
  {\eqref{eq:rough_continuity_equation}}. Furthermore the function $x
  \rightarrow \| \nabla \psi_{.} ( x ) \|_{\mathcal{D}^{\gamma}_{X}}$ is
  compactly supported.
\end{theorem}

When $b$ is non-smooth, one has to give a weak sense for the solutions, and use the fact that, as a distribution, the function $u$ is weakly controlled by the process $X$.

\begin{definition}
  \label{definition_WCS}Let $u_{0} \in L^{\infty} ( \mathbbm{R}^{d} )$, $1/3<
  \gamma \leqslant 1/2$, $\mathbf{X}= ( X,\mathbbm{X} ) \in
  \mathcal{R}^{\gamma}$, $b\in L^1_{loc}([0,T];L^1_{loc}(\RR^d))$ with $\mathrm{div}\, b \in L^1_{loc}([0,T];L^\infty(\RR^d))$ .We say that $u \in L^{\infty} ( [0,T] \times
  \mathbbm{R}^{d} )$ is a weak controlled solution (WCS) of equation
  {\eqref{eq:RTE}} with initial condition $u_{0}$ if for all $\varphi \in
  C^{\infty}_{c} ( \mathbbm{R}^{d} )$,
  \begin{enumerate}
    \item we have $u ( \varphi ) \in \mathcal{D}^{\gamma}_{X} ( [0,T] )$;
    
    \item for all $t \in [0,T]$,
    \[ u_{t} ( \varphi ) =u_{0} ( \varphi ) + \int_{0}^{t} u_{s} ( b_{s} .
       \nabla \varphi + \tmop{div} ( b_{s} ) \varphi ) \dd s+ \int_{0}^{t}
       u_{s} ( \nabla \varphi ) \dd \mathbf{X}_{s} \]
    where the quantity $\int_{0}^{t} u_{s} ( \nabla \varphi ) \dd
    \mathbf{X}_{s}$ is understood as the controlled rough integral of $u (
    \nabla \varphi )$ against $\dd \mathbf{X}$.
  \end{enumerate}
\end{definition}

\begin{remark}
  Since this definition is true for all $\varphi \in C^{\infty}_{c} (
  \mathbbm{R}^{d} )$, when $u$ is a weak Controlled solution to equation
  {\eqref{eq:RTE}} and $\varphi \in C^{\infty}_{c} ( \mathbbm{R}^{d} )$, for
  all $\alpha \in \{ 1, \ldots ,d \}^{d}$ with $| \alpha | =1$, we have
  \[ \delta u_{r,t} ( \partial^{\alpha} \varphi ) = \int_{r}^{t} u_{s} ( b_{s}
     . \nabla ( \partial^{\alpha} \varphi ) + \tmop{div} ( b_{s} ) (
     \partial^{\alpha} \varphi ) ) \dd s+ \int_{r}^{t} u_{s} ( \nabla (
     \partial^{\alpha} \varphi ) ) \dd \mathbf{X}_{s} \]
  and the derivative of $u ( \nabla \varphi )$ as a controlled path is $u (
  \nabla^{2} \varphi )$. Hence as $\mathbf{X}$ is geometric, an alternative
  formulation for Definition $\ref{definition_WCS}$ can be:
  \begin{enumerate}
    \item For all $\varphi \in C^{\infty}_{c} ( \mathbbm{R}^{d } )$, $u_{.} (
    \nabla \varphi ) \in \mathcal{D}^{\gamma}_{X} ( [0,T] )$, with $u_{.} (
    \varphi )' =u_{.} ( \nabla \varphi )$.
    
    \item There exists a remainder $R^{\varphi} : [0,T]^{2} \rightarrow
    \mathbbm{R}$ such that $| R_{s,t} ( \varphi ) | \lesssim | t-s |^{3
    \gamma}$ and for all $0 \leqslant s \leqslant t \leqslant T$ the following
    equation is verified
    \[ \delta u_{s,t} ( \varphi ) = \int_{s}^{t} u_{q} ( b_{q} . \nabla
       \varphi + \tmop{div} ( b_{q} ) \varphi ) \dd q+u_{s} ( \nabla
       \varphi ) X_{s,t} + \frac{1}{2} u_{s} ( \nabla^{2} \varphi ) X^{\otimes
       2}_{s,t} +R_{s,t} ( \varphi ) . \]
  \end{enumerate}
\end{remark}

Thanks to the definition of the rough integral (see Theorem
\ref{th:controlled_integral} above) these two formulations are equivalent.
Hence, in the following we will either use one or the other.

However from this second formulation it is clear that the solution $u$
depends only on the first level $X$ of the rough path $\mathbf{X}$ due to
the fact that only the symmetric part of the area $\mathbbm{X}_{s,t}$ is
needed to compute the rough integral on the r.h.s. if $u ( \nabla \varphi )'
=u ( \nabla^{2} \varphi )$ and that this symmetic part is canonical for
geometric rough paths and given by $X^{\otimes 2}_{s,t} /2$. It seems that this remark could be extended for general path $X$ with arbitrary regularity. Nevertheless, we restrict ourselves to $\gamma\in(1/3,1/2]$.

Whereas this notion of solution is quite different from the classical one,
when $X$ is smooth, the rough integral $\int_{0}^{t} u_{s} ( \nabla \varphi )
\dd \mathbf{X}_{s}$ is equal to the classical integral $\int_{0}^{t}
u_{s} ( \nabla \varphi ) \dot{X}_{s} \dd s$. In that case, we have to show
that the two notions of weak solutions coincide. In fact, we have the following theorem.

\begin{theorem}
  \label{theorem:smooth_transport}Let $b \in L^{\infty} ( [0,T];  C^{1}_b ( \mathbbm{R}^{d} ) )$ with $\tmop{div}  b \in
  L^{\infty} ( [0,T];C^{1}_{b} ( \mathbbm{R}^{d} ) )$, $u_{0} \in L^{\infty} (
  \mathbbm{R}^{d} )$ and $X \in C^{1} ( [0,T] )$. Let $1/3< \gamma \leqslant
  1/2$ and $\mathbf{X}$ be the natural lift of $X$ into the space
  $\mathcal{R}^{\gamma}$. Then there exists a unique weak controlled solution
  $u$ to equation {\eqref{eq:RTE}} with initial condition $u_{0}$. For almost
  every $t \in [0,T]$ and almost every $x \in \mathbbm{R}^{d}$ we have
  \[ u_{t} ( x ) =u_{0} ( \Phi_{t}^{-1} ( x ) ) \]
  where $\Phi$ is the flow of equation {\eqref{eq:differential}}.
\end{theorem}

Here, as $X$ is smooth, equation {\eqref{eq:RTE}} can be understood in the
weak sense as the classical transport equation
\begin{equation}
  \left\{\begin{array}{l}
    \partial_{t} u_{t} ( x ) +b_{t} ( x ) . \nabla u_{t} ( x ) + \nabla u_{t}
    ( x ) . \dot{X}_{t} =0\\
    u_{0} \in L^{\infty} ( \mathbbm{R}^{d} )
  \end{array}\right. . \label{eq:STE}
\end{equation}

Thanks to the hypothesis, and thanks to the usual method of characteristics, the
function $( t,x ) \rightarrow u_{0} ( \Phi^{-1}_{t} ( x ) )$ is the unique
weak solution of equation {\eqref{eq:STE}}. In order to prove that classical solutions are also controlled solutions, we will use some Taylor expansion of the flow $\Phi$, and that will be the aim of Subsection \ref{subsection:taylor}.

When the flow $\Phi$ of equation \eqref{eq:differential} does not exists, we will still have controlled solutions for equation \eqref{eq:RTE}. This is the purpose of the following result.

\begin{theorem}\label{theorem:existence_WCS}
  Let $\frac{1}{3} < \gamma \leqslant 1/2$, $b \in L^{\infty} ( [0,T];
  \tmop{Lin} ( \mathbbm{R}^{d} ) )$ with $\tmop{div}  b \in L^{\infty} (
  [0,T];L^{\infty} ( \mathbbm{R}^{d } ) )$ and let $\mathbf{X} \in
  \mathcal{R}^{\gamma} ( [0,T] )$. Then there exists a weak controlled
  solution $u \in L^{\infty} ( [0,T] \times \mathbbm{R}^{d} )$ to equation
  \eqref{eq:RTE}.
\end{theorem}

Finally, we aim to use those notions of solutions when the driving term $X$ and the solution $u$ are stochastic processes. As in the classical setting, the proof of Theorem  \ref{theorem:existence_WCS} relies on compactness arguments and extraction arguments. These arguments do not guaranty that the limit is measurable in $\omega$. The following definition is a refinement of Definition \ref{definition_WCS} in the case we are working with a rough path lift of a stochastic porcess. The subsequent theorem proves that such solutions exists in that case. 

\begin{definition}\label{def:stochastic_controlled_solutions}
  Let $1/3< \gamma \leqslant 1/2$, $( \Omega ,\mathcal{F},\mathbbm{P} )$ a
  probability space and $X$ a continuous stochastic process on $( \Omega
  ,\mathcal{F},\mathbbm{P} )$ such that almost surely $X$ lifts to a Rough
  Path $\mathbf{X} \in \mathcal{R}^{\gamma} ( [0,T] )$ in a measurable way.
  Let $b \in L^{\infty} ( \Omega \times [0,T]; \tmop{Lin} ( \mathbbm{R}^{d} )
  )$, $\div b \in \L^\infty([0,T]\times \RR^d)$ and $u_{0} \in L^{\infty} ( \Omega \times \mathbbm{R}^{d} )$.
  
  In that setting, a (Stochastic) weak controlled solution of the Rough
  transport equation with initial condition $u_{0}$ is a function $u \in
  L^{\infty} ( \Omega \times [0,T] \times \mathbbm{R}^{d} )$ such that almost
  surely, $u ( \varphi ) \in \mathcal{D}_{X}^{\gamma} ( [0,T] )$ for all
  $\varphi \in C^{\infty}_{c} ( \mathbbm{R}^{d} )$ and almost surely, for
  almost all $t \in [0,T]$,
  \[ u_{t} ( \varphi ) =u_{0} ( \varphi ) + \int_{0}^{t} u_{q} ( \tmop{div} (
     b_{q} \varphi ) ) \dd q+ \int_{0}^{t} u_{q} ( \nabla \varphi ) \dd
     \mathbf{X}_{q} , \]
  where that last integral is understood as the controlled rough integral of
  $u ( \nabla \varphi )$ against $\dd \mathbf{X}$.
\end{definition}

\begin{theorem}
  \label{theorem:existence_SWCS}Let $b,u_{0}$ and $X$ as in the last
  definition. We also assume that for a smooth measurable approximation
  $\mathbf{X}^{\varepsilon}$ of $\mathbf{X}$ and all $1/3< \gamma
  \leqslant 1/2$ and all $+ \infty >p \geqslant 1$, $\mathbbm{E} [ \|
  \mathbf{X}-\mathbf{X}^{\varepsilon} \|_{\mathcal{R}^{\gamma}}^{p} ]
  \rightarrow 0$.
  
  Then there exists a Stochastic weak controlled solution for the Rough
  transport equation with initial condition $u_{0}$.
\end{theorem}
 This notion of measurable solutions has to be compared with the usual notions of solutions for stochastic transport equation, when $X$ is a Brownian motion and the integral is understood in the Stratonovitch sense. Let us remind that in \cite{flandoli_well-posedness_2010}, the authors already show some regularizations effects for the stochastic transport equation driven by a (Stratonovitch) multiplicative Brownian motion. Let us remind the definition for stochastic solutions of the tranport equation in that article.

\begin{definition}[Definition 12 in \cite{flandoli_well-posedness_2010}]\label{def:linfty_solutions}
	Let $(\Omega,\cF,\PP)$ a probability space and $B$ a $d$-dimensional standard Brownian motion define on that space. Let  $b\in L^1_{loc}([0,T]\times \RR^d; \RR^d)$, $\div b \in L^1_{loc}([0,T]\times \RR^d)$ and $u_0\in L^\infty(\RR^d)$. A weal $L^\infty$-solutions of the transport equation
	\[\partial_t u(t,x) \dd t + b(t,x).\nabla u (t,x) \dd t + \nabla u(t,x) \circ \dd B_t =0; \quad u(0,x) = u((0,x)\]
	is a stochastic process $u\in L^\infty(\Omega\times[0,T]\times \RR^d)$ such that for every function test $\theta\in C^\infty_c(\RR^d,\RR)$, the process $t\to u_t(\theta) := \int_{\RR^d} u(t,x) \theta(x) \dd x$ has a continuous modification which is an $\cF$-semimartingale and for almost every $(w,t)\in\Omega\times \left[0,T\right]$,
	\[u_t(\theta) = u_0(\theta) + \int_0^t u_r(b(r,.).\nabla\theta+\mathrm{div}\, b(r,.) \theta ) \dd r + \int_0^t u_r (\nabla \theta) \circ \dd B_r.\]
	\end{definition}
 
One can wonder, if those two notions of solutions are the same, in the case of the Stratonovitch Brownian rough path. Proposition \ref{proposition:integral_strato_rough}, together with Definition \ref{def:stochastic_controlled_solutions} and Theorem \ref{theorem:existence_SWCS} allows us to say that whenever a weak controlled solution exists, in the setting of Brownian motion, then it is also a $L^\infty$ solution in the sense of Definition \ref{def:linfty_solutions}. One have to wonder if the contrary is true.
 
 \begin{proposition}
 Let $u_0$ and $b$ as in Theorem \ref{theorem:existence_SWCS}. Let $u$ a $L^\infty$ solution in the sense of Definition \ref{def:linfty_solutions}. There exists a modification such that $u$ is also a weak controlled solution. 
 \end{proposition}
 
 \begin{remark}
 The meaning of such a proposition, is that when one work in the setting of this work, which is slightly less general than the one of \cite{flandoli_well-posedness_2010} in the case of the Brownian motion, one have the same notions of solutions, and though the same results.  
 \end{remark}
 
 \begin{proof}
  Thanks to proposition \ref{proposition:integral_strato_rough} it is enough to prove that if $u$ is an $L^\infty$ solution, it is weakly controlled by $B$. Let $\theta\in C^\infty_c(\RR^d)$. Then, using the Ito formulation of the equation, we have
\begin{align*}
  u_t(\theta) - u_s(\theta) & = \int_s^t u_r( \div b(r,.) \theta + b(r,.). \nabla \theta)\dd r + \int_s^t u_r (\nabla \theta) \circ \dd B_r \\
  & = \int_s^t u_r( \div b(r,.) \theta + b(r,.). \nabla \theta)\dd r + \int_s^t u_r (\Delta \theta) \dd r +  \int_s^t u_r (\nabla \theta) \dd B_r \\
\end{align*}
Thanks to the hypothesis on $b$ (linear growth) and as $u$ is bounded and $\theta$ has compact support, for all $p\ge 1$, we have
\[\mathbb{E}\left[\left|u_t(\theta) - u_s(\theta)\right|^{2p}\right] \lesssim |t-s|^{2p} + \mathbb{E}\left(\int\left|u_r(\nabla \theta)\right|^2\dd r\right)^p \lesssim |t-s|^p.\]
Thanks to the Kolmogorov criterium, there exists a modification such that for all $\theta\in C^\infty_c$, $t\to u_t(\theta)$ is almost surely $\gamma$-Hölder continuous. 
As it is true for all $\theta$ this is also true for $\nabla\theta$. Furthermore, 
\begin{align*}
\mathbb{E}\left[\left(\int_s^t u_r(\nabla)-u_s(\nabla) \dd B_r\right)^{2p}\right]
&\lesssim \mathbb{E}\left(\int_s^t \left|u_r(\nabla \theta)-u_s(\nabla \theta)\right|^2\dd r\right)^p\\
& \lesssim |t-s|^{p(2\gamma+1)}.
\end{align*}
Again thanks to a Kolmogorov criterium, there exists a modification such that for all $\theta$, $t\to\left(\int_0^t u_r(\nabla)-u_s(\nabla) \dd B_r\right)$ is $2\gamma$-Hölder continuous, which ends the proof.
\end{proof}

\subsection{Strong controlled solutions}\label{subsection:strong}

In order to prove the existence of strong controlled solution (SCS) for the
rough continuity equation (RCE) we will proceed as for the weak solutions.
Namely we will approximate $X$ and show that classical solutions (see
Appendix~\ref{appendix:characteristics}) are controlled solutions for the
approximate equation, and then remove the approximation. Nevertheless, in
order to have an explicit form, we cannot use a compactness argument anymore,
and we will have to control all the objects in order to make them converge. As
we focus on regular $b$ only, standard argument will be used. The proof relies on a Taylor expansion of the potential solution given by the method of characteristics (Lemmas \ref{lemma:bounded_bounded_drift}, \ref{lemma:joker} and \ref{lemma:composition_flow_control}). As we intend to have a explicit representation of the strong controlled solutions, we can not use a compactness argument, but must show that a solution of the approximate equation (when $W^\eps$ is smooth) converges to a solution of the equation. For that, it is necessary to control the Taylor expansion of flow for different driving terms. It is the purpose of Lemmas \ref{lemma:comparison_controlled} and \ref{lemma:convergence_jacobien_smooth}.

The three following lemmas are quite similar to
Lemma \ref{lemma:bound_control}, Corollary \ref{corollary:holder_norm} and
Lemma \ref{lemma:smooth_control}, but since we are working with bounded
function $b$ while proving results for the strong solutions, we can get ride of the locality. As the proof are quite similar to the proof of the aforementioned results, we do not develop them here.

\begin{lemma}
  \label{lemma:bounded_bounded_drift}Let $b \in L^{\infty} ( [0,T];L^{\infty}
  ( \mathbbm{R}^{d} ) )$ and $X \in \mathcal{C}^{\gamma} ( \mathbbm{R}^{d} )$
  such that the flow $\Phi$ of the equation
  \[ \Phi_{t} ( x ) =x+ \int_{0}^{t} b_{q} ( \Phi_{q} ( x ) ) \dd q+X_{t} .
  \]
  Let $\varphi \in C_{c} ( \mathbbm{R}^{d} )$ a continuous function with
  compact support. Then for all $s,t \in [0,T]$ and all $r \in [ 0,1 ]$
  \[ x \rightarrow \varphi ( r ( \Phi_{t} ( x ) - \Phi_{s} ( x ) ) + \Phi_{s}
     ( x ) )   \tmop{has}   \tmop{compact}   \tmop{support} . \]
  Furthermore if $\tmop{supp}   \varphi \subset B ( 0,R_{\varphi} )$,then
  $\tmop{supp}   \varphi ( r ( \Phi_{t} ( . ) - \Phi_{s} ( . ) ) + \Phi_{s} (
  . ) ) \subset B ( 0,R )$ where $R=R_{\varphi} +2T^{\gamma} ( \| b
  \|_{\infty} + \| X \|_{\gamma} )$.
\end{lemma}

\begin{lemma}\label{lemma:joker}
  Let $b$ and $X$ as in Lemma \ref{lemma:bounded_bounded_drift}. Let $\varphi
  \in C^{1}_{c} (\mathbbm{R}^{d} ,\mathbbm{R}^{m} )$ a continuous
  differentiable function with compact support. Then $t \rightarrow \varphi (
  \Phi_{t} ( x ) ) \in \mathcal{C}^{\gamma} ( [0,T] )$, and furthermore
  \[ \| \varphi ( \Phi_{.} ( x ) ) \|_{\gamma} \lesssim_{T} ( \| b \|_{\infty}
     + \| X \|_{\gamma} ) \| D \varphi \|_{\infty} \mathbbm{1}_{B_{f} ( 0,R )}
     ( x ) \]
  where $R=R_{\varphi} +2 (  T \| b \|_{\infty} +T^{\gamma} \| X \|_{\gamma}
  )$ and $R_{\varphi}$ is such that $\tmop{supp}   \varphi \subset B_{f} (
  0,R_{\varphi} )$.
\end{lemma}

Those two previous lemmas guarantee that the function $t \rightarrow \varphi (
\Phi_{t} ( x ) )$ is controlled by $X$ and furthermore give a estimate on the
controlled norm. Indeed, the following lemma holds.

\begin{lemma}
  \label{lemma:composition_flow_control}Let $b$ and $X$ as in Lemma
  \ref{lemma:bounded_bounded_drift}. Let $\varphi \in C^{2}_{c} (
  \mathbbm{R}^{d} ,\mathbbm{R}^{m} )$, then $t \rightarrow \varphi ( \Phi_{t}
  ( x ) ) \in \mathcal{D}^{\gamma}_{X} ( \mathbbm{R}^{d} )$, and furthermore
  \[ \varphi ( \Phi_{.} ( x ) )' =D \varphi ( \Phi_{.} ( x ) ) \]
  and for $R$ as in the previous lemma, we have
  \[ \| \varphi ( \Phi_{.} ( x ) ) \|_{\mathcal{D}^{\gamma}_{X} ( [0,T] )}
     \leqslant C ( 1+ \| X \|_{\gamma} ) \mathbbm{1}_{B ( 0,R )} ( x ) . \]
\end{lemma}

In order to prove that there exist strong controlled solutions to equation
{\eqref{eq:rough_continuity_equation}}, we will approximate the rough path
$X$. Hence, we need some regularity for the controlled norm of the potential
solutions of the equations. The following lemma gives us the regularity w.r.t.
the rough path norm of $\mathbf{X}$ and $\mathbf{Y}$ of the controlled
test functions.

\begin{lemma}
  \label{lemma:comparison_controlled}Let $b \in L^{\infty} ( [0,T];C^{1}_{b} (
  \mathbbm{R}^{d } ) )$, $X \nocomma ,Y \in \mathcal{C}^{\gamma}$ and
  $\Phi^{X} , \Phi^{Y}$ the associated flows. Let $\varphi \in C^{3}_{c} (
  \mathbbm{R}^{d} )$. Then
  \[ \| \varphi ( \Phi_{.}^{X} ( x ) )' - \varphi ( \Phi_{.}^{Y} ( x ) )'
     \|_{\gamma} \leqslant C \| X-Y \|_{\gamma} ( 1+  \| Y \|_{\gamma} ) ( 1+ 
     \| X \|_{\gamma} ) \mathbbm{1}_{B_{f} ( 0,R )} ( x ) \]
  and
  \[ \| \varphi ( \Phi_{.}^{X} ( x ) )^{\#} - \varphi ( \Phi_{.}^{Y} ( x )
     )^{\#} \|_{2 \gamma} \leqslant C \| X-Y \|_{\gamma} ( 1+  \| Y
     \|_{\gamma} ) ( 1+  \| X \|_{\gamma} ) \mathbbm{1}_{B_{f} ( 0,R )} ( x )
     , \]
  where the two constants 
  $C=C_T (\| b \|_{\infty} , \| D b \|_{\infty} , \| \varphi
  \|_{\infty} , \| D  \varphi \|_{\infty} , \| D ^{2} \varphi \|_{\infty} , \|
  D^{3} \varphi \|_{\infty} )$ and $R=R_T (\| b \|_{\infty} , \| X \|_{\gamma} , \| Y \|_{\gamma} )$ are nondecreasing in all the parameters.
\end{lemma}

\begin{proof}
  We already know that the two functions $\varphi ( \Phi_{.}^{X} ( x ) )$ and
  $\varphi ( \Phi_{.}^{Y} ( x ) )$ are controlled by $X$ and by $Y$
  respectively and that
  \[ \varphi ( \Phi_{.}^{X} ( x ) )' =D \varphi ( \Phi_{.}^{X} ( x ) ) \]
  and
  
\begin{eqnarray*}
    \varphi ( \Phi_{.}^{X} ( x ) )^{\#} & = & \int_{0}^{1} \dd r D \varphi
    ( r \delta ( \Phi_{s,t}^{X} ) ( x ) + \Phi_{s}^{X} ( x ) ) . \int_{s}^{t}
    b_{q} ( \Phi_{q}^{X} ( x ) ) \dd q\\
    &  & + \int_{0}^{1} \int_{0}^{1} \dd r \dd q D^{2} \varphi ( r q
    \delta ( \Phi_{s,t}^{X} ) ( x ) + \Phi_{s}^{X} ( x ) ) . \delta (
    \Phi_{s,t}^{X} ) ( x ) .X_{s,t}
  \end{eqnarray*}
  and the same holds for $\varphi ( \Phi_{.}^{Y} ( x ) )$ when we replace $X$
  by $Y$. Hence

\begin{eqnarray*}
\lefteqn{\delta ( D \varphi ( \Phi_{.}^{X} ( x ) ) -D \varphi ( \Phi_{.}^{X} ( x
     ) ) )_{s,t}}\\
   &=& \int_{0}^{1} \dd r D^{2} \varphi ( r  \delta ( \Phi^{X}_{s,t} ) ( x
     ) + \Phi^{X}_{s} ( x ) ) . \delta ( \Phi^{X}_{s,t} ) ( x ) \\
   &&- \int_{0}^{1} \dd r D^{2} \varphi ( r  \delta ( \Phi^{Y}_{s,t} ) ( x
     ) + \Phi^{Y}_{s} ( x ) ) . \delta ( \Phi^{Y}_{s,t} ) ( x ) \\
   &=& \int_{0}^{1} \dd r D^{2} \varphi ( r \delta ( \Phi^{X}_{s,t} ) ( x )
     + \Phi^{X}_{s} ( x ) ) . \delta ( \Phi^{X} - \Phi^{Y} )_{s,t} ( x ) \\
  && + \int_{0}^{1}   ( D^{2} \varphi ( r \delta ( \Phi^{X}_{s,t} ) ( x ) +
     \Phi^{X}_{s} ( x ) ) - D^{2} \varphi ( r \delta ( \Phi^{Y}_{s,t} ) ( x )
     + \Phi^{Y}_{s} ( x ) ) ) . \delta ( \Phi^{Y}_{s,t} ) ( x ) \\
   &=& \int_{0}^{1} \dd r D^{2} \varphi ( r \delta ( \Phi^{X}_{s,t} ) ( x )
     + \Phi^{X}_{s} ( x ) ) . \delta ( \Phi^{X} - \Phi^{Y} )_{s,t} ( x ) \\
  && + \int_{0}^{1} \int_{0}^{1} \dd r  \dd q D^{3} \varphi ( q ( r 
     \delta ( \Phi^{X} - \Phi^{Y} )_{s,t} + \Phi^{X}_{s} ( x ) - \Phi^{Y}_{s}
     ( x ) ) +r \delta \Phi_{s,t}^{Y} ( x ) + \Phi^{Y}_{s} ( x ) )\\&& \qquad \qquad \qquad \qquad . ( r 
     \delta ( \Phi^{X} - \Phi^{Y} )_{s,t} + \Phi^{X}_{s} ( x ) - \Phi^{Y}_{s}
     ( x ) ) . \delta ( \Phi^{Y}_{s,t} ) ( x ) \\
   &=&A_{1} +A_{2} 
  \end{eqnarray*}
  
  Following the proof of Lemma~\ref{lemma:bounded_bounded_drift} and with the
  estimations of Lemma~$\ref{lemma:flow_convergence_X_smooth}$, there exists
  $R>0$ nondecreasing in all the parameters such that
  \[ | A_{1} | \leqslant C ( T, \| D^{3} \varphi \|_{\infty} , \| D^{2}
     \varphi \|_{\infty} \nocomma , \| D b \|_{\infty} , \| b \|_{\infty} ) \|
     X-Y \|_{\gamma} \mathbbm{1}_{B_{f} ( 0,R )} ( x ) | t-s | , \]
  where $C$ is nondecreasing with respect to the parameters. The same kind of
  bound holds for $A_{2}$ and we have
  \[ | A_{2} | \leqslant C | t-s |^{2 \gamma} \| X-Y \|_{\gamma}
     \mathbbm{1}_{B_{f} ( 0,R )} ( x ) \| X-Y \|_{\gamma} ( 1+  \| Y
     \|_{\gamma} ) . \]
  Let us turn now to the remainder. We decompose it into five terms:
  \[ \varphi ( \Phi_{.}^{X} ( x ) )_{s,t}^{\#} - \varphi ( \Phi_{.}^{Y} ( x )
     )_{s,t}^{\#} =B_{1} +B_{2} +B_{3} +B_{4} +B_{5} , \]
  where
  \[ B_{1} = \int_{0}^{1} \dd r D \varphi ( r \delta ( \Phi_{s,t}^{X} ) ( x
     ) + \Phi_{s}^{X} ( x ) ) . \int_{s}^{t} b_{q} ( \Phi_{q}^{X} ( x ) )
     -b_{q} ( \Phi_{q}^{Y} ( x ) ) \dd q, \]
  \[ B_{2} = \int_{0}^{1} \dd \tmop{r} D\varphi ( r \delta ( \Phi_{s,t}^{X}
     ) ( x ) + \Phi_{s}^{X} ( x ) ) -D \varphi ( r \delta ( \Phi_{s,t}^{Y} ) (
     x ) + \Phi_{s}^{Y} ( x ) ) . \int_{s}^{t} b_{q} ( \Phi_{q}^{Y} ( x ) )
     \dd q, \]
  \[ B_{3} = \int_{0}^{1} \int_{0}^{1} \dd r \dd q D^{2} \varphi ( r q
     \delta ( \Phi_{s,t}^{X} ) ( x ) + \Phi_{s}^{X} ( x ) ) . \delta (
     \Phi_{s,t}^{X} ) ( x ) . ( X-Y )_{s,t} , \]
  \[ B_{4} = \int_{0}^{1} \int_{0}^{1} \dd r \dd q D^{2} \varphi ( r q
     \delta ( \Phi_{s,t}^{X} ) ( x ) + \Phi_{s}^{X} ( x ) ) . \delta (
     \Phi^{X} - \Phi^{Y} )_{s,t} ( x ) .Y_{s,t} , \]
  \[ B_{5} = \int_{0}^{1} \int_{0}^{1} \dd r \dd q D^{2} \varphi ( r q
     \delta ( \Phi_{s,t}^{X} ) ( x ) + \Phi_{s}^{X} ( x ) ) -D^{2} \varphi ( r
     q \delta ( \Phi_{s,t}^{Y} ) ( x ) + \Phi_{s}^{Y} ( x ) ) . \delta (
     \Phi_{s,t}^{Y} ) ( x ) .Y_{s,t} . \]
  The analysis of those five terms being exactly the same as the analysis of
  the terms $A_{1}$ and $A_{2}$, the result follows easily.
\end{proof}
\begin{lemma}
  \label{lemma:convergence_jacobien_smooth}Let $b \in C^{1}_{b} (
  \mathbbm{R}^{d} )$, $X,Y \in \mathcal{C}^{\gamma} ( [0,T] )$ and $\Phi^{X}$
  and $\Phi^{Y}$ the associated flows. Let $c \in C^{1}_{b} ( \mathbbm{R}^{d }
  )$ and
  \[ K^{X}_{t} ( x ) = \exp \left( \int_{0}^{t} c_{q} ( \Phi^{X}_{q} ( x ) )
     \dd q \right) ,K^{Y}_{t} ( x ) = \exp \left( \int_{0}^{t} c_{q} (
     \Phi^{Y}_{q} ( x ) ) \dd q \right) . \]
  Then
  \[ \| K^{X} ( x ) -K^{Y} ( x ) \|_{\tmop{Lip}} \leqslant C ( T, \| D b
     \|_{\infty} , \| c \| , \| \nabla c \|_{\infty} ) \| X-Y \|_{\gamma} . \]
  Furthermore, if $b \in C^{2}_{b} ( \mathbbm{R}^{d} )$ and $c \in C^{2}_{b} (
  \mathbbm{R}^{d} )$, we have
  \[ \| \nabla K^{X} ( x ) - \nabla K^{Y} ( x ) \|_{\tmop{Lip}} \leqslant C (
     T, \| D b \|_{\infty} , \| D^{2} b \|_{\infty} , \| c \| , \| \nabla c
     \|_{\infty} , \| D^{2} c \|_{\infty} ) \| X-Y \|_{\gamma} . \]
  Finally the two constants are nondecreasing with respect to all the
  parameters. 
\end{lemma}

\begin{proof}
  As $K^{X}_{0} ( x ) =K^{X}_{0} ( y ) =1 \nocomma$, we only have to control
  the increments of the difference. Let $x \in \mathbbm{R}^{d}$, $s,t \in
  [0,T]$. Thanks to Lemma \ref{lemma:flow_convergence_X_smooth}, we have
  \begin{eqnarray*}
    | \delta ( K^{X} -K^{Y} )_{s,t} ( x ) | & \leqslant & | K^{X}_{s} ( x )
    -K^{Y}_{s} ( x ) | \left| \exp \left( \int_{s}^{t} c_{q} ( \Phi^{X}_{q} (
    x ) ) \dd q \right) -1 \right|\\
    &  & + \left| \exp \left( \int_{s}^{t} c_{q} ( \Phi^{X}_{q} ( x ) )
    \dd q \right) - \exp \left( \int_{s}^{t} c_{q} ( \Phi^{Y}_{q} ( x ) )
    \dd q \right) \right| K^{X}_{s} ( x )\\
    & \leqslant & e^{T \| c \|_{\infty}} \left| \int_{0}^{s} c_{q} (
    \Phi^{X}_{q} ( x ) ) \dd q- \int_{0}^{s} c_{q} ( \Phi^{Y}_{q} ( x ) )
    \dd q \right| | t-s | \| c \|_{\infty}\\
    &  & +e^{T \| c \|_{\infty}} \left| \int_{s}^{t} c_{q} ( \Phi^{X}_{q} ( x
    ) ) \dd q- \int_{s}^{t} c_{q} ( \Phi^{Y}_{q} ( x ) ) \dd q \right| T
    \| c \|_{\infty}\\
    & \lesssim & C ( \| D b \|_{\infty} , \| c \|_{\infty} , \| \nabla c
    \|_{\infty} ,T ) | t-s | \| X-Y \|_{\gamma} .
  \end{eqnarray*}
  Furthermore, as
  \[ \nabla K^{X}_{t} ( x ) = \int_{s}^{t} \nabla c_{q} ( \Phi^{X}_{q} ( x )
     ) .D \Phi^{X}_{q} ( x ) \dd q K_{t}^{X} ( x ) , \]
  we only have to prove the bound for $L_{t}^{X} ( x ) = \int_{0}^{t} \nabla
  c_{q} ( \Phi^{X}_{q} ( x ) ) .D \Phi^{X}_{q} ( x ) \dd q$. We have, again
  thanks to Lemma \ref{lemma:flow_convergence_X_smooth},
  \begin{eqnarray*}
    | \delta ( L^{X} -L^{Y} )_{s,t} ( x ) | & \leqslant & \left| \int_{s}^{t}
    \nabla c_{q} ( \Phi^{X}_{q} ( x ) ) . ( D \Phi^{X}_{q} ( x ) -D
    \Phi^{Y}_{q} ( x ) ) \dd q \right|\\
    &  & + \left| \int_{s}^{t} ( \nabla c_{q} ( \Phi^{X}_{q} ( x ) ) - \nabla
    c_{q} ( \Phi^{Y}_{q} ( x ) ) ) .D \Phi^{Y}_{q} ( x ) \dd q \right|\\
    & \lesssim & | t-s | \| \nabla c \|_{\infty} \| X-Y \|_{\gamma} + | t-s |
    \| D^{2} c \|_{\infty} \| X-Y \|_{\gamma} ,
  \end{eqnarray*}
  where the constant depends on $T$, $\| D b \|_{\infty}$, $\| D ^{2} b
  \|_{\infty}$, $\| c \|_{\infty}$ and $\| \nabla c \|_{\infty}$, which proves
  the result.
\end{proof}

We have gathered all the tools to prove the existence of strong controlled
solutions if the initial condition and the the two functions $b$ and $c$ are
regular enough.

\begin{proof}[Proof of Theorem \ref{theorem:strong_controlled_solutions}]
  Let us take a smooth approximation $\mathbf{X}^{\eta} = ( X^{\eta}
  ,\mathbbm{X}^{\eta} )$ of $\mathbbm{X}$ such that $\mathbf{X}^{\eta}
  \rightarrow \mathbf{X}$. By the method of characteristics (see
  Appendix~\ref{appendix:characteristics}) we already know that \[( t,x )
  \mapsto \psi^{\eta}_{t} ( x ) = \varphi_{0} ( ( \Phi^{\eta} )^{-1}_{t} ( x )
  ) \exp \left( - \int_{0}^{t} c_{t-q} ( ( \Phi^{\eta} )^{-1}_{q} ( x ) )
  \dd q \right)\] satisfies the following equation
  \begin{equation}
    \psi^{\eta}_{t} ( x ) - \psi^{\eta}_{s} ( x ) + \int_{s}^{t} b_{q} ( x ) .
    \nabla \psi^{\eta}_{q} ( x ) \dd q+ \int_{s}^{t} c_{q} ( x ) \psi_{q} (
    x ) \dd q+ \int_{s}^{t} \nabla \psi^{\eta}_{q} ( x ) .
    \dot{X}_{q}^{\eta} \dd q=0. \label{eq:strong_smooth_equation}
  \end{equation}
  Let us define $\varphi_{t}^{\eta} ( x ) = \varphi_{0} ( ( \Phi^{\eta}
  )^{-1}_{t} ( x ) )$ and $K^{\eta}_{t} ( x ) = \exp \left( - \int_{0}^{t}
  c_{t-q} ( ( \Phi^{\eta} )^{-1}_{q} ( x ) ) \dd q \right)$. Then
  \begin{eqnarray*}
    \nabla \psi^{\eta}_{t} ( x ) & = & \nabla \varphi_{t}^{\eta} ( x )
    K_{t}^{\eta} ( x ) + \nabla K_{t}^{\eta} ( x ) \varphi_{t}^{\eta} ( x )\\
    & = & \nabla \varphi_{0} ( ( \Phi^{\eta} )^{-1}_{t} ( x ) ) .D (
    \Phi^{\eta} )^{-1}_{t} ( x ) K_{t}^{\eta} ( x ) - \psi^{\eta}_{t} ( x )
    \int_{0}^{t} \nabla c_{t-q} ( ( \Phi^{\eta} )^{-1}_{q} ( x ) ) . (
    \Phi^{\eta} )^{-1}_{q} ( x ) \dd q,
  \end{eqnarray*}
  and since $( \Phi^{\eta} )^{-1}$ satisfies equation
  \[ ( \Phi^{\eta} )^{-1}_{t} ( x ) =x- \int_{0}^{t} b_{t-q} ( ( \Phi^{\eta}
     )^{-1}_{q} ( x ) ) \dd q-X_{t} , \]
  the function $t \mapsto \nabla \varphi_{0} ( ( \Phi^{\eta} )^{-1}_{t} ( x )
  )$ is controlled by $X^{\eta}$, thanks to Lemma
  \ref{lemma:composition_flow_control}. Furthermore,
  \[ t \mapsto D ( \Phi^{\eta} )^{-1}_{t} ( x ) \in \tmop{Lip} ( [0,T] ) \]
  and this function is controlled by $X^{\eta}$, with $( D ( \Phi^{\eta}
  )^{-1}_{.} ( x ) )' =0$. As $c \in C^{1}_{b} ( x )$, $t \rightarrow
  K_{t}^{\eta} ( x ) \in \tmop{Lip} ( \mathbbm{R}^{d } ) \cap L^{\infty} (
  \mathbbm{R}^{d } )$ and $t \mapsto \int_{0}^{t} \nabla c_{t-q} ( (
  \Phi^{\eta} )^{-1}_{q} ( x ) ) .D ( \Phi^{\eta} )^{-1}_{q} ( x ) \dd q
  \in \tmop{Lip} ( \mathbbm{R}^{d } ) \cap L^{\infty} ( \mathbbm{R}^{d } )$.
  The same arguments hold for $X$, and $t \mapsto \psi_{t} ( x )$. Hence,
  $\psi^{\eta}_{.} ( x )$ and $\psi_{.} ( x )$ are controlled respectively by
  $X^{\eta}$ and $X$. Furthermore, thanks to Lemmas
  \ref{lemma:comparison_controlled} and \ref{lemma:flow_convergence_X_smooth}:
  \begin{eqnarray*}
    \lefteqn{\| \nabla \varphi_{.}^{\eta} ( x ) - \nabla \varphi_{.} ( x ) \|_{\gamma}
    }\\
    & \leqslant & \| \nabla \varphi_{0} ( ( \Phi^{\eta} )^{-1}_{.} ( x ) ) -
    \nabla \varphi_{0} ( \Phi^{-1}_{.} ( x ) ) \|_{\gamma} ( \| D (
    \Phi^{\eta} )^{-1}_{t} ( x ) \|_{\gamma} + \| D \Phi^{-1}_{t} ( x )
    \|_{\gamma} )\\
    &  & + ( \| \nabla \varphi_{0} ( ( \Phi^{\eta} )^{-1}_{.} ( x ) )
    \|_{\mathcal{C}^{\gamma}} + \| \nabla \varphi_{0} \nobracket ( \Phi
    \nobracket^{-1}_{.} ( x ) ) \|_{\mathcal{C}^{\gamma}} ) \| D ( \Phi^{\eta}
    )^{-1}_{t} ( x ) - \nobracket \nobracket D \Phi^{-1}_{t} ( x )
    \|_{\gamma}\\
    & \lesssim & ( 1+ \| X \|_{\gamma} + \| X^{\eta  } \| )^{2} \| X-X^{\eta}
    \|_{\gamma} \mathbbm{1}_{B ( 0,R^{\eta} )} ( x ) ,
  \end{eqnarray*}
  Where $R^{\eta}$ is nondecreasing w.r.t. $\| X \|_{\gamma}$ and $\| X^{\eta}
  \|_{\gamma}$. As $( D ( \Phi^{\eta} )^{-1}_{.} ( x ) )' =0$ and $( D
  \Phi^{-1}_{.} ( x ) )' =0$, we also have, thanks to
  Lemma~\ref{lemma:product_controlled_path}
  \begin{eqnarray*}
    \| \nabla \varphi_{.}^{\eta} ( x )' - \nabla \varphi_{.} ( x )'
    \|_{\gamma} & \lesssim & \mathbbm{1}_{B_{f} ( 0,R^{\eta} )} ( x ) \|
    X-X^{\gamma} \| ( 1+ \| X \|_{\gamma} + \| X^{\eta} \|_{\gamma} )^{2}
  \end{eqnarray*}
  and
  \[ \| \varphi_{.}^{\eta} ( x )^{\#} - \varphi_{.} ( x )^{\#} \|_{\gamma}
     \lesssim \mathbbm{1}_{B_{f} ( 0,R^{\eta} )} ( x ) \| X-X^{\gamma} \| ( 1+
     \| X \|_{\gamma} + \| X^{\eta} \|_{\gamma} )^{2} , \]
  where the radius $R^{\eta}$ is nondecreasing with respect to $\| X \|$ and
  $\| X^{\eta} \|$. Furthermore, since $K^{\eta} ( x ) ,K ( x ) \in \tmop{Lip}
  ( [0,T] )$ and $\nabla K^{\eta} ( x ) , \nabla K ( x ) \in \tmop{Lip} (
  [0,T] )$, and thanks to Lemmas~\ref{lemma:convergence_jacobien_smooth}
  and~$\ref{lemma:product_controlled_path}$, and since $\| X^{\eta}
  \|_{\gamma} \lesssim \| X \|_{\gamma}$, there exists $R>0$ depending on $\|
  X \|_{\gamma} \nocomma , \| D b \|_{\infty}$ and $T$ and a constant $C$
  depending on $b,c, \varphi_{0} \nocomma ,T  \tmop{and}   \| X \|_{\gamma}$
  such that
  \[ \| \psi_{.}^{\eta} ( x ) - \psi_{.} ( x ) \|_{\gamma} + \|
     \psi_{.}^{\eta} ( x )' - \psi_{.} ( x )' \|_{\gamma} + \| \psi_{.}^{\eta}
     ( x )^{\#} - \psi_{.} ( x )^{\#} \|_{2 \gamma} \leqslant C \| X-X^{\eta}
     \|_{\gamma} \mathbbm{1}_{B_{f} ( 0,R )} ( x ) . \]
  Furthermore, $X \rightarrow^{\mathcal{C}^{\gamma}} X^{\eta}$, hence, by the
  definition of the rough integral and the comparison between controlled path
  (see Theorem \ref{th:controlled_integral} and Lemma
  \ref{lemma:product_controlled_path}), for all $s,t \in [0,T]$,
  \[ \int_{s}^{t} \nabla \psi^{\eta}_{q} ( x ) . \dot{X}_{q}^{\eta} \dd q
     \rightarrow_{\eta \rightarrow 0} \nabla \psi_{s} ( x ) .X_{s,t} +
     \frac{1}{2} \nabla^{2} \psi_{s} ( x ) .X^{\otimes 2}_{s,t}
     +R^{\psi}_{s,t} ( x ) \]
  and
  \[ |R^{\psi}_{s,t} ( x ) | \lesssim |t-s|^{3 \gamma} \mathbbm{1}_{B_{f} (
     0,R )} ( x ) . \]

  It remains to show that the other terms of Equation
  {\eqref{eq:strong_smooth_equation}} converge to the right quantities. But,
  thanks to Lemma~\ref{lemma:flow_convergence_X_smooth}, we know that $(
  \Phi^{\eta} )^{-1}_{t} ( x ) \rightarrow \Phi^{-1}_{t} ( x )$ and $D (
  \Phi^{\eta} )_{t}^{-1} ( x ) \rightarrow D \Phi_{t}^{-1} ( x )$, hence, as
  $\varphi_{0}$, $\nabla \varphi_{0}$, $c$ and $\nabla c$ are continuous,
  \[ \psi^{\eta}_{t} ( x ) \rightarrow \psi_{t} ( x )   \tmop{and}   \nabla
     \psi^{\eta}_{t} ( x ) \rightarrow \nabla \psi_{t} ( x ) . \]
  Furthermore
  \[ | b_{q} ( x ) . \nabla \psi^{\eta}_{t} ( x ) +c_{q} ( x ) \psi^{\eta}_{q}
     ( x ) | \lesssim ( 1+ \| X^{\eta} \| ) \mathbbm{1}_{B_{f} ( 0,R )} ( x )
     , \]
  hence
  \[ \int_{s}^{t} b_{q} ( x ) . \nabla \psi^{\eta}_{t} ( x ) +c_{q} ( x )
     \psi^{\eta}_{q} ( x ) \dd q \rightarrow \int_{s}^{t} b_{q} ( x ) .
     \nabla \psi_{t} ( x ) +c_{q} ( x ) \psi_{q} ( x ) \dd q \]
  and
  \[ \left| \int_{s}^{t} b_{q} ( x ) . \nabla \varphi_{q} ( x ) +c_{q} ( x )
     \psi_{q} ( x ) \dd q \right| \lesssim | t-s | \mathbbm{1}_{B_{f} ( 0,R
     )} ( x ) . \]
  Finally all the quantities converge and we have
  \[ \psi_{t} ( x ) - \psi_{s} ( x ) + \int_{s}^{t} b_{q} ( x ) . \nabla
     \psi_{q} ( x ) +c_{q} ( x ) \psi_{q} ( x ) \dd q+ \nabla \psi_{s} ( x
     ) .X_{s,t} + \frac{1}{2} \nabla \psi_{s} ( x ) .X^{\otimes 2}_{s,t}
     +R^{\psi}_{s,t} ( x ) =0 \]
  with
  \[ \psi_{t} ( x ) = \varphi_{0} ( \Phi^{-1}_{t} ( x ) ) \exp \left( -
     \int_{0}^{t} c_{t-q} ( \Phi^{-1}_{q} ( x ) ) \dd q \right) , \]
  which ends the proof.
\end{proof}

If $c=0$ in the last theorem, we have the existence of a strong controlled
solution for the rough transport equation. When $c= \tmop{div}  b$, it is a
solution for the rough continuity equation. This result gives us the good
dynamic to solve the rough transport equation by a duality argument. Indeed,
as stated before, we will be able to test weak controlled solution against
good test functions, \tmtextit{i.e.} the solution of the Rough Continuity
Equation with an approximate vector.

	\subsection{Taylor expansion of potential solutions}\label{subsection:taylor}.
	As explain before, the whole method of the existence results relies on a Taylor expansion of the candidate for the solution, $u_0(\Phi^{-1}_t(x)$, as soon as the flow exists. As we consider weak solution, we need some integrability conditions on the Taylor expansion of the potential solution. Here, since we will rely on a compactness result, we will not need comparison lemmas.

\begin{lemma}
  \label{lemma:bound_control}Let $b \in L^{\infty} ([0,T], \tmop{Lin}
  (\mathbbm{R}^{d} ))$ and $X \in \mathcal{C}^{\gamma} (\mathbbm{R}^{d} )$
  such that the flow $\Phi$ and its inverse $\Phi^{-1}$ exist.
   There exists
  $\varepsilon ( T, \| b \|_{\infty ; \tmop{Lin}} , \| X \|_{\gamma} ) >0$,
  such that for all $0 \leqslant s<t \leqslant T$ with $| t-s | <
  \varepsilon$, all $r \in [ 0,1 ]$, all test functions $\psi \in
  C^{\infty}_{c} ( \mathbbm{R}^{d} ,\mathbbm{R}^{m} )$ and all $N \in
  \mathbbm{N}$, there exists a constant $C_{N, \psi} >0$ such that
  \[ \psi ( r \Phi_{t} ( x ) + ( 1-r ) \Phi_{s} ( x ) ) ( 1+ | x | )^{N}
     \leqslant  C_{N, \psi} . \]
\end{lemma}

\begin{proof}
We only give the proof of the first point, as the proof of the second point is an straightforward adaptation of the one for the first point.
  Let $\varepsilon >0$ to be specified later and $s,t \in [0,T]$ with $| t-s |
  < \varepsilon$. Thanks to Proposition \ref{proposition:ODE_inverse_flow},
  where we set $y= \Phi_{s} ( x )$
  \begin{eqnarray*}
    | \psi ( r \Phi_{t} ( x ) + ( 1-r ) \Phi_{s} ( x ) ) | ( 1+ | x | )^{N} &
    = & | \psi ( r \Phi_{t} ( \Phi^{-1}_{s} ( y ) ) + ( 1-r ) y ) | ( 1+ |
    \Phi^{-1}_{s} ( y ) | )^{N}\\
    & \lesssim & ( 1+ | y | )^{N} | \psi ( r ( \Phi_{t} ( \Phi^{-1}_{s} ( y )
    ) -y ) +y ) | .
  \end{eqnarray*}
  Furthermore, applying Lemma \ref{lemma:holder_norm_flow},
  \begin{eqnarray*}
    | \Phi_{t} ( \Phi^{-1}_{s} ( y ) ) -y | & = & | \Phi_{t-s} ( y ) -y |\\
    & \leqslant & | t-s |^{\gamma} \| \Phi ( y ) \|_{\gamma}\\
    & \leqslant & K \varepsilon^{\gamma} ( 1+ | y | ) ( 1+ \| X \|_{\gamma} )
    .
  \end{eqnarray*}
  If $| y | \leqslant 1$, then for $\varepsilon \leqslant ( 2K ( 1+ \| X
  \|_{\gamma} ) )^{-1/ \gamma}$,
  \[ | \Phi_{t} ( \Phi^{-1}_{s} ( y ) ) -y | \leqslant 1 \]
  and $r \Phi_{t} ( \Phi^{-1}_{s} ( y ) ) + ( 1-r ) y \in B ( 0,2 )$, so that
  \[ \mathbbm{1}_{B ( 0,1 )} ( | y | ) | \psi ( r \Phi_{t} ( \Phi^{-1}_{s} ( y
     ) ) + ( 1-r ) y ) | ( 1+ | y | )^{N} \lesssim \sup_{B ( 0,2 )} | \psi ( z
     ) | . \]
  When $| y | >1$ and $\varepsilon \leqslant 2^{-2/ \gamma} ( K ( 1+ \| X
  \|_{\gamma} ) )^{-1/ \gamma}$, we have
  \[ | \Phi_{t} ( \Phi^{-1}_{s} ( y ) ) -y | \leqslant | y | /2 \]
  and $| r \Phi_{t} ( \Phi^{-1}_{s} ( y ) ) + ( 1-r ) y | \geqslant | y | /2$.
  Now let us take $\tilde{C}_{\psi} >0$ such that $| \psi ( z ) | \leqslant
  \tilde{C}_{\psi} / ( 1+2 | z |^{} )^{N}$. We have
  \begin{eqnarray*}
    \mathbbm{1}_{^{c} B ( 0,1 )} ( | y | ) | \psi ( r ( \Phi_{t} (
    \Phi^{-1}_{s} ( y ) ) -y ) +y ) | & \lesssim & ( 1+2 | r \Phi_{t} (
    \Phi^{-1}_{s} ( y ) ) + ( 1-r ) y | )^{-N}\\
    & \lesssim & ( 1+ | y | )^{-N} .
  \end{eqnarray*}
  Since $\Phi^{-1}$ satisfies the same type of equation as $\Phi$, thanks to
  Lemma \ref{lemma:holder_norm_flow}, for $\varepsilon$ as before,
  \[ ( 1+ | x | ) \leqslant 1+ | \Phi^{-1}_{s} ( y ) -y | + | y | \leqslant 2
     ( 1+ | y | ) \]
  and finally
  \[ | \psi ( r ( \Phi_{t} ( x ) - \Phi_{s} ( x ) ) + \Phi_{s} ( x ) ) |
     \lesssim ( 1+ | x | )^{-N} \]
  which ends the proof.
\end{proof}

\begin{remark}
  \label{remark:eps_choice}We can choose $\varepsilon =2^{-2/ \gamma} ( K ( 1+
  \| X \|_{\gamma} ) )^{-1/ \gamma}$ where $K$ is the constant of Lemma
  \ref{lemma:holder_norm_flow}.
\end{remark}

An immediate corollary gives an estimate of the growth of the function
$\varphi ( \Phi_{t} ( . ) )$.

\begin{corollary}
  \label{corollary:flot_schwarz}Let $b$ and $\Phi$ as in the previous lemma
  and let $\varphi \in C^{\infty}_{c} ( \mathbbm{R}^{d} ,\mathbbm{R}^{m} )$.
  Then for all $t \in [0,T]$ and all $N \in \mathbbm{N}$, $x \rightarrow ( 1+
  | x | )^{N} \varphi ( \Phi_{t} ( x ) ) \in L^{\infty} ( \mathbbm{R}^{d} )$.
\end{corollary}

\begin{proof}
  Let $0=t_{0} <t_{1} < \cdots <t_{n+1} =t$ such that $| t_{i+1} -t_{i} | <
  \varepsilon$ where $\varepsilon >0$ is chosen as in the previous lemma and
  $n$ is chosen as small as possible. Hence
  \[ | ( 1+ | x | )^{N} \varphi ( \Phi_{t} ( x ) ) | \leqslant ( 1+ | x |
     )^{N} \left( | \varphi ( x ) | + \sum_{i=0}^{n} | \varphi (
     \Phi_{t_{i+1}} ( x ) ) - \varphi ( \Phi_{t_{i}} ( x ) ) | \right) \]
  \[ \leqslant ( 1+ | x | )^{N} \left( \sum_{i=0}^{n} \int_{0}^{1} \dd r |
     D \varphi ( r ( \Phi_{t_{i+1}} ( x ) - \Phi_{t_{i}} ( x ) ) +
     \Phi_{t_{i}} ( x ) ) | | \Phi_{t_{i+1}} ( x ) - \Phi_{t_{i}} ( x ) | + |
     \varphi ( x ) | \right) \]
  \[ \lesssim ( n+1 ) \lesssim \frac{1}{\varepsilon} . \]
  
\end{proof}

Finally, thanks to Lemma~\ref{remark:global_holder_norm}, we are able to give
some estimates for the $\mathcal{C}^{\gamma}$ and $\mathcal{D}_{X}^{\gamma}$
norms of $\varphi ( \Phi_{.} ( x ) )$.

\begin{corollary}
  \label{corollary:holder_norm}Let $b$ as in Lemma~\ref{lemma:bound_control},
  and $\varphi \in C^{\infty}_{c} ( \mathbbm{R}^{d } ,\mathbbm{R}^{m} )$, then
  $\varphi ( \Phi_{.} ( x ) ) \in \mathcal{D}^{\gamma}_{X} ( [0,T] )$.
  Furthermore for all $N \geqslant 0$, we have
  \[ \| \varphi ( \Phi_{.} ( x ) ) \|_{\gamma} \lesssim ( 1+ \| X \|_{\gamma}
     )^{1+1/ \gamma} ( 1+ | x | )^{-N} \]
  and the implicit constant on the right hand side is nondecreasing in all the
  parameters.
\end{corollary}

\begin{proof}
  Let $\varepsilon >0$ as in Lemma \ref{lemma:bound_control} and $| t-s |
  \leqslant \varepsilon$ and $N>0$. If $C$ denotes the constant of Lemma
  \ref{lemma:bound_control}, we have
  \begin{eqnarray*}
    | \varphi ( \Phi_{t} ( x ) ) - \varphi ( \Phi_{s} ( x ) ) | & = & \left|
    \int_{s}^{t} \dd r D \varphi ( r ( \Phi_{t} ( x ) - \Phi_{s} ( x ) ) +
    \Phi_{s} ( x ) ) . ( \Phi_{t} ( x ) - \Phi_{s} ( x ) ) \right|\\
    & \leqslant & C \| \Phi_{.} ( x ) \|_{\gamma} ( 1+ | x | )^{- ( N+1 )} |
    t-s |^{\gamma} .\\
    & \leqslant & C K ( 1+ \| X \|_{\gamma} ) ( 1+ | x | )^{-N} | t-s
    |^{\gamma}
  \end{eqnarray*}
  Hence, thanks to Lemma \ref{remark:global_holder_norm}, we have
  \[ \| \varphi ( \Phi_{.} ( x ) ) \|_{\gamma} \leqslant \frac{T}{\varepsilon}
     C_{N+1}   ( 1+ \| X \|_{\gamma} ) K ( 1+ | x | )^{-N} . \]
  Thanks to Remark \ref{remark:eps_choice} we can choose $\varepsilon =2^{-2/
  \gamma} K^{-1/ \gamma}$, we finally have
  \[ \| \varphi ( \Phi_{.} ( x ) ) \|_{\gamma} \leqslant T C_{N+1} K^{1+1/
     \gamma} ( 1+ \| X \|_{\gamma} )^{1+1/ \gamma} ( 1+ | x | )^{-N} . \]
  To end the proof, we just have to remember that $K$ and $C_{N+1}$ are
  nondecreasing in the parameters.
\end{proof}

\begin{lemma}
  \label{lemma:smooth_control}Let $b$ and $X$ as in Lemma
  \ref{lemma:bound_control} and with $\tmop{div}  b \in L^{\infty} (
  \mathbbm{R}^{d} )$. For all $\varphi \in C^{\infty}_{c} ( \mathbbm{R}^{d}
  )$, $u ( \nabla \varphi ) \in \mathcal{D}^{\gamma}_{X}$ is controlled by $X$
  and there exists a constant $C_{\mathcal{D}} ( \varphi )$ which is
  nondecreasing in $\| b \|_{\tmop{Lin}} , \| \tmop{div}  b \|_{\infty}$ and
  $T$ such that
  \[ \| u ( \nabla \varphi ) \|_{\mathcal{D}^{\gamma}_{X}} \leqslant
     C_{\mathcal{D}} ( \varphi ) \| u_{0} \|_{\infty} ( 1+ \| X \|_{\gamma}
     )^{1+1/ \gamma} . \]
\end{lemma}

\begin{proof}
  We first rewrite $u ( \varphi )$ in a more suitable way. We use
  Theorem~\ref{theorem:jacobien} and Lemma~\ref{proposition:ODE_inverse_flow}
  and we have
  \begin{eqnarray*}
    u_{t} ( \nabla \varphi )_{} & = & \int_{\mathbbm{R}^{d}} u_{t} ( x )
    \nabla \varphi ( x ) \dd x\\
    & = & \int_{\mathbbm{R}^{d}} u_{0} ( \Phi^{-1}_{t} ( x ) ) \nabla \varphi
    ( x ) \dd x\\
    & = & \int_{\mathbbm{R}^{d}} u_{0} ( x ) \nabla \varphi ( \Phi_{t} ( x )
    ) | \tmop{Jac} ( \Phi_{t} ( x ) ) | \dd x.
  \end{eqnarray*}
  Thanks to Lemma \ref{lemma:product_controlled_path}, we know that when $a,b$
  are controlled by $X$, the product $a b$ is also controlled by $X$ and
  furthermore $\| a b \|_{\mathcal{D}^{\gamma}_{X}} \lesssim \| a
  \|_{\mathcal{D}^{\gamma}_{X}} \| b \|_{\mathcal{D}^{\gamma}_{X}}$. Hence, in
  order to prove that $u ( \nabla \varphi )$ is controlled it is enough to
  prove that for all $x \in \mathbbm{R}^{d}$, $t \mapsto \nabla \varphi (
  \Phi_{t} ( x ) )$ and $t \mapsto | \tmop{Jac} ( \Phi_{t} ( x ) ) |$ are
  controlled, with good estimates in $x$ for the controlled norms of those two
  functions. Since everything is smooth here, we can apply Theorem
  \ref{theorem:jacobien} and we have
  \[ | \tmop{Jac} ( \Phi_{t} ( x ) ) | = \exp \left( \int_{0}^{t} ( \tmop{div}
     b ) ( \Phi_{q} ( x ) ) \dd q \right) . \]
  Hence
  \[ | \tmop{Jac} ( \Phi_{t} ( x ) ) | \leqslant \exp ( T \| \tmop{div}  b
     \|_{\infty} ) \nosymbol . \]
  Furthermore
  \[ |   | \tmop{Jac} ( \Phi_{t} ( x ) ) | - | \tmop{Jac} ( \Phi_{s} ( x ) ) |
     |_{} \leqslant | t-s | \| \tmop{div}  b \|_{\infty} \exp ( T \|
     \tmop{div}  b \|_{\infty} ) , \]
  hence $| \tmop{Jac} ( \Phi_{.} ( x ) ) | \in \mathcal{D}^{\gamma}_{X}$, its
  derivative is zero and its recaller is itself, and we have
  \[ \| \tmop{Jac} ( \Phi_{.} ( x ) ) \|_{\mathcal{D}^{\gamma}_{X}} \lesssim
     ( 1+ \| \tmop{div}  b \|_{\infty} ) \exp ( T \| \tmop{div}  b \|_{\infty}
     ) . \]
  We need a bit more work to handle $\nabla \varphi ( \Phi_{.} ( x ) )$.
  Thanks to Corollaries~\ref{corollary:flot_schwarz}
  and~\ref{corollary:holder_norm}, for all $N \geqslant 0$ we already have
  \[ \| \nabla \varphi ( \Phi_{.} ( x ) ) \|_{\gamma} + \| \nabla \varphi (
     \Phi_{.} ( x ) ) \|_{\infty} \lesssim ( 1+ \| X \|_{\gamma} )^{1+1/
     \gamma} ( 1+ | x | )^{-N} . \]
  Moreover
  \begin{eqnarray*}
    \lefteqn{ \nabla \varphi ( \Phi_{t} ( x ) ) - \nabla \varphi ( \Phi_{s} ( x ) )}\\ & =
    & D^{2} \varphi ( \Phi_{s} ( x ) ) . ( X_{t} -X_{s} )\\
    &  & + \int_{0}^{1} \dd r D^{2} \varphi ( r ( \Phi_{t} ( x ) -
    \Phi_{s} ( x ) ) + \Phi_{s} ( x ) ) -D^{2} \varphi ( \Phi_{s} ( x ) ) . (
    X_{t} -X_{s} )\\
    &  & + \int_{0}^{1} \dd r D^{2} \varphi ( r ( \Phi_{t} ( x ) -
    \Phi_{s} ( x ) ) + \Phi_{s} ( x ) ) . \int_{s}^{t} b_{q} ( \Phi_{q} ( x )
    ) \dd q\\
    & = & \nabla \varphi ( \Phi_{.} ( x ) )'_{s} X_{s,t} + \nabla \varphi (
    \Phi_{.} ( x ) )_{s,t}^{\#}  .
  \end{eqnarray*}
  Thanks to Corollary~\ref{corollary:holder_norm},
  \[ \| D^{2} \varphi ( \Phi_{.} ( x ) ) \|_{\gamma} \lesssim ( 1+ \| X
     \|_{\gamma} )^{1+1/ \gamma} ( 1+ | x | )^{-N} . \]
  Next, thanks to Lemma~\ref{remark:global_holder_norm}, it is enough to
  control the local norm of $(s,t) \mapsto \nabla \varphi ( \Phi_{.} ( x )
  )_{s,t}^{\#}$ when we choose $\varepsilon$ as in
  Lemma~\ref{lemma:bound_control}. By Lemma~\ref{lemma:holder_norm_flow}, we
  already know that,
  \[ \left| \int_{s}^{t} b_{q} ( \Phi_{q} ( x ) ) \dd q \right| \lesssim K
     | t-s | ( 1+ \| b \|_{\tmop{Lin}} ) ( 1+ | x | ) ( 1+ \| X \|_{\gamma} )
     , \]
  hence it is enough to bound $\int_{0}^{1} \dd r D^{2} \varphi ( r (
  \Phi_{t} ( x ) - \Phi_{s} ( x ) ) + \Phi_{s} ( x ) )$. But thanks to
  Lemma~\ref{lemma:bound_control}, this quantity is bounded by $( 1+ | x |
  )^{- ( N+1 )}$. Furthermore
  \[ \int_{0}^{1} \dd r  ( D^{2} \varphi ( r ( \Phi_{t} ( x ) - \Phi_{s} (
     x ) ) + \Phi_{s} ( x ) ) -D^{2} \varphi ( \Phi_{s} ( x ) ) ) . ( X_{t}
     -X_{s} ) \]
  \[ = \int_{0}^{1} \dd r \int_{0}^{1} \dd q D^{3} \varphi ( r q (
     \Phi_{t} ( x ) - \Phi_{s} ( x ) ) + \Phi_{s} ( x ) ) . ( \Phi_{t} ( x ) -
     \Phi_{s} ( x ) ) . ( X_{t} -X_{s} ) \]
  and again thanks to Lemmas \ref{lemma:holder_norm_flow} and
  \ref{lemma:bound_control} we have
  \[ \sup_{x \in \mathbbm{R}^{d}} ( 1+ | x | )^{-N} \sup_{s \neq t, | t-s |
     \leqslant \varepsilon} | \nabla \varphi ( \Phi_{.} ( x ) )_{s,t}^{\#} | /
     | t-s | <+ \infty . \]
  Furthermore, since $\varepsilon \sim ( ( 1+ \| X \|_{\gamma} ) K )^{-1/
  \gamma}$ thanks to Remark \ref{remark:eps_choice}, thanks to Lemma
  \ref{remark:global_holder_norm} we have
  \[ \| \nabla \varphi ( \Phi_{.} ( x ) ) \|_{\mathcal{D}^{\gamma}_{X}}
     \lesssim ( 1+ | x | )^{-N} ( 1+ \| X \|_{\gamma} )^{1+1/ \gamma} . \]
  As all the previous constants are nondecreasing with respect to $\| b
  \|_{\tmop{Lin}}$, so is the implicit constant in this last inequality.
  Hence, when we put this inequality and the inequality for $| \tmop{Jac} (
  \Phi ( x ) ) |$ together, we have
  \begin{equation}
    \| \nabla \varphi ( \Phi_{.} ( x ) ) \tmop{Jac} ( \Phi_{.} ( x ) )
    \|_{\mathcal{D}^{\gamma}_{X}} \leqslant C_{N,T, \varphi} (    \| b
    \|_{\tmop{Lin}} , \| \tmop{div}  b \|_{\infty} ) ( 1+ | x | )^{-N} ( 1+ \|
    X \|_{\gamma} )^{1+1/ \gamma} \label{eq:integrability_controlled_norm}
  \end{equation}
  where the constant $C$ is nondecreasing in $T $, $\| b \|_{\tmop{Lin}}$ and
  $\| \tmop{div}  b \|_{\infty}$. Finally as
  \[ u_{t} ( \nabla \varphi ) = \int_{\mathbbm{R}^{d}} u_{0} ( x ) \nabla
     \varphi ( \Phi_{t} ( x ) ) \tmop{Jac} ( \Phi_{t} ( x ) ) \dd x, \]
  the function $u ( \nabla \varphi )$ is controlled by $X$ and
  \[ \| \nabla \varphi ( \Phi_{.} ( x ) ) \tmop{Jac} ( \Phi_{.} ( x ) )
     \|_{\mathcal{D}^{\gamma}_{X}} \leqslant C ( T , \| X \|_{\gamma}  , \| b
     \|_{\tmop{Lin}} , \| \tmop{div}  b \|_{\infty} , \varphi ,D \varphi
     ,D^{2} \varphi ,D^{3} \varphi ) \| u_{0} \|_{\infty} . \]
  where the constant $C$ is the constant of Equation
  {\eqref{eq:integrability_controlled_norm}} with $N=d+1$.
\end{proof}

\begin{remark}
  \label{remark:def_sol}The last proof shows gives that for all $\psi \in
  \mathcal{C}^{\infty}_{c} ( \mathbbm{R}^{d} )$, $u ( \psi ) \in
  \mathcal{C}^{\gamma} ( \mathbbm{R}^{d} )$ with the bound
  \[ \| u ( \psi ) \|_{\gamma} \leqslant C ( T, \| u_{0} \|_{\infty} , \| b
     \|_{\infty , \tmop{Lin}} , \psi ) ( 1+ \| X \|_{\gamma} )^{1+1/ \gamma} .
  \]
  Hence, for all $t \in [0,T]$, $u_{t} ( \psi )$ has a meaning as a continuous
  function of $t$.
\end{remark}

The last proof, and particularly
Equation~{\eqref{eq:integrability_controlled_norm}} gives us the following
corollary

\begin{corollary}
  \label{coro:regular_control}Let $\varphi_{0} \in C^{\infty}_{c} (
  \mathbbm{R}^{d} )$. For all $x \in \mathbbm{R}^{d}$ the function $t
  \rightarrow \varphi_{0} ( \Phi_{t} ( x ) ) e^{\int_{0}^{t} \tmop{div}  b (
  \Phi_{q} ( x ) ) \dd q}$ is controlled by $X \nosymbol$. Furthermore $x
  \rightarrow \| \varphi_{0} ( \Phi_{.} ( x ) ) \|_{\mathcal{D}^{\gamma}_{X}}$
  is decreasing faster than any polynomial and
  \[ \| \| \varphi_{0} ( \Phi_{.} ( x ) ) \|_{\mathcal{D}^{\gamma}_{X} (
     [0,T] )} ( 1+ | x | )^{-N} \|_{L^{\infty} ( \mathbbm{R}^{d} )} \leqslant
     C ( T, \| b \|_{\infty , \tmop{Lin}} , \varphi_{0} ,D \varphi_{0} ,D^{2}
     \varphi_{0} ) ( 1+ \| X \|_{\gamma} )^{1+1/ \gamma} . \]
\end{corollary}

\begin{remark}
  In fact, to prove that $u ( \nabla \varphi_{0} )$ is controlled by $X$ we do
  not need $\varphi_{0} \in C^{\infty}_{c} ( \mathbbm{R}^{d} )$, but only
  $\varphi_{0} \in C^{3} ( \mathbbm{R}^{d} )$ with $x \rightarrow ( 1+ | x |
  )^{d+2} ( | \nabla \varphi_{0} ( x ) | + | D^{2} \varphi ( x ) | + | D^{3}
  \varphi ( x ) | ) \in L^{\infty} ( \mathbbm{R}^{d} )$.
\end{remark}

When we test $( t,x ) \mapsto u_{0} ( \Phi^{-1}_{t} ( x ) )$ against smooth
compactly supported function $\varphi$, thanks to a change of variable,
proving that $u ( \varphi )$ is controlled by $X$ is equivalent proving that
$t \mapsto \varphi ( \Phi_{t} ( x ) )$ is controlled by $X$ and has good
integrability properties, which is what the previous lemmas have just
achieved.
	
	\subsection{Weak controlled solutions}
	The results of subsection \ref{subsection:taylor} allow us to prove the existence of weak controlled solution in the smooth case (Theorem \ref{theorem:smooth_transport}) and in the non-smooth case (Theorem \ref{theorem:existence_WCS}).
	
	\begin{proof}[Proof of Theorem \ref{theorem:smooth_transport}]
  Let us first consider the weak solution $u$ of equation {\eqref{eq:STE}}
  (see Appendix~\ref{appendix:characteristics}). We have shown that for all
  $\varphi \in C^{\infty}_{c} ( \mathbbm{R}^{d } )$, $u ( \varphi ) \in
  \mathcal{C}^{\gamma} ( [0,T] )$ and $u ( \nabla \varphi ) \in
  \mathcal{D}^{\gamma}_{X} ( [0,T] )$, hence $\int_{0}^{.} u_{s} ( \nabla
  \varphi ) \dd \mathbbm{X}_{s}$ is well defined as a rough integral.
  Furthermore, for $t \in [0,T]$
  \begin{eqnarray*}
    u_{t} ( \tmop{div} ( b_{t} \varphi ) ) & = & \int_{\mathbbm{R}^{d}} \dd
    x u_{t} ( x ) ( b_{t} ( x ) . \nabla \varphi_{} ( x ) + \varphi ( x )  
    \tmop{div}  b_{t} ( x ) ) \dd x\\
    & = & \int_{\mathbbm{R}^{d}} u_{0} ( x ) ( b_{t} ( \Phi_{t} ( x ) ) .
    \nabla \varphi ( \Phi_{t} ( x ) ) + \varphi ( \Phi_{t} ( x ) )  
    \tmop{div}  b_{t} ( \Phi_{t} ( x ) ) ) | \tmop{Jac} ( \Phi_{t} ( x ) ) |
    \dd x
  \end{eqnarray*}
  and as $| \tmop{Jac} ( \Phi_{t} ( x ) ) | \leqslant e^{\| \tmop{div}  b
  \|_{\infty} T}$ and $| \tmop{Jac} ( \Phi_{t}^{-1} ( x ) ) | \leqslant e^{\|
  \tmop{div}  b \|_{\infty} T}$
  \begin{eqnarray*}
    | u_{t} ( \tmop{div}   ( b_{t} \varphi ) ) | & \leqslant & \| u_{0}
    \|_{\infty} e^{\| \tmop{div}  b \|_{\infty} T} \int_{\mathbbm{R}^{d}} | (
    b_{t} ( \Phi_{t} ( x ) ) . \nabla \varphi ( \Phi_{t} ( x ) ) + \varphi (
    \Phi_{t} ( x ) )   \tmop{div}  b_{t} ( \Phi_{t} ( x ) ) ) | \dd x\\
    & \leqslant & \| u_{0} \|_{\infty}  e^{2 \| \tmop{div}  b \|_{\infty} T}
    \int_{\mathbbm{R}^{d}} \dd x | \nobracket b_{t} ( x ) . \nabla \varphi
    ( x ) + \varphi ( x )   \tmop{div}  b_{t} ( x \nobracket ) | \dd x\\
    & \lesssim_{\varphi , \nabla \varphi} & \| u_{0} \|_{\infty} \| b
    \|_{\infty ; \tmop{Lin}}  e^{2 \| \tmop{div}  b \|_{\infty} T} ( \| b
    \|_{\infty , \tmop{Lin}} + \| \tmop{div}  b \|_{\infty} )
  \end{eqnarray*}
  Hence $u ( \tmop{div}   ( b. \nabla \varphi ) ) \in L^{\infty} ( [0,T] )$
  and $\int_{0}^{t} \dd s u_{s} ( b_{s} . \nabla \varphi )$ is well-defined
  for all $t \in [0,T]$. As $u$ is a weak solution of {\eqref{eq:STE}}, for
  all $\varphi \in C^{\infty}_{c} ( \mathbbm{R}^{d} )$ and all $f \in
  C^{\infty}_{c} ( [0,T] )$, the smooth functions with compact support in
  $[0,T]$,
  \[ u_{0} ( \varphi ) f_{0} -u_{T} ( f_{T} ) + \int_{0}^{T} \left( u_{t} (
     \varphi ) - \int_{0}^{t} u_{s} ( b_{s} . \nabla \varphi + \tmop{div} (
     b_{s} ) \varphi ) \dd s- \int_{0}^{t} u_{s} ( \nabla \varphi )
     \dot{X}_{s} \dd s \right) \dot{f}_{t} \dd t=0 \]
  and since $\int_{0}^{t} u_{s} ( \nabla \varphi ) \dot{X}_{s} \dd s=
  \int_{0}^{t} u_{s} ( \nabla \varphi ) \dd \mathbf{X}_{s}$, thanks to
  Lemma~\ref{lemma:smooth_control}, we have

  \[ \int_{0}^{T} \left( u_{t} ( \varphi ) -u_{0} ( \varphi ) - \int_{0}^{t}
     u_{s} ( \tmop{div}   ( b_{s} \varphi ) ) \dd s- \int_{0}^{t} u_{s} (
     \nabla \varphi ) \dd \mathbf{X}_{s} \right) \dot{f}_{t} \dd t=0.
  \]
  But $t \rightarrow u_{t} ( \varphi ) -u_{0} ( \varphi ) - \int_{0}^{t} u_{s}
  ( \tmop{div}   ( b_{s} \varphi ) ) \dd s- \int_{0}^{t} u_{s} ( \nabla
  \varphi ) \dd \mathbf{X}_{s}$ is a continuous function, thanks to
  Remark~\ref{remark:def_sol}, Lemma~\ref{lemma:smooth_control} and the
  previous computation. Hence, for every $t \in [0,T]$,

  \[ u_{t} ( \varphi ) -u_{0} ( \varphi ) - \int_{0}^{t} u_{s} ( \tmop{div}  
     ( b_{s} \varphi ) ) \dd s- \int_{0}^{t} u_{s} ( \nabla \varphi )
     \dd \mathbf{X}_{s} =0 \]
  and $u$ is a WCS of equation {\eqref{eq:RTE}}.
  
  We will show that it is unique. Let $v$ a WCS to eq.~{\eqref{eq:STE}}. Then,
  we have for all $f \in C^{\infty}_{c} ( [0,T] )$ and all $\varphi \in
  C^{\infty}_{c} ( \mathbbm{R}^{d} )$,
  \[ \int_{0}^{T} v_{t} ( \varphi ) \dot{f}_{t} \dd t= \int_{0}^{T} \left(
     v_{0} ( \varphi ) + \int_{0}^{t} v_{s} ( \tmop{div}   ( b_{s} \varphi ) )
     \dd s + \int_{0}^{t} v_{s} ( \nabla f ) \dot{X}_{s} \dd s \right)
     \dot{f}_{t} \dd t \]
  hence
  \[ -v ( \partial_{t} f \otimes \varphi ) -v_{0} ( f_{0} \otimes \varphi ) +v
     ( \tmop{div} \{ ( b+ \dot{X} ) ( f \otimes \varphi ) \} ) =0. \]
  As this equation is linear, and thanks to the density of the linear span of
  $C^{\infty}_{c} ( [0,T] ) \otimes C^{\infty}_{c} ( \mathbbm{R}^{d} )$ into
  $C^{\infty}_{c} ( [0,T] \times \mathbbm{R}^{d} )$, for all $  \psi \in
  C^{\infty}_{c} ( [0,T] \times \mathbbm{R}^{d} )$,
  \[ -v ( \partial_{t} \psi ) -v_{0} ( \psi ( 0,. ) ) +v ( \tmop{div} \{ ( b+
     \dot{X} ) \psi \} ) =0 \]
  Hence, $v$ is a weak solution of equation {\eqref{eq:STE}}. By Appendix
  \ref{appendix:characteristics}, $v=u$ and uniqueness holds for the weak
  controlled solutions in the smooth case.
\end{proof}

Thanks to Theorem \ref{theorem:smooth_transport}, we now have all the tools to deal with the rough transport equation. We are
looking for weak controlled solutions. Here we will no longer suppose that $X$
is smooth. As usual in such a setting, we will approximate $X$ and $b$ in a
smooth way, and use the a priori bounds of the previous part to obtain
compactness.

\begin{proof}[Proof of Theorem \ref{theorem:existence_WCS}]
  As $\mathbf{X} \in \mathcal{R}^{\gamma} ( [0,T] )$, there exists a
  sequence $( \mathbf{X}^{\varepsilon} )_{\varepsilon} = ( X^{\varepsilon}
  ,\mathbbm{X}^{\varepsilon} )_{\varepsilon} \in \mathcal{R}^{\gamma} ( [0,T]
  )$ with $\mathbbm{X}^{\varepsilon}_{s,t} = \int_{s}^{t} X^{\varepsilon}_{s}
  -X^{\varepsilon}_{r} \dd X^{\varepsilon}_{r}$ such that $X^{\varepsilon}
  \in C^{1} ( \mathbbm{R}^{d} )$ and $\|
  \mathbf{X}-\mathbf{X}^{\varepsilon} \|_{\mathcal{R}^{\gamma}}
  \rightarrow 0$. We can also approximate $b$ by $b^{\varepsilon}$ such that
  $b^{\varepsilon} \in L^{\infty} ( [0,T]; \tmop{Lip} ( \mathbbm{R}^{d} ) \cap
  C^{1} ( \mathbbm{R}^{d} ) )$ and $\tmop{div}  b^{\varepsilon} \in L^{\infty}
  ( [0,T];C^{1}_{b} ( \mathbbm{R}^{d} ) )$. Let us consider the weak Rough
  solution $u^{\varepsilon}$ of the equation $\partial_{t} u^{\varepsilon}
  +b^{\varepsilon} . \nabla u^{\varepsilon} + \dot{X}^{\varepsilon}_{t} =0$
  with $u^{\varepsilon}_{0} =u_{0}$. Thanks to Theorem
  \ref{theorem:smooth_transport}, we know that for all $t \in [0,T]$ and all
  $\varphi \in C^{\infty}_{c} ( \mathbbm{R}^{d} )$
  \begin{equation}
    u^{\varepsilon}_{t} ( \varphi ) =u_{0} ( \varphi ) + \int_{0}^{t}
    u^{\varepsilon}_{s} ( b^{\varepsilon}_{s} . \nabla \varphi ) \dd s+
    \int_{0}^{t} u^{\varepsilon}_{s} ( \nabla \varphi ) \dd
    \mathbbm{X}^{\varepsilon}_{s} . \label{eq:mollified_RTE}
  \end{equation}
  The strategy here is to extract a subsequence such that each term from the
  previous equality converges. First let us show that there exists $u \in
  L^{\infty} ( [0,T] \times \mathbbm{R}^{d} )$ and a subsequence
  $\varepsilon_{n}$ such that
  \[ u^{\varepsilon_{n}} \xrightarrow{w- \ast  L^{1}} u. \]
  This is nearly straightforward since $u^{\varepsilon}_{t} ( x ) =u_{0} ( (
  \Phi^{\varepsilon} )^{-1}_{t} ( x ) )$ and $\| u^{\varepsilon}
  \|_{L^{\infty}} \leqslant \| u_{0} \|_{L^{\infty}}$. Hence $(
  u^{\varepsilon} )_{\varepsilon}$ is relatively compact for the weak-star
  topology of $L^{1}$, and we take a subsequence
  $\left(\varepsilon_{n}\right)$ and $u \in L^{\infty} ( [0,T] \times
  \mathbbm{R}^{d} )$ such that $u^{\varepsilon_{n}} \xrightarrow{w- \ast 
  L^{1}} u$. Furthermore, for all $\varphi \in C^{\infty}_{c} (
  \mathbbm{R}^{d} )$, we have
  \[ u^{\varepsilon_{n}} ( \varphi ) \xrightarrow{w- \ast  L^{1} ( [0,T] )} u
     ( \varphi ) \]
  But thanks to Remark~\ref{remark:def_sol}, we know that $u^{\varepsilon} (
  \varphi ) \in \mathcal{C}^{\gamma} ([0,T])$ and that
  \begin{eqnarray*}
    \| u^{\varepsilon} ( \varphi ) \|_{\gamma} & \leqslant & C ( T, \| u_{0}
    \|_{\infty} , \| b^{\varepsilon} \|_{\infty , \tmop{Lin}} , \varphi ) ( 1+
    \| X^{\varepsilon} \|_{\gamma} )^{1+1/ \gamma}\\
    & \lesssim & C ( T, \| u_{0} \|_{\infty} , \| b \|_{\infty , \tmop{Lin}}
    , \varphi ) ( 1+ \| X \|_{\gamma} )^{1+1/ \gamma}\\
    & < & + \infty .
  \end{eqnarray*}
  We can apply Arzel{\`a}-Ascoli to $u^{\varepsilon} ( \varphi )$, and there
  exists $l ( u, \varphi ) \in C ( [0,T] )$ such that, for another subsequence
  $( \tilde{\varepsilon}_{n} )$ of $( \varepsilon_{n} )$, $(
  u^{\tilde{\varepsilon}_{n}} ( \varphi ) )$ converges uniformly to $l ( u,
  \varphi )$. Hence we have $u ( \varphi ) =l ( u, \varphi )$ and $u ( \varphi
  ) \in \mathcal{C}^{\gamma} ( [0,T] )$. The same strategy works for
  $\int_{0}^{t} u^{\varepsilon}_{s} ( \tmop{div}   ( b^{\varepsilon}_{s}
  \varphi ) ) \dd s$ up to extraction of an other subsequence. Hence, for
  all $\varphi \in C^{\infty}_{c} ( \mathbbm{R}^{d} )$, there exists another
  subsequence, let us denote it again by $( \varepsilon_{n} )$ such that
  $\left( \int_{0}^{.} u^{\varepsilon_{n}}_{s} ( \tmop{div}   (
  b^{\varepsilon_{n}}_{s} . \nabla \varphi ) ) \dd s \right)$ converges
  uniformly to $\int_{0}^{.} u_{s} ( b_{s} . \nabla \varphi ) \dd s$.
  Thanks to Lemma~\ref{lemma:smooth_control}, as $\sup_{\varepsilon >0} \|
  X^{\varepsilon} \|_{0, \gamma} + \| b^{\varepsilon} \|_{\infty , \tmop{Lin}}
  + \| \tmop{div}  b^{\varepsilon} \|_{\infty} <+ \infty$ we have
  \[ \sup_{\varepsilon >0} \| u^{\varepsilon}_{s} ( \nabla \varphi )
     \|_{\mathcal{D}^{\gamma}_{X^{\varepsilon}}} <+ \infty . \]
  Hence, $u^{\varepsilon} ( \nabla \varphi )$ is bounded in the space
  $\mathcal{D}^{\gamma}_{X^{\varepsilon}}$, uniformly in $\varepsilon$. It is
  possible to apply the Arzel{\`a}-Ascoli theorem to $u^{\varepsilon} ( \nabla
  \varphi )$, $( u^{\varepsilon} ( \nabla \varphi ) )'$ and $( u^{\varepsilon}
  ( \nabla \varphi ) )^{\#}$, and there exists $u ( \nabla \varphi ) \in
  \mathcal{D}^{\gamma}_{X}$ such that
  \[ u^{\varepsilon_{n}} ( \nabla \varphi )
     \xrightarrow{\mathcal{D}^{\gamma}} u ( \nabla \varphi ) . \]
  Furthermore, thanks to the definition of the rough integral (see Theorem
  \ref{th:controlled_integral} above) and the comparison between controlled
  paths (see Lemma \ref{lemma:product_controlled_path}), we know that
  \[ \int_{0}^{.} u_{s}^{\varepsilon_{n}} ( \nabla \varphi ) \dd
     \mathbf{X}^{\varepsilon_{n}} \xrightarrow{\mathcal{C}^{\gamma}}
     \int_{0}^{.} u_{s} ( \nabla \varphi ) \dd \mathbf{X}_{s} . \]
  Hence the last term in equation {\eqref{eq:mollified_RTE}} converges.
  Finally we have shown that $u \in L^{\infty} ( [0,T] \times \mathbbm{R}^{d}
  )$ is such that for all $\varphi \in C^{\infty}_{c} ( \mathbbm{R}^{d} )$, $u
  ( \nabla \varphi )$ is controlled by $X$ and for all $t \in [0,T]$,
  \[ u_{t} ( \varphi ) =u_{0} ( \varphi ) + \int_{0}^{t} u_{s} ( b_{s} .
     \nabla \varphi ) \dd s+ \int_{0}^{t} u_{s} ( \nabla \varphi ) \dd
     \mathbf{X}_{s} \]
  and $u$ is a weak controlled solution to equation {\eqref{eq:RTE}}.
\end{proof}
	\subsection{Stochastic processes, measurable weak controlled solutions}
	The proof of Theorem \ref{theorem:existence_SWCS} is quite similar to the one in the deterministic case. The only -
but significant - difference is that we can no longer apply na{\"i}vely the
Arzel{\`a}-Ascoli theorem. As before we will use a weak-$\ast$-compactness
theorem to identify a limit. In order to find a H{\"o}lder continuous version
of this limit, we will use a sequence of partitions (here the dyadic numbers)
and Riemann sums. The end of the proof will be devoted to prove the
convergence of each term to the wanted quantities.

\begin{proof}[Proof of Theorem \ref{theorem:existence_SWCS}]
  Let $( \mathbf{X}^{\varepsilon} )$ the smooth approximation of
  $\mathbf{X}$. Let $k :\mathbbm{R}^{d} \rightarrow \mathbbm{R}^{d}$ a
  smooth mollifier, \tmtextit{i.e.} $k \in C^{\infty}_{c} ( \mathbbm{R}^{d}
  )$, $k=1$ for $| x | \leqslant 1$ and $k=0$ for $| x | >2$, and
  $k_{\varepsilon} ( x ) = \frac{1}{\varepsilon^{d}} k \left(
  \frac{x}{\varepsilon} \right)$. Let $b^{\varepsilon} =k_{\varepsilon} \ast
  b$, hence $b^{\varepsilon} \in L^{\infty} ( \Omega \times [0,T];C^{\infty} (
  \mathbbm{R}^{d} ) \cap \tmop{Lin} ( \mathbbm{R}^{d} ) )$ and for all $Y \in
  L^{1} ( \Omega )$
  \[ \mathbbm{E} [ \| b^{\varepsilon} -b \|_{L^{\infty} ( [0,T]; \tmop{Lin} (
     \mathbbm{R}^{d} ) )} Y ] \rightarrow 0 \]
  and $\| b^{\varepsilon} \|_{L^{\infty} ( \Omega \times [0,T]; \tmop{Lin} (
  \mathbbm{R}^{d} ) )} \leqslant_{} \| b \|_{L^{\infty} ( \Omega \times [0,T];
  \tmop{Lin} ( \mathbbm{R}^{d} ) )}$. Furthermore $\tmop{div}  b^{\varepsilon}
  \rightharpoonup^{L^{\infty} ( \Omega \times [0,T] \times \mathbbm{R}^{d} )}
  \tmop{div}  b$, \tmtextit{i.e.} for all $f \in L^{1} \nobracket ( \Omega
  \times [0,T] \times \nobracket \mathbbm{R}^{d} )$,
  \[ \mathbbm{E} \left[ \int_{[0,T] \times \mathbbm{R}^{d}} \tmop{div} 
     b^{\varepsilon}_{t} ( x ) f_{t} ( x ) \dd x \dd t \right]
     \rightarrow \mathbbm{E} \left[ \int_{[0,T] \times \mathbbm{R}^{d}}
     \tmop{div}  b_{t} ( x ) f_{t} ( x ) \dd x \dd t \right] \]
  and $\| \tmop{div}  b^{\varepsilon} \|_{L^{\infty} ( \Omega \times [0,T]
  \times \mathbbm{R}^{d} )} \lesssim \| \tmop{div}  b \|_{L^{\infty} ( \Omega
  \times [0,T] \times \mathbbm{R}^{d} )}$.
  
  Let $\Phi^{\varepsilon}$ the flow of the approximate equation
  $\Phi^{\varepsilon}_{t} ( x ) =x+ \int_{0}^{t} b^{\varepsilon}_{q} (
  \Phi^{\varepsilon}_{q} ( x ) ) \dd q+X^{\varepsilon}_{t}$ and $(
  \Phi^{\varepsilon} )^{-1}$ its inverse. We know, thanks to Theorem \
  \ref{theorem:smooth_transport}, that $u^{\varepsilon}_{t} ( x ) =u_{0} ( (
  \Phi^{\varepsilon} )^{-1}_{t} ( x ) )$ is a weak controlled solution of the
  approximate rough transport equation with initial condition $u_{0}$,
  \tmtextit{i.e.} for all $\varphi \in C^{\infty}_{c} ( \mathbbm{R}^{d} )$, $t
  \rightarrow u^{\varepsilon}_{t} ( \varphi )$ is controlled by
  $X^{\varepsilon}$ almost surely, and for all $s<t \in [0,T]$,
  \[ u^{\varepsilon}_{t} ( \varphi ) -u^{\varepsilon}_{s} ( \varphi ) =
     \int_{s}^{t} u_{q}^{\varepsilon} ( \tmop{div}   ( b_{q}^{\varepsilon}
     \varphi ) ) \dd q+ \int_{s}^{t} u_{q}^{\varepsilon} ( \nabla \varphi )
     \dd \mathbf{X}^{\varepsilon}_{q} \]
  First, as $u_{0} \in L^{\infty} ( \Omega \times \mathbbm{R}^{d} )$,
  \[ \| u^{\varepsilon} \|_{L^{\infty} ( \Omega \times [0,T] \times
     \mathbbm{R}^{d} )} \leqslant \| u_{0} \|_{\infty} . \]
  Hence there exists a subsequence abusively denoted again by $(
  u^{\varepsilon} )$, and $u \in L^{\infty} ( \Omega \times [0,T] \times
  \mathbbm{R}^{d} )$ such that for all $Y \in L^{1} ( \Omega )$, $f \in L^{1}
  ( [0,T] ) \nocomma$, $\varphi \in L^{1} ( \mathbbm{R}^{d} )$
  \[ \mathbbm{E} \left[ \int_{0}^{T} \dd t \int_{\mathbbm{R}^{d}}
     u^{\varepsilon}_{t} ( x ) \varphi ( x ) \dd x f_{t} \dd t Y \right]
     \rightarrow \mathbbm{E} \left[ \int_{0}^{T} \int_{\mathbbm{R}^{d}} u_{t}
     ( x ) \varphi ( x ) \dd x f_{t} \dd t Y \right] . \]
  Let us find a weakly continuous version of the limit $u$. Let us define for
  $n \geqslant 0$ the points of the dyadic partition of $[0,T]$ by $t^{n}_{i}
  =i T2^{-n}$. We denote by $\Pi_{n} = \{ t_{i}^{n} :i \in \{ 0, \ldots ,2^{n}
  \} \}$ and by \ $\Pi = \cup_{n \geqslant 1} \Pi_{n}$. As $\Pi$ is countable
  and for all $t \in \Pi$
  \[ \| u^{\varepsilon}_{t} \|_{L^{\infty} ( \Omega \times \mathbbm{R}^{d} )}
     \leqslant \| u_{0} \|_{\infty} <+ \infty , \]
  there exists a subsequence of $\left(u^{\varepsilon}\right)$, denoted again
  by $( u^{\varepsilon} )$ such that for all $t \in \Pi$ there exists
  $\tilde{u}_{t} \in L^{\infty} ( \Omega \times \mathbbm{R}^{d} )$ such that
  $u_{t}^{\varepsilon} \rightharpoonup \tilde{u}_{t}$ weakly in $L^{\infty} (
  \Omega \times \RR^d )$. Furthermore, for all $t \in \Pi \nocomma , \|
  \tilde{u}_{t} \|_{L^{\infty} ( \Omega \times \mathbbm{R}^{d} )} \leqslant \|
  u_{0} \|_{L^{\infty} ( \Omega \times \mathbbm{R}^{d} )}$.
  
  Furthermore, let us recall that thanks to Lemma \ref{lemma:smooth_control},
  $u^{\varepsilon} ( \varphi ) \in \mathcal{D}^{\gamma}_{X^{\varepsilon}}$
  almost surely and since $\| b^{\varepsilon} \|_{L^{\infty} ( \Omega \times
  [0,T]; \tmop{Lin} ( \mathbbm{R}^{d} ) )} \leqslant \| b \|_{L^{\infty} (
  \Omega \times [0,T]; \tmop{Lin} ( \mathbbm{R}^{d} ) )}$ and $\|
  \mathbf{X}^{\varepsilon} \|_{\mathcal{R}^{\gamma}} \leqslant \|
  \mathbf{X} \|_{\mathcal{R}^{\gamma}}$
  \[ \| u^{\varepsilon} ( \varphi )
     \|_{\mathcal{D}^{\gamma}_{X^{\varepsilon}}} \lesssim K_{\varphi ,b,u_{0}}
     ( 1+ \| \mathbf{X} \|_{\mathcal{R}^{\gamma}} )^{1+1/ \gamma} \nosymbol
     . \]
  We have, for all $s,t \in \Pi$, all $\varphi \in C^{\infty}_{c} (
  \mathbbm{R}^{d} )$, since $\tilde{u}_{t} ( \varphi ) - \tilde{u}_{s} (
  \varphi ) \in L^{\infty} ( \Omega )$, for all $p \geqslant 1$,
  \[ \mathbbm{E} [ ( \delta u^{\varepsilon}_{s,t} ( \varphi ) ) ( \delta
     \tilde{u}_{s,t} ( \varphi ) )^{2p-1} ] \rightarrow \mathbbm{E} [ ( \delta
     \tilde{u}_{s,t} ( \varphi ) )^{2p} ] . \]
  But, by H{\"o}lder's inequality,
  \begin{eqnarray*}
    \mathbbm{E} [ \delta u^{\varepsilon_{n}}_{s,t} ( \varphi ) ( \delta
    \tilde{u}_{s,t} ( \varphi ) )^{2p-1} ] & \leqslant & \mathbbm{E} [ (
    \delta \tilde{u}_{s,t} ( \varphi )_{} )^{2p} ]^{1-1/2p} \mathbbm{E} [ (
    \delta u^{\varepsilon_{n}}_{s,t} ( \varphi ) )^{2p} ]^{1/2p}\\
    & \lesssim_{p} & \mathbbm{E} [ ( \delta \tilde{u}_{s,t} ( \varphi )
    )^{2p} ]^{2p/ ( 2p-1 )} | t-s |^{\gamma}
  \end{eqnarray*}
  Hence
  \begin{equation}
    \mathbbm{E} [ ( \delta \tilde{u}_{s,t} ( \varphi ) )^{2p} ] \lesssim_{p}
    K_{\varphi , \| b \| , \| u_{0} \|} | t-s |^{2 \gamma p} .
    \label{eq:holder_bound_u_tilde}
  \end{equation}
  Hence, for $t \in [0,T]$, $( t_{k} )_{k} , ( \tilde{t}_{k} )_{k} \in
  \Pi^{\mathbbm{N}}$ such that $t_{k} \rightarrow t$ and $\tilde{t}_{k}
  \rightarrow t$, $( \tilde{u}_{t_{k}} ( \varphi ) - \tilde{u}_{\tilde{t}_{k}}
  ( \varphi ) )_{k}$ converges to zero in $L^{p} ( \Omega )$. Furthermore $(
  \tilde{u}_{t_{k}} ( \varphi ) )_{k}$ is Cauchy in $L^{2p} ( \Omega )$. We
  call its limit $\tilde{u}_{t} ( \varphi )$. Note that it is independent of
  the sequence in $\Pi$. Since $\| \tilde{u}_{t} ( \varphi ) \|_{L^{\infty} (
  \Omega )} \leqslant \| u_{0} \|_{L^{\infty} ( \Omega \times \mathbbm{R}^{d}
  )} \| \varphi \|_{L^{1} ( \mathbbm{R}^{d} )}$, by the dominated convergence
  theorem, the bound in Equation {\eqref{eq:holder_bound_u_tilde}} holds
  for all $s,t \in [0,T]$. Hence, thanks to the usual Kolmogorov continuity
  Theorem, almost surely for all $\gamma' < \gamma$, $\tilde{u} ( \varphi )
  \in \mathcal{C}^{\gamma'} ( [0,T] )$ and for all $\mathbbm{E} [ \| \tilde{u}
  ( \varphi ) \|_{\mathcal{C}^{\gamma}}^{p} ] <+ \infty$ for $1 \leqslant p<+
  \infty$.
  
  Let us take $f \in C^{\infty}_{c} ( [0,T] )$, $Y \in L^{q} ( \Omega )$ for
  $1<q \leqslant + \infty$ and, $\varphi \in C^{\infty}_{c} ( \mathbbm{R}^{d}
  )$. As for every point $s \in \Pi$, $\tilde{u}_{s} ( \varphi ) =u_{s} (
  \varphi )$, and thanks to Riemann sum estimates, where we define
  \[ S_{n}^{\varepsilon} ( \varphi ,f ) =2^{-n} \sum_{i=0}^{2^{n} -1}
     f_{t_{i}} u^{\varepsilon}_{t_{i}^{n}} ( \varphi )   \tmop{and}  S_{n} (
     \varphi ,f ) =2^{-n} \sum_{i=0}^{2^{n} -1} f_{t_{i}} u_{t_{i}^{n}} (
     \varphi ) \]
  we have
  \begin{eqnarray}
    \left| \mathbbm{E} \left[ \int_{0}^{T} ( u^{\varepsilon}_{t} ( \varphi ) -
    \tilde{u}_{t} ( \varphi ) ) f_{t} \dd t Y \right] \right| & \leqslant &
    \mathbbm{E} \left[ \left| \int^{T}_{0} f_{t} u^{\varepsilon}_{t} ( \varphi
    ) \dd t-S_{n}^{\varepsilon} ( \varphi ,f ) \right|   | Y | \right] 
    \label{eq:riemann_weak_cv}\\
    &  & + | \mathbbm{E} [ ( S_{n}^{\varepsilon} ( \varphi ,f ) -S_{n} (
    \varphi ,f ) )  Y ] | \nonumber\\
    &  & +\mathbbm{E} \left[ \left| \int^{T}_{0} f_{t} u_{t} ( \varphi )
    \dd t-S_{n} ( \varphi ,f ) \right|   | Y | \right] \nonumber\\
    & \lesssim & \| f' \|_{\infty} \| u_{0} \|_{L^{\infty} ( \Omega \times
    \mathbbm{R}^{d} )} \mathbbm{E} [ | Y | ] 2^{-n} \nonumber\\
    &  & + \| f \|_{\infty} 2^{-n} \sum_{i=0}^{2^{n} -1} | \mathbbm{E} [
    u^{\varepsilon}_{t^{n}_{i}} ( \varphi ) -u^{\varepsilon}_{t^{n}_{i}} (
    \varphi ) Y ] | \nonumber\\
    &  & + \| f \|_{\infty} 2^{- \gamma n} \mathbbm{E} [ ( \sup_{\varepsilon}
    \| u^{\varepsilon} ( \varphi ) \|_{\gamma} \nosymbol + \| \tilde{u}
    \|_{\gamma} ) | Y | ] . \nonumber
  \end{eqnarray}
  Letting $\varepsilon$ go to zero and $n$ go to infinity, we have
  \[ \mathbbm{E} \left[ \int_{0}^{T} f_{t} u_{t} ( \varphi ) \dd t Y
     \right] =\mathbbm{E} \left[ \int_{0}^{T} f_{t} \tilde{u}_{t} ( \varphi )
     \dd t Y \right] . \]
  hence almost surely and for almost all $t \in [0,T]$ $u_{t} ( \varphi ) =
  \tilde{u}_{t} ( \varphi )$. Since it is also true for all $\varphi \in
  C^{\infty}_{c} ( \mathbbm{R}^{d} )$, we have shown that there exists a
  version $\tilde{u}$ of the weak limit $u$ such that $\tilde{u} \in
  L^{\infty} ( \Omega \times [0,T] \times \mathbbm{R}^{d} )$ and
  \[ u ( \varphi ) \in L^{p} ( \Omega ;\mathcal{C}^{\gamma} ( [0,T] ) ) . \]
  From now on we will consider only this version and denote it by $u$.
  
  Thanks to the hypothesis on $b$ and $b^{\varepsilon}$, we know that for all
  $\varphi \in C^{\infty}_{c} ( \mathbbm{R}^{d} )$, $\tmop{div}   (
  b^{\varepsilon} \varphi ) \rightarrow \tmop{div}   ( b \varphi )$ strongly
  in $L^{1} ( [0,T] \times \mathbbm{R}^{d} )$. Hence, for all $t \in [0,T]$
  \[ \int_{0}^{t} u_{q}^{\varepsilon} ( \tmop{div}   ( b^{\varepsilon}_{q}
     \varphi ) ) \dd q= \int_{0}^{T} \dd q \int_{\mathbbm{R}^{d}} \dd
     x u^{\varepsilon}_{q} ( x ) \tmop{div}   ( b^{\varepsilon}_{q} \varphi )
     \mathbbm{1}_{[ 0,t ]} ( q ) \rightarrow \int_{0}^{t} u_{q} ( \tmop{div}  
     ( b_{q} \varphi ) ) \dd q, \]
  where the convergence is weak in $L^{\infty} ( \Omega )$. Furthermore, since
  the following bound holds
  \[ \| u_{q}^{\varepsilon} ( \tmop{div} ( b^{\varepsilon}_{q} \varphi ) )
     \|_{L^{\infty} ( \Omega \times [0,T] )} \leqslant \| u_{0} \|_{( \Omega
     \times \mathbbm{R}^{d} )} \| \tmop{div}  b \varphi \|_{L^{\infty} (
     \Omega \times [0,T];L^{1} ( \mathbbm{R}^{d} ) )} , \]
  by the dominated convergence theorem we have
  \[ \int_{0}^{t} u_{q}^{\varepsilon} ( \tmop{div}   ( b^{\varepsilon}_{q}
     \varphi ) ) \dd q \rightharpoonup^{L^{\infty} ( \Omega \times [0,T] )}
     \int_{0}^{t} u_{q} ( \tmop{div}   ( b_{q} \varphi ) ) \dd q. \]
  In order to prove that $u$ is a weak controlled solution of the Rough
  transport equation, it remains to show that the last term $\int_{0}^{.}
  u_{q}^{\varepsilon} ( \nabla \varphi ) \dd
  \mathbf{X}^{\varepsilon}_{q}$ converges weakly in $L^{\infty} ( \Omega
  ;\mathcal{D}' ( [0,T] ) )$ where $\mathcal{D}'$ is the set of distributions
  on $[0,T]$ to $\int_{0}^{.} u_{q} ( \nabla \varphi ) \dd
  \mathbf{X}_{q}$. In order to do that, it is necessary to show that $u (
  \nabla \varphi )$ is controlled by $X$. Thanks to the last construction, for
  all $m \in \mathbbm{N}$ and all $\varphi \in C^{\infty}_{c} (
  \mathbbm{R}^{d} ,\mathbbm{R}^{m} )$ and all $t \in \Pi$,
  $u^{\varepsilon}_{t} ( \varphi )$ converges to $u_{t} ( \varphi )$ weakly in
  $L^{\infty} ( \Omega )$ and furthermore $t \rightarrow u_{t} ( \varphi )$ is
  almost surely in $\mathcal{C}^{\gamma'} ( [0,T] )$ for all $\gamma' <
  \gamma$, and $\| u ( \varphi ) \|_{\mathcal{C}^{\gamma'}} \in L^{p} ( \Omega
  )$ for all $1 \leqslant p<+ \infty$. Hence, this holds for $\nabla \varphi$
  and $\nabla^{2} \varphi$ when $\varphi \in C^{\infty}_{c} ( \mathbbm{R}^{d}
  )$. Furthermore, thanks to Lemma \ref{lemma:smooth_control},
  $u^{\varepsilon} ( \nabla \varphi )' =u^{\varepsilon} ( \nabla^{2} \varphi
  )$.
  
  For all $s \leqslant t \in \Pi$ and all $Y \in L^{q} ( \Omega )$ for $q>1$,
  we have $X_{s,t} Y \in L^{1} ( \Omega )$, and the following computation
  holds
  \begin{eqnarray*}
    | \mathbbm{E} [ u^{\varepsilon}_{s} ( \nabla \varphi )
    X^{\varepsilon}_{s,t} Y ] -\mathbbm{E} [ u_{s} ( \nabla \varphi ) X_{s,t}
    Y ] | & \leqslant & | \mathbbm{E} [ u^{\varepsilon}_{s} ( \nabla \varphi )
    X_{s,t} Y ] -\mathbbm{E} [ u_{s} ( \nabla \varphi ) X_{s,t} Y ] |\\
    &  & +\mathbbm{E} [ | u^{\varepsilon}_{s} ( \nabla \varphi ) | |
    X_{s,t}^{\varepsilon} -X_{s,t} | | Y | ]\\
    & \lesssim & | \mathbbm{E} [ u^{\varepsilon}_{s} ( \nabla \varphi )
    X_{s,t} Y ] -\mathbbm{E} [ u_{s} ( \nabla \varphi ) X_{s,t} Y ] |\\
    &  & +\mathbbm{E} [ | X_{s,t}^{\varepsilon} -X_{s,t} |^{p} ]^{1/p}
    \mathbbm{E} [ | Y |^{q} ]^{1/q} .
  \end{eqnarray*}
  Since $\mathbbm{E} [ \| \mathbf{X}^{\varepsilon} -\mathbf{X}
  \|^{p}_{\mathcal{R}^{\gamma'}} ] \rightarrow_{\varepsilon \rightarrow 0} 0$
  and $| u^{\varepsilon}_{s} ( \nabla \varphi ) | \leqslant \| u_{0}
  \|_{L^{\infty} ( \Omega \times \mathbbm{R}^{d} )} \| \nabla \varphi
  \|_{L^{1} ( \mathbbm{R}^{d} )}$, for all $1 \leqslant p<+ \infty$ and all
  $s,t \in \Pi$, $u^{\varepsilon}_{s} ( \varphi ) X^{\varepsilon}_{s,t}$
  converge weakly in $L^{p} ( \Omega )$ to $u_{s} ( \varphi ) X_{s,t}$. The
  same computation holds for $u^{\varepsilon}_{s} ( \nabla^{2} \varphi )
  X^{\varepsilon}_{s,t}$ and $u_{s} ( \nabla^{2} \varphi ) X_{s,t}$ and
  $u^{\varepsilon}_{s} ( \nabla^{2} \varphi ) \mathbbm{X}^{\varepsilon}_{s,t}$
  and $u_{s} ( \nabla^{2} \varphi ) \mathbbm{X}_{s,t}$.
  
  Furthermore, for all $s,t \in \Pi$
  \[ r^{\varepsilon}_{s,t} ( \nabla \varphi ) =u^{\varepsilon} ( \nabla
     \varphi )_{t} -u^{\varepsilon} ( \nabla \varphi )_{s}
     -u^{\varepsilon}_{s} ( \nabla^{2} \varphi ) X^{\varepsilon}_{s,t} . \]
  Hence $r^{\varepsilon}_{s,t} ( \nabla \varphi ) \rightharpoonup r_{s,t} (
  \nabla \varphi )$ weakly in $L^{p}$ for all $1 \leqslant p<+ \infty$ and we
  have for all $Y \in L^{q} ( \Omega )$,
  \[ \mathbbm{E} [ ( r_{s,t} ( \nabla \varphi ) - ( u ( \nabla \varphi )_{t}
     -u ( \nabla \varphi )_{s} ) -u_{s} ( \nabla^{2} \varphi ) X_{s,t} ) Y ]
     =0. \]
  Hence, almost surely for all $s,t \in \Pi$,
  \[ u ( \nabla \varphi )_{t} -u ( \nabla \varphi )_{s} =u_{s} ( \nabla
     \varphi )' X_{s,t} +r_{s,t} ( \nabla \varphi ) . \]
  thanks to the same limiting procedure, the previous equation is true for all
  $s,t \in [0,T]$. Furthermore, by the same computation, we also have that for
  all $s,t \in [0,T]$.
  \[ \mathbbm{E} [ | r_{s,t} |^{2p} ] \lesssim_{p} | t-s |^{4 \gamma p} . \]
  Therefore, by Kolmogorov's continuity theorem, $r_{s,t} \in \mathcal{C}^{2
  \gamma'}$. Hence, almost surely $u ( \nabla \varphi )$ is
  $\gamma'$-controlled by $X$ and $\| u ( \nabla \varphi )
  \|_{\mathcal{D}^{\gamma'}_{X}} \in L^{p} ( \Omega )$ for all $1 \leqslant
  p<+ \infty$.
  
  For all $n \geqslant 0$, $k \geqslant n$ and all $t \in \cap_{k \geqslant n}
  \Pi_{k}$, there exists $i_{t}^{k} \in \{ 0, \ldots ,2^{k} \}$ such that
  $t=t^{k}_{i^{k}_{t}}$. Then we define
  \[ S^{\varepsilon}_{k} ( \nabla \varphi ,t ) = \sum^{i_{t}^{k} -1}_{i=0} [
     u_{t^{k}_{i}}^{\varepsilon} ( \nabla \varphi ) \delta
     X^{\varepsilon}_{t_{i}^{k} ,t^{k}_{i+1}} +u_{t^{k}_{i}}^{\varepsilon} (
     \nabla^{2} \varphi ) \mathbbm{X}^{\varepsilon}_{t^{k}_{i} ,t^{k}_{i+1}} ]
  \]
  and $S_{k} ( \nabla \varphi ,t )$ as the same quantity for $u$ and
  $\mathbf{X}$. Thanks to the definitions of $u$, we know that
  $S^{\varepsilon}_{k} ( \nabla \varphi ,r )$ converges weakly in all $L^{p} (
  \Omega )$ for $1 \leqslant p<+ \infty$ to $S_{k} ( \nabla \varphi ,r )$.
  Furthermore, thanks to the estimates for the rough integrals, and since $\|
  u^{\varepsilon} ( \nabla \varphi )
  \|_{\mathcal{D}^{\gamma'}_{X^{\varepsilon}}} \lesssim ( 1+ \| \mathbf{X}
  \|_{\mathcal{R}^{\gamma'}} )^{1+1/ \gamma}$, we have
  \begin{eqnarray*}
    \left| \int_{0}^{t} u_{q}^{\varepsilon} ( \nabla \varphi ) \dd
    \mathbf{X}_{q}^{\varepsilon} -S^{\varepsilon}_{k} ( \nabla \varphi ,t )
    \right| & \lesssim & \| u_{.}^{\varepsilon} ( \nabla \varphi )
    \|_{\mathcal{D}^{\gamma'}_{X^{\varepsilon}}} \| \mathbf{X}^{\varepsilon}
    \|_{\mathcal{R}^{\gamma'}} 2^{- ( 3 \gamma' -1 )}\\
    & \lesssim & ( 1+ \| \mathbf{X} \|_{\mathcal{R}^{\gamma'}} )^{2+1/
    \gamma'} 2^{- ( 3 \gamma' -1 )}
  \end{eqnarray*}
  and
  \begin{eqnarray*}
    \left| \int_{0}^{t} u_{q} ( \nabla \varphi ) \dd \mathbf{X}_{q}
    -S_{k} ( \nabla \varphi ,t ) \right| & \lesssim & \| u_{.} ( \nabla
    \varphi ) \|_{\mathcal{D}^{\gamma'}_{X}} \| \mathbf{X}
    \|_{\mathcal{R}^{\gamma'}} 2^{- ( 3 \gamma' -1 )} .
  \end{eqnarray*}
  Hence, for all $1<q \leqslant + \infty$, and $Y \in L^{q} ( \Omega )$,
  \begin{eqnarray*}
    \lefteqn{\left| \mathbbm{E} \left[ \left( \int_{0}^{t} u_{q}^{\varepsilon} ( \nabla
    \varphi ) \dd \mathbf{X}_{q}^{\varepsilon} - \int_{0}^{t} u_{q} (
    \nabla \varphi ) \dd \mathbf{X}_{q} \right) Y \right] \right|^{p} }\\
    &
    \leqslant & \mathbbm{E} \left[ \left| \int_{0}^{t} u_{q}^{\varepsilon} (
    \nabla \varphi ) \dd \mathbf{X}_{q}^{\varepsilon}
    -S^{\varepsilon}_{k} ( \nabla \varphi ,t ) \right|^{p} \right] \| Y
    \|^{p}_{L^{q} ( \Omega )}+ | \mathbbm{E} [ ( S^{\varepsilon}_{k} ( \nabla \varphi ,t ) -S_{k}
    ( \nabla \varphi ,t ) ) Y ] |^{p}\\
    &  & +\mathbbm{E} \left[ \left| \int_{0}^{t} u_{q} ( \nabla \varphi )
    \dd \mathbf{X}_{q}^{\varepsilon} -S_{k} ( \nabla \varphi ,t )
    \right|^{p} \right] \| Y \|^{p}_{L^{q} ( \Omega )} .
  \end{eqnarray*}
  As $\varepsilon \rightarrow 0$ and $k \rightarrow + \infty$, the right hand
  side of the previous inequality goes to zero, and for all $t \in \Pi$,
  $\int_{0}^{t} u_{q}^{\varepsilon} ( \nabla \varphi ) \dd
  \mathbf{X}_{q}^{\varepsilon}$ converges weakly to $\int_{0}^{t} u_{q} (
  \nabla \varphi ) \dd \mathbf{X}_{q}$ in $L^{p} ( \Omega )$ for all $1
  \leqslant p<+ \infty$.
  
  Furthermore, as for all $t \in [0,T]$
  \[ \left| \int_{0}^{t} u_{q}^{\varepsilon} ( \nabla \varphi ) \dd
     \mathbf{X}_{q}^{\varepsilon} \right| \lesssim \left( 1+ \|
     u^{\varepsilon} ( \nabla \varphi )
     \|_{\mathcal{D}^{\gamma'}_{X^{\varepsilon}}} \right) ( ( 1+ \|
     \mathbf{X}^{\varepsilon} \|_{\mathcal{R}^{\gamma'}} ) ) \lesssim ( 1+
     \| \mathbf{X} \|_{\mathcal{R}^{\gamma'}} )^{2+1/ \gamma'} \]
  and the same kind of bound holds for $\left| \int_{0}^{t} u_{q} ( \nabla
  \varphi ) \dd \mathbf{X}_{q} \right|$. By a similar computation to
  {\eqref{eq:riemann_weak_cv}}, for all $1 \leqslant p<+ \infty$, all $Y \in
  L^{q} ( \Omega )$ and all $f \in C^{\infty}_{c} ( [0,T] )$,
  \[ \mathbbm{E} \left[ \int_{0}^{T} f_{t} \int_{0}^{t} u_{q}^{\varepsilon} (
     \nabla \varphi ) \dd \mathbf{X}_{q}^{\varepsilon} \dd t Y \right]
     \rightarrow \mathbbm{E} \left[ \int_{0}^{T} f_{t} \int_{0}^{t} u_{q} (
     \nabla \varphi ) \dd \mathbf{X}_{q} \dd t Y \right] . \]
  Hence all terms of the approximate equation converge and we have, for all
  test functions $f$ and $Y$,
  \[ \mathbbm{E} \left[ \int_{0}^{T} \left( u_{t} ( \varphi ) -u_{0} ( \varphi
     ) - \int_{0}^{t} u_{q} ( \tmop{div} ( b_{q} \varphi ) ) \dd q-
     \int_{0}^{t} u_{q} ( \nabla \varphi ) \dd \mathbf{X}_{q} \right)
     f_{t} \dd t Y \right] =0. \]
  Almost surely and for almost all $t \in [0,T]$,
  \[ u_{t} ( \varphi ) =u_{0} ( \varphi ) + \int_{0}^{t} u_{q} ( \tmop{div} (
     b_{q} \varphi ) ) \dd q+ \int_{0}^{t} u_{q} ( \nabla \varphi ) \dd
     \mathbf{X}_{q} . \]
  Hence, $u \in L^{\infty} ( \Omega \times [0,T] \times \mathbbm{R}^{d} )$ is
  a weak controlled solution of the rough transport equation driven by
  $\mathbf{X}$.
\end{proof}
\section{Uniqueness of solutions}\label{section:uniqueness}
In order to prove the uniqueness of weak controlled solutions, we will use a
duality argument. Indeed, we will suppose that everything is smooth, and that
$\psi$ is a strong solution of the Continuity Equation
\begin{equation*}
  \partial_{t} \psi + \tmop{div} ( b ) \psi +b/ \nabla \psi   \noplus + \nabla
  \psi . \dot{X} =0 
\end{equation*}
with $\psi_{t} \in C^{\infty}_{c} ( \mathbbm{R}^{d } )$ for all $t \in [ 0,T
]$, we have for any weak solution of the rough transport equation, using the
Leibniz rule on $u_{t} ( \psi_{t} ) = \langle u_{t} , \psi_{t} \rangle$,
\begin{eqnarray*}
  \partial_{t} ( u_{t} ( \psi_{t} ) ) & = & \partial_{t} u_{t} ( \psi_{t} )
  +u_{t} ( \partial_{t} ( \psi_{t} ) )\\
  & = & u_{t} (  \tmop{div}   ( b_{t} \psi_{t} ) ) +u_{t} ( \nabla \psi_{t} )
  \dot{X}_{t} -u_{t} ( \tmop{div} ( b \psi_{t} ) - \nabla \psi_{t} \dot{X}_{t}
  )\\
  & = & 0.
\end{eqnarray*}
Hence $u_{t} ( \psi_{t} ) =u_{0} ( \psi_{0} )$. As the equation is linear, it
is enough to prove uniqueness when $u_{0} =0$, then $u_{t} ( \psi_{t} ) =0$.
The trick here is to solve the rough continuity equation 
backward, such that for any fixed $\varphi \in C^{\infty}_{c} ( \mathbbm{R}^{d
} )$ and any $t \in [0,T]$, there exists a solution of the rough continuity equation such that $\psi_{t} = \varphi$. Hopefully the work for such an existence has been done in Section \ref{section:existence} and in subsection \ref{subsection:strong}.

As the transport equation is linear, it is enough to prove uniqueness when
$u_{0} =0$. We would like to use the standard duality argument to prove that
in that case the only solution is zero. As we need to test the weak controlled
solution against smooth compactly supported functions, it is not possible to
do it directly. The idea is to approximate the vectorfield $b$ with a smooth
one, and hence to show that the error we make by such a trick goes to zero
when the regularization goes to zero. This is the purpose of the lemma of the following section, which allows us to compare different weak controlled solution associated to different vector fields.
	\subsection{The fundamental lemma}

\begin{lemma}[Fundamental Lemma]
  \label{lemma:uniqueness}Let $1/3< \gamma \leqslant 1/2$, $\mathbf{X}= (
  X,\mathbbm{X} ) \in \mathcal{R}^{\gamma}$ be a geometric rough path, $b \in L^{\infty} ( [0,T];
  \tmop{Lin} ( \mathbbm{R}^{d } ) )$ be two vector spaces such that $\tilde{b} \in L^{\infty} (
  [0,T];C^{\infty}_{c} ( \mathbbm{R}^{d} ) )$ with $\tmop{div}  b \nocomma \in
  L^{\infty} ( [0,T];L^{\infty} ( \mathbbm{R}^{d} ) )$. Let $\tilde{\Phi}$ the
  flow associated to $X$ and $\tilde{b}$, and let $\tilde{R}$ the radius of
  the ball of Lemma \ref{lemma:bounded_bounded_drift} associated to
  $\tilde{b}$ and $X$.
  
  Let $u \in L^{\infty} ( [0,T] \times \mathbbm{R}^{d} )$ a weak controlled
  solution of the rough transport equation with $u_{0} =0$, and
  \[ \tilde{G}_{q}^{t_{0}} ( x ) = \int_{q}^{t_{0}} ( \tmop{div}  
     \tilde{b}_{r} ) ( \tilde{\Phi}_{r-q} ( x ) ) \dd r. \]
  Then
  \[ | u_{t_{0}} ( \varphi_{0} ) | \lesssim \| u \|_{\infty} \sup_{q \in [
     0,t_{0} ]} \int_{B_{f} ( 0, \tilde{R} )} \dd x  [ | \tmop{div}   ( b-
     \tilde{b} )_{q} |+| ( b- \tilde{b} )_{q} ( x ) | ] ( | D
     \tilde{\Phi}_{t_{0} -q} ( x ) | + | \nabla \widetilde{G^{}}^{t_{0}}_{q} (
     x ) | ) . \]
\end{lemma}

\begin{proof}
  Let $\varphi_{0} \in C^{\infty}_{c}$ and $\tilde{b} \in C^{\infty}_{b}$. For
  all $t \in [0,T]$ we define,
  \[ \tilde{\psi}_{t} :x \rightarrow \tilde{\psi}_{t} ( x ) = \varphi_{0} (
     \tilde{\Phi}^{-1}_{t} ( x ) ) \exp \left[ - \int_{0}^{t} ( \tmop{div}  
     \tilde{b} ) ( \tilde{\Phi}^{-1}_{q} ( x ) ) \dd q \right] \in
     C^{\infty}_{c} ( \mathbbm{R}^{d} ) . \]
  By Theorem \ref{theorem:strong_controlled_solutions}, $\tilde{\psi}$ is a
  strong controlled solution of the rough continuity equation.
  
  Furthermore, thanks to the definition of $\tilde{R}$, for all $n \in
  \mathbbm{N}$
  \[ t \rightarrow D^{n} \tilde{\psi}_{t} ( x ) \in \mathcal{D}^{\gamma}_{X} (
     [0,T] ) \]
  and
  \[ \| D^{n} \tilde{\psi}_{.} ( x ) \|_{\mathcal{D}^{\gamma}_{X}}
     \lesssim_{n} \mathbbm{1}_{B_{f} ( 0, \tilde{R} )} ( x ) . \]
  Since $\nabla \tilde{\psi}_{.} ( x ) \in \mathcal{D}^{\gamma} ( [0,T] )$ and
  the function $x \rightarrow \| \nabla \tilde{\psi}_{.} ( x )
  \|_{\mathcal{D}^{\gamma} ( [0,T] )}$ is compactly supported, we have
  \begin{eqnarray*}
    u_{t} ( \tilde{\psi}_{t} ) -u_{s} ( \tilde{\psi}_{s} ) & = & ( u_{t}
    -u_{s} ) ( \tilde{\psi}_{s} ) +u_{s} ( \tilde{\psi}_{t} - \tilde{\psi} ) +
    ( u_{t} -u_{s} ) ( \tilde{\psi}_{t} - \tilde{\psi}_{s} )\\
    & = & \int_{s}^{t} u_{q} ( b. \nabla \tilde{\psi}_{s} + ( \tmop{div}  b )
    \tilde{\psi}_{s} ) \dd q+u_{s} ( \nabla \tilde{\psi}_{s} ) .X_{s,t} +
    \frac{1}{2} u_{s} ( \nabla^{2} \tilde{\psi}_{s} ) .X^{\otimes 2}_{s,t}\\
    &  & - \int_{s}^{t} u_{s} ( \tilde{b} . \nabla \tilde{\psi}_{q} + (
    \tmop{div}   \tilde{b} ) \tilde{\psi}_{s} ) \dd q-u_{s} ( \nabla
    \tilde{\psi}_{s} ) .X_{s,t} + \frac{1}{2} u_{s} ( \nabla^{2}
    \tilde{\psi}_{s} ) .X^{\otimes 2}_{s,t}\\
    &  & + \int_{s}^{t} u_{q} ( b. ( \nabla \tilde{\psi}_{t} - \nabla
    \tilde{\psi}_{s} ) + ( \tmop{div}  b ) ( \tilde{\psi}_{t} -
    \tilde{\psi}_{s} ) ) \dd q+u_{s} ( \nabla \tilde{\psi}_{t} - \nabla
    \tilde{\psi}_{s} ) .X_{s,t}\\
    &  & +R_{s,t} ,
  \end{eqnarray*}
  where $| R_{s,t} | \lesssim | t-s |^{3 \gamma}$.
  
  But
  \[ u_{s} ( \nabla \tilde{\psi}_{t} - \nabla \tilde{\psi}_{s} ) =-u_{s} (
     \nabla^{2} \tilde{\psi}_{s} ) .X_{s,t} + \tilde{R}_{s,t} \]
  and all the rough terms cancelled. Finally we have
  \begin{eqnarray*}
    u_{t} ( \tilde{\psi}_{t} ) -u_{s} ( \tilde{\psi}_{s} ) & = & \int_{s}^{t}
    u_{q} ( ( b- \tilde{b} ) . \nabla \tilde{\psi}_{q} ) \dd q+
    \int_{s}^{t} u_{q} ( \tmop{div}   ( b- \tilde{b} ) \tilde{\psi}_{q} )
    \dd q\\
    &  & + \int_{s}^{t} u_{q} ( b. ( \nabla \tilde{\psi}_{t} - \nabla
    \tilde{\psi}_{q} ) ) \dd q+ \int_{s}^{t} u_{q} ( ( \tmop{div}  b ) (
    \tilde{\psi}_{t} - \tilde{\psi}_{q} ) ) \dd q\\
    &  & + \int_{s}^{t} ( u_{q} -u_{s} ) ( \tilde{b} . \nabla
    \tilde{\psi}_{q} + ( \tmop{div}   \tilde{b} ) \tilde{\psi}_{q} ) \dd q+
    \tilde{R}_{s,t} .
  \end{eqnarray*}
  Furthermore
  \[ \left| \int_{s}^{t} u_{q} ( b. ( \nabla \tilde{\varphi}_{t} - \nabla
     \tilde{\psi}_{q} ) ) \dd q \right| \lesssim_{\varepsilon ,T, \varphi
     ,b, \| X \|_{\gamma}} \| u \|_{\infty} \int_{s}^{t} | t-q |^{\gamma}
     \dd q \lesssim | t-s |^{1+ \gamma} , \]
  \[ \left| \int_{s}^{t} u_{q} ( \tmop{div}  b ( \tilde{\psi}_{t} -
     \tilde{\psi}_{q} ) ) \dd q \right| \lesssim | t-s |^{1+ \gamma} , \]
  and since $\tilde{b} . \nabla \tilde{\psi}_{q} + ( \tmop{div}   \tilde{b} )
  \tilde{\psi}_{q} \in C^{\infty}_{c} ( \mathbbm{R}^{d} )$ and $u$ is a WCS
  \[ | ( u_{q} -u_{s} ) ( \tilde{b} . \nabla \tilde{\psi}_{q} + ( \tmop{div}  
     \tilde{b} ) \tilde{\psi}_{q} ) | \lesssim | q-s |^{\gamma} . \]
  Hence, we have the following decomposition
  \[ u_{t} ( \tilde{\psi}_{t} ) -u_{s} ( \tilde{\psi}_{s} ) = \int_{s}^{t}
     u_{q} ( ( b- \tilde{b} ) . \nabla \tilde{\psi}_{q} + \tmop{div} ( b-
     \tilde{b} ) \tilde{\psi}_{q} ) \dd q+ \tilde{R}_{s,t} \]
  where $| \tilde{R}_{s,t} | \lesssim | t-s |^{3 \gamma}$. But, thanks to the
  last equation,
  \[ \delta \tilde{R}_{r,s,t} =0 \]
  and the Lemma \ref{th:sewing_map} gives
  \[ \tilde{R}_{s,t} =0. \]
  Finally, we have
  \begin{eqnarray}
    u_{t} ( \tilde{\psi}_{t} ) -u_{s} ( \tilde{\psi}_{s} ) & = & \int_{s}^{t}
    u_{q} ( ( b- \tilde{b} ) . \nabla \tilde{\psi}_{q} + \tmop{div} ( b-
    \tilde{b} ) \tilde{\psi}_{q} ) \dd q \nonumber\\
    & = & \int_{s}^{t} u_{q} ( \tmop{div} [ ( b_{q} - \tilde{b}_{q} )
    \tilde{\psi}_{q} ] ) \dd q.  \label{eq:increment_duality_method}
  \end{eqnarray}
  As Equation {\eqref{eq:RTE}} is linear, in order to prove the uniqueness of
  the solutions, it is enough to prove it when $u_{0} =0$. As stated above, we
  need backward solutions of the Continuity Equation. When $\tilde{b}$ is
  bounded smooth, for $\varphi_{0} \in C^{\infty}_{c} ( \mathbbm{R}^{d} )$, $x
  \rightarrow \varphi_{0} ( x ) = \varphi_{0} ( \tilde{\Phi}_{t_{0}} ( x ) )
  \exp \left( \int_{0}^{t_{0}} ( \tmop{div}   \tilde{b}_{q} ) (
  \tilde{\Phi}_{q} ( x ) ) \dd q \right)$ is smooth, compactly supported
  and $\tilde{\psi}$ is a solution of the rough continuity equation.
  Furthermore $\tilde{\psi}_{t_{0}} ( x ) = \varphi_{0} ( x )$, and
  \[ \tilde{\psi}_{t} ( x ) = \varphi_{0} ( \tilde{\Phi}_{t_{0} -t} ( x ) )
     \exp \left( \int_{t}^{t_{0}} ( \tmop{div}   \tilde{b}_{q} ) (
     \tilde{\Phi}_{q-t} ( x ) ) \dd q \right) . \]

  Hence, we can choose $t=t_{0}$ in {\eqref{eq:increment_duality_method}}
  and $s=0$, and we have
  \[ u_{t_{0}} ( \varphi_{0} ) = \int_{0}^{t_{0}} u_{q} ( \tmop{div} [ ( b_{q}
     - \tilde{b}_{q} ) \tilde{\psi}_{q} ] ) \dd q. \]
  We can split the right hand side into three parts
  \[ A_{1} = \int_{0}^{t_{0}} u_{q} ( \tmop{div} ( b- \tilde{b} )_{q}
     \tilde{\psi}_{q} ) \dd q, \]
  \[ A_{2} = \int_{0}^{t} \dd q \int_{\mathbbm{R}^{d}} \dd x u_{q} ( x )
     ( b- \tilde{b} )_{q} ( x ) . \nabla \varphi_{0} ( \tilde{\Phi}_{t-q} ( x
     ) ) .D \tilde{\Phi}_{t-q} ( x ) \exp \left( \int_{q}^{t} ( \tmop{div}  
     \tilde{b}_{r} ) ( \tilde{\Phi}_{r-q} ( x ) ) \dd q \right) \dd q,
  \]
  and
  \[ A_{3} = \int_{0}^{t} \dd q \int_{\mathbbm{R}^{d }} \dd x u_{q} ( x
     ) \tilde{\psi}_{q} ( x ) ( b- \tilde{b} )_{q} ( x ) \int_{q}^{t} \nabla (
     \tmop{div}   \tilde{b}_{r} ) ( \tilde{\Phi}_{r-q} ( x ) ) .D
     \tilde{\Phi}_{r-q} ( x ) \dd q. \]

  Let us recall that $| \tilde{\psi}_{q} ( x ) | \leqslant \| \varphi_{0}
  \|_{\infty} \exp ( T \| \tmop{div}   \tilde{b} \|_{\infty} )
  \mathbbm{1}_{B_{f} ( 0, \tilde{R} )} ( x )$ where $\tilde{R}$ is
  nondecreasing in $T$, $\| X \|_{\gamma}$ and $\| \tilde{b} \|_{\infty}$.
  Hence
  \begin{eqnarray*}
    | A_{1} | & \leqslant & \| u \|_{\infty} T  \sup_{q \in [0,T]}
    \int_{\mathbbm{R}^{d}} \dd x  | \tmop{div}  b_{q} - \tmop{div}  
    \tilde{b}_{q} | \mathbbm{1}_{B_{f} ( 0, \tilde{R} )} ( x ) .
  \end{eqnarray*}
  For $A_{2}$ we will use the same trick: as $| \nabla \varphi_{0} (
  \tilde{\Phi}_{t-q} ( x ) ) | \leqslant \| \nabla \varphi_{0} \|_{\infty}
  \mathbbm{1}_{B_{f} ( 0, \tilde{R} )} ( x )$, we have
  \[ | A_{2} | \lesssim   \sup_{q \in [0,T]} \int_{\mathbbm{R}^{d}} \dd x |
     b_{q} ( x ) - \tilde{b}_{q} ( x ) | \mathbbm{1}_{B_{f} ( 0, \tilde{R} )}
     ( x ) | D \tilde{\Phi}_{t-q} ( x ) | . \]
  The same holds for $A_{3}$, and we have
  \begin{eqnarray*}
    | A_{3} | & \lesssim & \sup_{q \in [0,T]} \int_{\mathbbm{R}^{d}} \dd x
    | b_{q} ( x ) - \tilde{b}_{q} ( x ) | \mathbbm{1}_{B_{f} ( 0, \tilde{R} )}
    ( x ) \left| \int_{q}^{t} \nabla ( \tmop{div}   \tilde{b}_{r} ) (
    \tilde{\Phi}_{r-q} ( x ) ) .D \tilde{\Phi}_{r-q} ( x ) \dd r \right|
  \end{eqnarray*}
  Once put together, this gives the wanted result.
\end{proof}

\begin{remark}
  \label{remark:decomposition_weak_solution}In fact, the proof gives us a
  decomposition of $u_{q} ( \varphi_{0} )$ as the follows:
  \[   \tilde{\psi}_{t} ( x ) = \varphi_{0} ( \tilde{\Phi}_{t_{0} -t} ( x ) )
     \exp \left( \int_{t}^{t_{0}} ( \tmop{div}   \tilde{b}_{q} ) (
     \tilde{\Phi}_{q-t} ( x ) ) \dd q \right) \]
  and every weak controlled solution of the rough transport equation verifies
  \[ u_{t_{0}} ( \varphi_{0} ) = \int_{0}^{t_{0}} u_{q} ( \tmop{div} [ ( b_{q}
     - \tilde{b}_{q} ) \tilde{\psi}_{q} ] ) \dd q. \]
\end{remark}
	\subsection{Strong uniqueness}

In the case of the fractional Brownian motion, a phenomenon of regularization
by noise will occur. But even without any regularization, we have the
following theorem.

\begin{theorem}\label{theorem:strong_uniqueness}
  Let $b \in L^{\infty} ( [0,T];L^{\infty} ( \mathbbm{R}^{d} ) \cap C^{1} (
  \mathbbm{R}^{d} ) )$, $\delta >0$ and $\tmop{div}  b \in L^{\infty} ( [0,T]
  \times \mathbbm{R}^{d} )$, $1/3< \gamma \leqslant 1/2$ and $\mathbf{X} \in
  \mathcal{R}^{\gamma} ( [0,T] )$. There exists a unique weak controlled
  solution with $u_{0} \in L^{\infty} ( \mathbbm{R}^{d} )$ to the rough
  transport equation.
\end{theorem}

\begin{proof}
  As $b \in L^{\infty} ( [0,T]; \tmop{Lin} ( \mathbbm{R}^{d} ) )$, and since
  $\tmop{div}  b \in L^{\infty} ( [0,T] \times \mathbbm{R}^{d} )$ and
  $\mathbf{X} \in \mathcal{R}^{\gamma}$, there exists a weak controlled
  solution. Furthermore, $b$ is locally Lipschitz continuous in the second
  variable and $b$ has linear growth in the second variable. It is well-known
  that there exists a unique solution $\Phi$ to the equation
  \[ \Phi_{t} ( x ) =x+ \int_{0}^{t} b_{q} ( \Phi_{q} ( x ) ) \dd q+X_{q} .
  \]
  Furthermore, $\Phi$ is differentiable in space, and its spatial derivative
  satisfies the following equation
  \[ D \Phi_{t} ( x ) = \tmop{id} + \int_{0}^{t} D b_{q} ( \Phi_{q} ( x ) ) .D
     \Phi_{q} ( x ) \dd q. \]
  Furthermore, for for $r>0$ and $x \leqslant r$
  \[ | D \Phi_{t} ( x ) | \lesssim \sup_{q \in [0,T]} e^{\sup_{y \in B ( r+ \|
     X \|_{\infty} +T \| b \|_{\infty} )}   | D b_{q} ( y ) |} . \]
  In order to prove uniqueness, with the notations of Remark
  \ref{remark:decomposition_weak_solution}, we only have to check that there
  exists a sequence $b^{\varepsilon} \in C^{\infty} ( \mathbbm{R}^{d} )$ such
  that
  \begin{equation}
    | u_{t_{0}} ( \varphi ) | = \left| \int_{0}^{t_{0}} u_{q} ( \tmop{div} [ (
    b_{q} -b^{\varepsilon}_{q} ) \psi^{\varepsilon}_{q} ] ) \dd q \right|
    \rightarrow_{\varepsilon \rightarrow 0} 0
    \label{eq:strong_uniqueness_estimate} .
  \end{equation}
  Let $k \in C_{c}^{\infty} ( \mathbbm{R}^{d} ,\mathbbm{R} )$ a
  mollifier,\tmtextit{i.e.} $k \geqslant 0$ such that $\int_{\mathbbm{R}^{d}}
  k ( x ) \dd x=1$ and $k ( x ) =k ( -x )$. Let $k_{\varepsilon} ( x ) =
  \frac{1}{\varepsilon^{d}} k ( x/ \varepsilon )$ and $b^{\varepsilon}
  =k_{\varepsilon} \ast b$. Hence, $b^{\varepsilon} \in C_{b}^{\infty} (
  \mathbbm{R}^{d} )$ with $\| \tmop{div}  b^{\varepsilon} \|_{\infty}
  \leqslant \| \tmop{div}  b \|_{\infty}$. Let us recall, thanks to Lemma
  \ref{lemma:smooth_control} and since $\varphi \in C^{\infty}$ and
  $b^{\varepsilon} \in C^{\infty}_{b} ( \mathbbm{R}^{d} ) \cap \tmop{Lin} (
  \mathbbm{R}^{d} )$, for all $t \in [0,T]$, $x \rightarrow \varphi (
  \Phi^{\varepsilon}_{t} ( x ) ) \in C^{\infty}_{c} ( \mathbbm{R}^{d} )$ and
  since $\| b^{\varepsilon} \|_{\infty} \lesssim \| b \|_{\infty}$, for all $N
  \in \mathbbm{N}$, there exists $R>0$ independent of $\varepsilon$ such that
  for all $t \in [0,T]$,
  \[ \tmop{supp}   \nobracket \varphi_{0} ( \Phi^{\varepsilon}_{t} ( . ) ) |
     \subset B ( 0,R ) . \]
  Furthermore, as $b^{\varepsilon} \in C^{\infty}$, for all $t \in [0,T]
  \nocomma , \Phi^{\varepsilon}_{t} \in C^{1} ( \mathbbm{R}^{d} )$ and for $x
  \leqslant R$
  \[ \sup_{\varepsilon >0} \sup_{t \in [0,T]} | D \Phi^{\varepsilon}_{q} ( x )
     | \leqslant \sup_{q \in [0,T]} e^{\sup_{y \in B ( R+ \| X \|_{\infty} +T
     \| b \|_{\infty} )}   | D b_{q} ( y ) |} . \]
  Furthermore, as $b \in L^{\infty} ( [0,T];C^{1} ( \mathbbm{R}^{d} ) \cap
  L^{\infty} ( \mathbbm{R}^{d} ) )$, by localization
  \[ \sup_{\varepsilon >0}   \varepsilon^{-1} \sup_{q \in [0,T] ,x \in B ( 0,R
     )} | b_{q} ( x ) -b_{q}^{\varepsilon} ( x ) | <+ \infty . \]
  In order to prove the theorem when $\tmop{div}  b \in L^{\infty} (
  \mathbbm{R}^{d} )$, we need to use an approximation argument. As all the
  function are localized in a ball of radius $R$, let $\eta >0$ and let
  $\theta \in C^{\infty} ( [0,T] \times \mathbbm{R}^{d} )$ such that $\| (
  \theta -u ) \mathbbm{1}_{B ( 0,R )} \|_{L^{1} ( [0,T] \times \mathbbm{R}^{d}
  )} < \eta$. We have
  \begin{eqnarray*}
    \int_{0}^{t_{0}} \dd q \left| \int_{B ( 0,R )} \dd x  \theta_{q} ( x
    ) \tmop{div} ( ( b_{q} -b^{e}_{q} ) \psi^{\varepsilon}_{q} ) ( x ) \right|
    & = & \int_{0}^{t_{0}} \dd q \left| \int_{B ( 0,R )} \dd x  \nabla
    \theta_{q} ( x ) . ( b_{q} -b^{e}_{q} ) ( x ) \psi^{\varepsilon}_{q} ( x )
    \right|\\
    & \lesssim & \| b-b^{\varepsilon} \|_{L^{\infty} ( [0,T] \times B ( 0,R )
    )} \| \nabla \theta \|_{L^{\infty} ( [0,T] \times B ( 0,R ) )} .
  \end{eqnarray*}
  On the other hand, we have

  \begin{multline*}
    \int_{0}^{t_{0}} \dd q \left| \int_{B ( 0,R )} \dd x ( u_{q} ( x ) -
    \theta_{q} ( x ) ) \tmop{div} ( ( b ( x ) -b^{\varepsilon} ( x ) )
    \psi^{\varepsilon}_{q} ( x ) ) \right|\\  
    \lesssim  \| u- \theta \|_{L^{\infty} ( [0,T] \times B ( 0,R ) )}
     \| \tmop{div} ( ( b_{q} -b^{e}_{q} ) \psi^{\varepsilon}_{q} )
    \|_{L^{\infty} ( [0,T] \times B ( 0,R ) )} .
  \end{multline*}
  
  But
  \[ \mathbbm{1}_{B ( 0,R )} ( x ) \tmop{div} ( ( b_{q} -b^{e}_{q} )
     \psi^{\varepsilon}_{q} ) ( x ) \leqslant e^{T \| \tmop{div}  b
     \|_{L^{\infty} ( [0,T] \times \mathbbm{R}^{d} )}} ( A^{\varepsilon}_{1} (
     x ) +A^{\varepsilon}_{2} ( x ) +A^{\varepsilon}_{3} ( x ) ) , \]
  where
  \[ A^{\varepsilon}_{1} = | ( \tmop{div}  b_{q} - \tmop{div} 
     b^{\varepsilon}_{q} ) ( x ) | | \varphi ( \Phi^{\varepsilon}_{t_{0} -q} (
     x ) ) | \mathbbm{1}_{B ( 0,R )} ( x ) \lesssim \| \tmop{div}  b_{q}
     \|_{L^{\infty} ( [0,T] \times \mathbbm{R}^{d} )} , \]
  \[ A^{\varepsilon}_{2} = | b_{q} ( x ) - b^{\varepsilon}_{q} ( x ) | |
     \nabla \varphi ( \Phi^{\varepsilon}_{t_{0} -q} ( x ) ) | | D
     \Phi^{\varepsilon}_{t_{0} -q} ( x ) | \mathbbm{1}_{B ( 0,R )} ( x )
     \lesssim \sup_{q \in [0,T]} \| b_{q} \mathbbm{1}_{B ( 0,R )}
     \|_{L^{\infty} ( \mathbbm{R}^{d} )} \]
  and
  \begin{eqnarray*}
    A^{\varepsilon}_{3} & = & | b_{q} ( x ) - b^{\varepsilon}_{q} ( x ) | |
    \varphi ( \Phi^{\varepsilon}_{t_{0} -q} ( x ) ) | \int_{q}^{t_{0}} \nabla
    \tmop{div}  b^{\varepsilon} ( \Phi^{\varepsilon}_{r-q} ( x ) ) | D
    \Phi^{\varepsilon}_{r-q} ( x ) | \dd r \mathbbm{1}_{B ( 0,R )} ( x )\\
    & \lesssim & \sup_{q,r \in [0,T]} \| ( b_{q} - b^{\varepsilon}_{q} )
    \mathbbm{1}_{B ( 0,R )} \nabla \tmop{div}  b^{\varepsilon} (
    \Phi^{\varepsilon}_{r} ( . ) ) \|\\
    & \lesssim & 1,
  \end{eqnarray*}
  since $\| ( b_{q} - b^{\varepsilon}_{q} ) \mathbbm{1}_{B ( 0,R )}
  \|_{\infty} \lesssim \varepsilon$ and $\| \nabla \tmop{div}  b^{\varepsilon}
  ( \Phi^{\varepsilon}_{r} ( . ) ) \mathbbm{1}_{B ( 0,R )} \| \lesssim 1/
  \varepsilon$. Hence
  \[ \int_{0}^{t_{0}} \dd q \left| \int_{B ( 0,R )} \dd x ( u_{q} ( x )
     - \theta_{q} ( x ) ) \tmop{div} ( ( b ( x ) -b^{\varepsilon} ( x ) )
     \psi^{\varepsilon}_{q} ( x ) ) \right| \lesssim \eta . \]
  Finally thanks to Remark~\ref{remark:decomposition_weak_solution},
  \begin{eqnarray*}
    | u_{t_{0}} ( \varphi ) | & \leqslant & \int_{0}^{t_{0}} \dd q \left|
    \int_{B ( 0,R )} \dd x ( u_{q} ( x ) - \theta_{q} ( x ) ) \tmop{div} (
    ( b ( x ) -b^{\varepsilon} ( x ) ) \psi^{\varepsilon}_{q} ( x ) )
    \right|\\
    &  & + \int_{0}^{t_{0}} \dd q \left| \int_{B ( 0,R )} \dd x 
    \theta_{q} ( x ) \tmop{div} ( ( b_{q} -b^{e}_{q} ) \psi^{\varepsilon}_{q}
    ) ( x ) \right|\\
    & \lesssim & \eta + \| b-b^{\varepsilon} \|_{L^{\infty} ( [0,T] \times B
    ( 0,R ) )} \| \nabla \theta \|_{L^{\infty} ( [0,T] \times B ( 0,R ) )} .
  \end{eqnarray*}
  Now $\varepsilon$ goes to zero, and $u_{t_{0}} ( \varphi ) =0$ for all
  $\varphi \in C^{\infty}_{c} ( \mathbbm{R}^{d} )$, which is exactly the
  wanted result.
\end{proof}

\begin{remark}
  If $b \in \tmop{Lin} ( \mathbbm{R}^{d} )$ the last argument does not work.
  Indeed, we were not able to prove that in that case the the functions $x
  \mapsto \varphi ( \Phi^{\varepsilon}_{t} ( x ) )$ is compactly supported. In
  order to fix that, one solution might be to localize the function $b$ before
  mollifying it. Let $\theta \in C^{\infty}_{c} ( \mathbbm{R}^{d} )$ such that
  $\theta ( x ) =1$ for $| x | \leqslant 1$ and $\theta ( x ) =0$ for $| x |
  \geqslant 2$. For $r>0$ let us define $\theta_{r} ( x ) = \theta ( x/r )$
  and
  \[ b^{\varepsilon ,r}_{q} ( x ) =k_{\varepsilon} \ast ( b_{q} \theta_{r} ) .
  \]
  In order to complete the duality argument in Lemma $\ref{lemma:uniqueness}$
  we would have to test $u$ against $\psi^{\varepsilon ,r} \theta_{r}$. In
  that procedure other remainder terms should appear.
\end{remark}

\begin{corollary}
  With the hypothesis of the previous theorem, for $u_{0} \in C^{3}_{b} (
  \mathbbm{R}^{d } )$ there exists a unique strong controlled solution $u$ to
  the rough transport equation. Furthermore, for all $( t,x ) \in [0,T] \times
  \mathbbm{R}^{d}$, $u_{t} ( x ) =u_{0} ( \Phi^{-1}_{t} ( x ) )$.
\end{corollary}

\begin{proof}
  Thanks to Theorem \ref{theorem:strong_controlled_solutions}, for $u_{0} \in
  C^{3}_{b} ( \mathbbm{R}^{d} )$ we know that $( t,x ) \rightarrow u_{0} (
  \Phi^{-1}_{t} ( x ) )$ is a strong controlled solution, and then a weak
  controlled solution. But the previous theorem guarantees that there is only
  one weak controlled solution.
\end{proof}
\subsection{Regularization by noise}
	In the article by Catellier and Gubinelli {\cite{catellier_averaging_2012}}, as
we presented in Subsection \ref{subsec:regularization}, that when the process $X$ is
$\rho$-irregular, a phenomenon of regularization occurs. Indeed, for less
regular vectorfields $b$ the flow of the equation exists, and furthermore its
averaging properties are nice. We will give two different results: for a
general $\rho$-irregular path, and for the fractional Brownian motion.

\subsubsection{General $\rho -$irregular paths}

\begin{theorem}
  Let $\frac{1}{3} < \gamma \leqslant 1$, $\mathbf{X}= ( X,\mathbbm{X} ) \in
  \mathcal{R}^{\gamma} ( [0,T] )$, $\rho >0$ such that $X$ is
  $\rho$-irregular. Let $\alpha >- \rho$ such that $\alpha +3/2>0$ and $b \in
  \mathcal{F}L^{\alpha +3/2} ( \mathbbm{R}^{d} )$ and $\tmop{div}  b  \in
  \mathcal{F}L^{\alpha +3/2}$. Let $u_{0} \in L^{\infty} ( \mathbbm{R}^{d} )$.
  There exists a unique weak controlled solution to the Rough Transport
  Equation with initial condition $u_{0}$.
\end{theorem}

\begin{proof}
  Thanks to Theorem \ref{theorem:rho_irregular_flow}, the flow $\Phi$ of the
  equation
  \[ \Phi_{t} ( x ) =x+ \int_{0}^{t} b ( \Phi_{q} ( x ) ) \dd q+X_{t} \]
  exists. Furthermore, we know that for a mollification $b^{\varepsilon}$ of
  $b$ such that $\| b-b^{\varepsilon} \|_{\mathcal{F}L^{\alpha +3/2}}
  \rightarrow 0$,
  \[ \sup_{\varepsilon >0}   \sup_{t \in [0,T] ,x \in \mathbbm{R}^{d}} | D
     \Phi^{\varepsilon}_{q} ( x ) | <+ \infty . \]
  As $\alpha +3/2>0$,
  \[ | b ( x ) -b ( y ) | \lesssim | x-y |^{\alpha +3/2} \| b
     \|_{\mathcal{F}L^{\alpha +3/2}} . \]
  Hence $b \in \tmop{Lin} ( \mathbbm{R}^{d} )$, $\tmop{div}  b \in L^{\infty}
  ( \mathbbm{R}^{d} )$ and there exists a weak controlled solution of the
  rough transport equation. Finally for all $t \in [0,T]$, thanks to
  Proposition \ref{prop:fBm_flow} the function
  \[ ( G^{\varepsilon} )^{t_{0}}_{q} :x \rightarrow \int_{q}^{t_{0}} (
     \tmop{div}  b^{\varepsilon} ) ( \Phi^{\varepsilon}_{r-q} ( x ) ) \dd q
  \]
  is differentiable and
  \[ \sup_{\varepsilon >0}   \sup_{q \leqslant t_{0} \in [0,T]} (
     G^{\varepsilon} )^{t_{0}}_{q} ( | x | ) \lesssim K ( | x | ) N_{\alpha
     +1} ( \tmop{div}  b ) . \]
  Thanks to Lemma \ref{lemma:uniqueness}, it is enough to prove that
  \[ A_{\varepsilon} = \sup_{q \in [0,T]} \int_{\mathbbm{R}^{d}}
     \mathbbm{1}_{B_{f} ( 0,R^{\varepsilon} )} ( x ) | b ( x )
     -b^{\varepsilon} ( x ) | ( | D \Phi^{\varepsilon}_{t_{0} -q} ( x ) | + |
     \nabla   ( G^{\varepsilon} )^{t_{0}}_{q} ( x ) | ) \rightarrow 0 \]
  and
  \[ B_{\varepsilon} = \sup_{q \in [0,T]} \int_{\mathbbm{R}^{d}}
     \mathbbm{1}_{B_{f} ( 0,R^{\varepsilon} )} ( x ) | \tmop{div}  b ( x ) -
     \tmop{div}  b^{\varepsilon} ( x ) | ( | D \Phi^{\varepsilon}_{t_{0} -q} (
     x ) | + | \nabla   ( G^{\varepsilon} )^{t_{0}}_{q} ( x ) | ) . \]
  Since $\alpha +3/2>0$, $b \in C^{0}_{b} ( [0,T] )$, $R^{\varepsilon}
  \lesssim R$ where $R^{\varepsilon}$ and $R$ are the radii defined in Lemma
  \ref{lemma:bounded_bounded_drift}, and $\| b-b^{\varepsilon} \|_{\infty}
  \rightarrow 0$. Furthermore, thanks to Theorem
  \ref{theorem:rho_irregular_flow}, $( | D \Phi^{\varepsilon}_{t_{0} -q} ( x )
  | + | \nabla   ( G^{\varepsilon} )^{t_{0}}_{q} ( x ) | ) \lesssim K ( | x |
  )$ uniformely in $\varepsilon$. Hence $A_{\varepsilon} \rightarrow 0$.
  Furthermore, $\tmop{div}  b^{\varepsilon} =k_{\varepsilon} \ast ( \tmop{div}
  b )$ and since $\alpha +3/2>0$, $\tmop{div}  b^{\varepsilon}
  \rightarrow^{L^{\infty}}   \tmop{div}  b^{\varepsilon}$, and then the result
  follows.
\end{proof}

As the proof of the last theorem is pathwise, one can prove the following
corollary where $b$ and $u_{0}$ are random and $X$ a generic continuous and
$\rho$-irregular stochastic process which lifts into $\mathcal{R}^{\gamma}$
and with good approximation properties in $L^{p} ( \Omega )$.

\begin{corollary}\label{theorem:regu_uniqueness}
  Let $\rho >0$, $\alpha >- \rho$ such that $3/2+ \alpha >0$ and $1/3< \gamma
  \leqslant 1.2$. Let $( \Omega ,\mathcal{F},\mathbbm{P} )$ a probability
  space, X a continuous stochastic process on $( \Omega
  ,\mathcal{F},\mathbbm{P} )$ such that $X$ is almost surely $\rho$-irregular
  and almost surely $X$ lifts in a measurable way to $\mathbf{X} \in
  \mathcal{R}^{\gamma}$. Suppose furthermore that for any smooth approximation
  $\mathbf{X}^{\varepsilon}$ of $\mathbf{X}$ and any $1 \leqslant p<+
  \infty$, $\mathbbm{E} [ \| \mathbf{X}^{\varepsilon} -\mathbf{X}
  \|_{\mathcal{R}^{\gamma}}^{p} ] \rightarrow 0$.
  
  Let $b  \in L^{\infty} ( \Omega \times \mathbbm{R}^{d} )$, $\tmop{div}  b
  \in L^{\infty} ( \Omega \times \mathbbm{R}^{d} )$ such that almost surely $b
  ( \omega ,. ) \in \mathcal{F}L^{\alpha +3/2}$ and $\tmop{div}  b \in
  \mathcal{F}L^{\alpha +3/2}$. Let $u_{0} \in L^{\infty} ( \Omega \times
  \mathbbm{R}^{d} )$.
  
  There exists a unique Stochastic weak controlled solution $u \in L^{\infty}
  ( \Omega \times [0,T] \times \mathbbm{R}^{d} )$ of the Rough Transport
  Equation with initial condition $u_{0}$.
\end{corollary}

\begin{proof}
  The existence of such a solution is proved in Theorem
  \ref{theorem:existence_SWCS}. For the uniqueness result, remark that the
  lift $\mathbf{X}$ of $X$ and the functions $b$ and $u_{0}$ satisfy the
  conditions of the last theorem almost surely. Hence, the solution is unique.
\end{proof}

\begin{remark}
  For $H \in ( 0,1 )$, $- \alpha < \rho <1/2H$ with $\alpha +3/2>0$, the
  fractional Brownian motion $B^{H}$ satisfies the hypothesis of the last
  theorem. The last result allows us to take random vectorfields.
\end{remark}

\subsubsection{Fractional Brownian motion, the h{\"o}lder continuous case}

The last proofs, and in particular the proof of the Lemma
$\ref{lemma:uniqueness}$ relies on the existence of a flow for the
characteristic equation associated to the vectorfield $b$ and the path $X$. In
order to prove uniqueness, a uniform bound on the differential of the flow for
regularized vectorfields $\tilde{b}$ is also needed. In
{\cite{catellier_averaging_2012}}, the authors use a Girsanov transform for
the fractional Brownian motion in order to extend the space of vectorfields,
namely from $\mathcal{F}L^{\alpha}$ to $\mathcal{C}^{\beta}$, the space of
H{\"o}lder continuous function.

\begin{theorem}\label{theorem:unique}
  Let $H \in ( 1/3,1/2 ]$ and $\alpha >-1/2H$ with $\alpha +1>0$. Let $(
  \Omega ,\mathcal{F},\mathbbm{P} )$ a probability space and $B^{H}$ a
  $d$-dimensional fractional Brownian motion of Hurst parameter $H$ associated
  to the probability space and $\mathbf{B}^{H}$ its natural lift. Let $b \in
  \mathcal{C}^{\alpha +1}_{b} ( \mathbbm{R}^{d} )$ with $\tmop{div}  b  \in 
  \mathcal{C}^{\alpha +1}_{b} ( \mathbbm{R}^{d} )$, and $u_{0} \in L^{\infty}
  ( \mathbbm{R}^{d} )$. There exists a unique Stochastic weak controlled
  solution $u$ of the Stochastic rough transport equation driven by
  $\mathbf{B}^{H}$ with initial condition $u_{0}$.
\end{theorem}

\begin{proof}
  The existence holds thanks to Theorem \ref{theorem:existence_SWCS}.

  Let $-1/2H< \alpha' < \alpha$ such that $\alpha' +1>0$. Let
  $b^{\varepsilon}$ a mollification of $b$. We know that $\| b-b^{\varepsilon}
  \|_{\infty} \rightarrow 0$. Let $\varphi_{0} \in C^{\infty}_{c} (
  \mathbbm{R}^{d} )$ a test function and the radius $R^{\varepsilon}$
  associated to the flow $\Phi^{\varepsilon}$ and $\varphi_{0}$. As
  $R^{\varepsilon} =R_{\varphi_{0}} +2T \| b^{\varepsilon} \|_{\infty} + \|
  B^{H} \|_{\infty} \leqslant R=R_{\varphi_{0}} +2T \| b \|_{\infty} + \|
  B^{H} \|_{\infty}$. Hence, when $u_{0} =0$, we have, thanks to the Lemma
  \ref{lemma:uniqueness},
  
\begin{multline}
    \nobracket u_{0} ( \varphi_{0} ) | \lesssim ( \| b-b^{\varepsilon}
    \|_{\infty} + \| \tmop{div}  b -  \tmop{div}  b^{\varepsilon} \|_{\infty}
    )\\
    \times
    \int_{\mathbbm{R}^{d}} \dd x \mathbbm{1}_{B_{f} ( 0,R )} ( x )
    \sup_{q \in [ 0,t_{0} ]} ( | D \Phi^{\varepsilon}_{t_{0} -q} ( x ) | + |
    \nabla ( G^{\varepsilon} )^{t_{0}}_{q} ( x ) | )
    \label{eq:uniqueness_bound_stochastic_case} .
  \end{multline}

  Furthermore, thanks to Theorem \ref{theorem:fBm_flow} and Proposition
  \ref{prop:fBm_flow},
  \[ \mathbbm{E} [ \mathbbm{1}_{B_{f} ( 0,R )} ( x ) \sup_{q \in [ 0,t_{0} ]}
     ( | D \Phi^{\varepsilon}_{t_{0} -q} ( x ) | + | \nabla ( G^{\varepsilon}
     )^{t_{0}}_{q} ( x ) | ) ] \lesssim \mathbbm{E} [ \mathbbm{1}_{B_{f} ( 0,R
     )} ( x ) ]^{1/2} K ( | x | ) . \]
  But
  \[ \mathbbm{E} [ \mathbbm{1}_{B_{f} ( 0,R )} ( x ) ] =\mathbbm{P} (
     R_{\varphi_{0}} +T \| b \|_{\infty} + \| B^{H} \|_{\infty} \geqslant | x
     | ) \]
  and thanks to Fernique's theorem {\cite{da_prato_stochastic_2014}}, $\exp (
  2 \| B^{H} \|_{\infty} ) \in L^{1} ( \Omega )$ for all $p \geqslant 1$, so
  that by the inquelaity of Markov
  \[ \mathbbm{E} [ \mathbbm{1}_{B_{f} ( 0,R )} ( x ) ] \lesssim \exp ( -2 | x
     | ) . \]
  Hence
  \[ \mathbbm{E} [ \mathbbm{1}_{B_{f} ( 0,R )} ( x ) \sup_{q \in [ 0,t_{0} ]}
     ( | D \Phi^{\varepsilon}_{t_{0} -q} ( x ) | + | \nabla ( G^{\varepsilon}
     )^{t_{0}}_{q} ( x ) | ) ] \leqslant \exp ( - | x | ) K ( | x | ) \]
  and
  \[ \int \mathbbm{1}_{B_{f} ( 0,R )} ( x ) \sup_{q \in [ 0,t_{0} ]} ( | D
     \Phi^{\varepsilon}_{t_{0} -q} ( x ) | + | \nabla ( G^{\varepsilon}
     )^{t_{0}}_{q} ( x ) | ) \dd x \in L^{1} ( \Omega ) \]
  and is finite almost surely. By letting $\varepsilon$ go to zero in
  {\eqref{eq:uniqueness_bound_stochastic_case}}, almost surely $| u_{0} (
  \varphi_{0} ) | =0$, which ends the proof.
\end{proof}

\appendix
\section{Standard theory of flows for additive rough ODEs}\label{appendix:standard_flow}

Our approach for the study of rough transport equations comes mostly from the
method of characteristics. This method requires lots of informations about the
regularity of the flow of the characteristic differential equation linked to
the transport equation. We give her some standard theorem about flows of differential equations. Although these properties are quite standard, we give them with some proof.
\subsection{Standard theory for flows}

The following Theorem is standard, but we recall the form of the Jacobian
determinant of the flow. All along this subsection, the perturbation $X$ will
be a path lying in $\mathcal{C}^{\gamma}$ and $\gamma \in ( 0,1 ]$.

\begin{theorem}
  \label{theorem:jacobien}Let $b \in L^{\infty} ( [0,T]; \tmop{Lip} (
  \mathbbm{R}^{d} ) \cap C^{1} ( \mathbbm{R}^{d} ) )$ and $X \in
  \mathcal{C}^{\gamma} ( [0,T] )$. The equation
  \begin{equation}
    \left\{\begin{array}{l}
      \dd y_{t} =b_{t} ( y_{t} ) + \dd X_{t}\\
      y_{0} =x
    \end{array}\right. , \label{eq:differential_annex}
  \end{equation}
  understood in its integral form
  \[ y_{t} =x+ \int_{0}^{t} b_{s} ( y_{s} ) \dd s+X_{t} -X_{0} \]
  has a unique solution $\Phi ( x ) \in \mathcal{C}^{\gamma} ( [0,T] )$.
  Furthermore, in that case, for all $t \in [0,T]$, $x \rightarrow \Phi_{t} (
  x )$ is differentiable and its derivative $D \Phi_{t}$ is the unique
  solution of the linear ordinary differential equation
  \[ D \Phi_{t} ( x ) = \tmop{id} + \int_{0}^{t} D b_{s} ( \Phi_{s} ( x ) ) .D
     \Phi_{s} ( x ) \dd s. \]
  In that case, we have
  \[ \tmop{Jac} ( \Phi_{t} ( x ) ) = \exp \left( \int_{0}^{t} \tmop{div} 
     b_{s} ( \Phi_{s} ( x ) ) \dd s \right) . \]
\end{theorem}

The first assertions are quite standard. The proof of the last assertion
relies on the so called Liouville Lemma. 

In the classical case of the transport equation, the existence of a solution
is granted when the vectorfield $b$ has spatial linear growth. The reminder of
this section presents a brief study of flows when $b$ has linear growth. 

As in the standard theory of transport equations, the solutions will involve
the inverse of the flow of equation {\eqref{eq:differential_annex}}. The following
lemma gives informations about the equation verified by this inverse.

\begin{proposition}
  \label{proposition:ODE_inverse_flow}Let $b$ and $X$ such that that the
  equations {\eqref{eq:ODE_classic}} and {\eqref{eq:reverse_time}} below
  have unique solution (for all $t_{0} \in [0,T]$):
  \begin{equation}
    \Phi_{t} ( x ) =x+ \int_{0}^{t} b_{q} ( \Phi_{u} ( x ) ) \dd u+X_{t}
    \label{eq:ODE_classic} \hspace{1em} t \in [ 0,T ] .
  \end{equation}
  and
  \begin{equation}
    \psi_{t}^{t_{0}} ( y ) =y- \int_{0}^{t} b_{t_{0} -q} ( \psi^{t_{0}}_{u} (
    y ) ) \dd u- ( X_{t_{0}} -X_{t_{0} -t} ) , \label{eq:reverse_time}
    \hspace{1em} t \in [ 0,t_{0} ]
  \end{equation}
  Then for all $t_{0} \in [0,T]$, $\Phi^{-1}_{t_{0}} ( x )$ exists, and
  $\Phi^{-1}_{t_{0}} ( x ) = \psi_{t_{0}} ( x )$ where $\psi$ is the unique
  solution of the equation. Furthermore, $\Phi^{-1}$ also verifies the
  following equation:
  \[ \Phi^{-1}_{t} ( x ) =x- \int_{0}^{t} b_{T-u} ( \Phi^{-1}_{u} ( x ) )
     \dd u- ( X_{T} -X_{T-t} ) \nocomma , \hspace{1em} t \in [ 0,T ] . \]
\end{proposition}

\begin{proof}
  We have
  \begin{eqnarray*}
    \Phi_{t_{0} -t} ( x ) & = & x+ \int_{0}^{t_{0} -t} b_{u} ( \Phi_{u} ( x )
    ) \dd u+X_{t_{0} -t}\\
    & = & x+ \int_{t}^{t_{0}} b_{t_{0} -u} ( \Phi_{t_{0} -u} ( x ) ) \dd
    u+X_{t_{0} -t}\\
    & = & \Phi_{t_{0}} ( x ) - \int_{0}^{t} b_{t_{0} -u} ( \Phi_{t_{0} -u} (
    x ) ) \dd u- ( X_{t_{0}} -X_{t_{0} -t} )\\
    &  & +x+X_{t_{0}} + \int_{0}^{t_{0}} b_{t_{0} -u} ( \Phi_{u} ( x ) )
    \dd u- \Phi_{t_{0}} ( x )\\
    & = & \Phi_{t_{0}} ( x ) - \int_{0}^{t} b_{t_{0} -u} ( \Phi_{t_{0} -u} (
    x ) ) \dd u- ( X_{t_{0}} -X_{t_{0} -t} ) .
  \end{eqnarray*}
  Hence, since the solution of equation {\eqref{eq:reverse_time}} is unique,
  $\Phi_{t_{0} -t} ( x ) = \psi_{t}^{t_{0}} ( \Phi_{t_{0}} ( x ) )$, 
  and we
  have $\psi_{t_{0}}^{t_{0}} ( \Phi_{t_{0}} ( x ) ) =x$ which is the wanted
  result. As equation {\eqref{eq:reverse_time}} is of same form as equation
  {\eqref{eq:ODE_classic}}, we also have $\Phi_{t_{0}} \circ
  \psi^{t_{0}}_{t_{0}} = \tmop{id} .$ Hence $\Phi^{-1}_{t_{0}}$ exists and we
  have $\Phi^{-1}_{t_{0}} = \psi^{t_{0}}_{t_{0}}$. Furthermore for $s,t \in
  [0,T]$
  \[ \Phi_{s+t} \circ \Phi_{s}^{-1} \circ \Phi^{-1}_{t} = \Phi_{t} \circ
     \Phi_{s} \circ \Phi^{-1}_{s} \circ \Phi^{-1}_{t} = \tmop{id} . \]
  Hence \ $\Phi_{s}^{-1} \circ \Phi^{-1}_{t} = \Phi_{s+t}^{-1} = \Phi_{t}^{-1}
  \circ \Phi^{-1}_{s}$ and $\Phi^{-1}$ has the semi-group property.
  Furthermore,
  \begin{eqnarray*}
    x- \Phi_{t}^{-1} ( x ) & = & \Phi_{T} ( \Phi^{-1}_{T} ( x ) ) - \Phi_{T-t}
    ( \Phi_{T}^{-1} ( x ) )\\
    & = & \int_{T-t}^{T} b_{r} ( \Phi_{r} ( \Phi^{-1}_{T} ( x ) ) ) \dd
    r+X_{T} -X_{T-t}\\
    & = & \int_{T-t}^{T} b_{r} ( \Phi_{T-r}^{-1} ( x ) ) \dd r+X_{T}
    -X_{T-t}\\
    & = & \int_{0}^{t} b_{T-u} ( \Phi_{u}^{-1} ( x ) ) \dd u+X_{T}
    -X_{T-t} ,
  \end{eqnarray*}
  and the result follows.
\end{proof}

As in the method of characteristics (see Appendix
\ref{appendix:characteristics}), in the following, we will strongly use the
spatial regularity of the flow. When the vector-field $b$ is regular enough,
the following result gives a bound for the norm of the spatial derivative.

\begin{proposition}
  \label{propo:flot_n_diff}Let $b: [0,T] \times \mathbbm{R}^{d} \rightarrow
  \mathbbm{R}^{d}$, $N \in \mathbbm{N}$,and $b \in L^{\infty} (
  [0,T];C^{N}_{b} ( \mathbbm{R}^{d} ) )$. Let $X \in \mathcal{C}^{\gamma} (
  [0,T] )$ and $\Phi$ the flow of the equation
  \[ \Phi_{t} ( x ) =x+ \int_{0}^{t} b_{r} ( \Phi_{r} ( x ) ) \dd r+X_{t} .
  \]
  For all $k \in \{ 1, \ldots N \}$ and all $t \in [0,T]$, the $k$-th spatial
  derivative $D^{k} \Phi_{t}$ exists, and there exists a constant $C_{k}$
  depending on $T$ $\| X \|_{\mathcal{D}^{\gamma}}$ and $\| D^{k} b \|_{\infty
  ;L^{\infty}}$ for all $k \in \{ 0, \ldots ,n \}$ such that
  \[ | D^{k} \Phi_{t} ( x ) | \leqslant C_{k} . \]
\end{proposition}

\begin{proof}
  The fact that $\Phi_{t} \in C^{N} ( \mathbbm{R}^{d} )$ is standard. Let us
  recall the equation verified by the flow
  \[ D \Phi_{t} ( x ) = \tmop{id} + \int_{0}^{t} D b_{r} ( \Phi_{r} ( x ) ) .D
     \Phi_{r} ( x ) \dd r. \]
  Hence, the proof of the proposition is quite straightforward. We have to use
  the equation verified by $D^{n} \Phi ( x )$, where all the other derivatives
  of $\Phi$ and all the derivatives of $b$ up to order $n$ appear. We then
  prove by induction the wanted bound for $| D^{n} \Phi_{t} ( x ) |$, using
  the bound for the smaller derivatives of $\Phi$ and the Gronwall lemma.
\end{proof}

\begin{remark}
  The proof also gives that $C_{k}$ has exponential growth with respect to $\|
  D b \|_{\infty ;L^{\infty}}$ but polynomial growth wrt the other quantities.
\end{remark}


\subsection{A priori bound for flows: proof of Lemmas \ref{lemma:holder_norm_flow} and \ref{lemma:flow_convergence_X_smooth}}

\begin{lemma}[Lemma \ref{lemma:holder_norm_flow}] Let $b \in L^{\infty} ( [0,T], \tmop{Lin} (
  \mathbbm{R}^{d} ) )$ such that the flow $\Phi$ of the equation
  {\eqref{eq:differential}} exists for all $t \in [ 0,T ]$. There exists a constant
  $K ( T, \| b \|_{\infty ; \tmop{Lin}} ) >0$ such that for all $x \in
  \mathbbm{R}^{d}$, $\Phi ( x ) \in \mathcal{C}^{\gamma} ( [0,T] )$ and
  \[ \| \Phi ( x ) \|_{\gamma} \leqslant K ( 1+ | x | ) ( 1+ \| X \|_{\gamma}
     ) . \]
\end{lemma}

\begin{proof}
  Let $x \in \mathbbm{R}^{d}$ and $0 \leqslant s<t \leqslant T$
  \begin{eqnarray*}
    \frac{| \Phi ( x )_{t} - \Phi ( x )_{s} |}{| t-s |^{\gamma}} & \leqslant &
    \| X \|_{\gamma} + \frac{1}{| t-s |^{\gamma}} \int_{s}^{t} | b_{r} (
    \Phi_{r} ( x ) ) | \dd r\\
    & \lesssim & \| X \|_{\gamma} + \| b \|_{\infty ; \tmop{Lin}} \left(
    T^{1- \gamma} ( 1+ | \Phi_{s} ( x ) | ) + \int_{s}^{t} \frac{| \Phi_{r} (
    x ) - \Phi_{s} ( x ) |}{| r-s |^{\gamma}} \dd r \right) .
  \end{eqnarray*}
  By Gronwall's lemma,
  \[ \frac{| \Phi ( x )_{t} - \Phi ( x )_{s} |}{| t-s |^{\gamma}} \leqslant (
     \| X \|_{\gamma} +T^{1- \gamma} ( 1+ | \Phi_{s} ( x ) | ) \| b \|_{\infty
     ; \tmop{Lin}} ) \exp ( T \| b \|_{\infty ; \tmop{Lin}} ) , \]
  hence
  \[ \llbracket \Phi ( x ) \rrbracket_{\gamma} \leqslant ( \| X \|_{\gamma}
     +T^{1- \gamma} ( 1+ \| \Phi ( x ) \|_{\infty} ) \| b \|_{\infty ;
     \tmop{Lin}} ) \exp ( T \| b \|_{\infty ; \tmop{Lin}} ) . \]
  It remains to estimate the supremum of $\Phi ( x )$. We have
  \[ | \Phi ( x )_{t} | \leqslant | x | + \| b \|_{\infty ; \tmop{Lin}} \left(
     T+ \int_{0}^{t} | \Phi ( x )_{r} | \dd r \right) + \| X \|_{\infty} ,
  \]
  which gives, again by Gronwall's lemma
  \[ \| \Phi ( x ) \|_{\infty} \leqslant ( | x | +T \| b \|_{\infty ;
     \tmop{Lin}} + \| X \|_{\infty} ) \exp ( T \| b \|_{\infty ; \tmop{Lin}} )
     . \]
  The result follows after recalling that $\| X \|_{\infty} \lesssim \| X
  \|_{\gamma}$.
\end{proof}

\begin{remark}
  The constant $K$ can be chosen as $K ( T, \| b \|_{\infty ; \tmop{Lin}} )
  =K_{T}  g ( T \| b \|_{\infty ; \tmop{Lin}} )$ with $g ( x ) = ( ( x^{2} +x
  ) e^{x} +x+1 ) e^{x}$.
\end{remark}

The heart of the Section \ref{section:uniqueness} is a smooth approximation of
the characteristics equations of the transport equation by an approximate one.
Indeed, when $X \notin C^{1} ( [0,T] )$ and $b \notin C^{1} ( \mathbbm{R}^{d}
)$, the above results of Proposition \ref{propo:flot_n_diff} about the differentiability of the flow can not be used.
The following result gives the regularity of the flow regarding the
perturbation $X$.

\begin{lemma}[Lemma \ref{lemma:flow_convergence_X_smooth}] Let $b \in L^{\infty} (
  [0,T];C^{1}_{b} ( \mathbbm{R}^{d } ) )$ and $X,Y \in \mathcal{C}^{\gamma}$.
  Then
  \[ \| \Phi^{X} ( x ) - \Phi^{Y} ( x ) \|_{\gamma , [ 0,T ]} \leqslant C (
     T, \| D b \|_{\infty} ) \| X-Y \|_{\gamma , [ 0,T ]} \]
  where $C$ is independent of $x$ and nondecreasing in $T$ and $\| D b
  \|_{\infty}$. Furthermore, when $b \in L^{\infty} ( [0,T];C^{2}_{b} (
  \mathbbm{R}^{d } ) )$, we also have
  \[ \| D \Phi^{X} ( x ) -D \Phi^{Y} ( x ) \|_{\tmop{Lip}} \leqslant C ( T, \|
     D b \|_{\infty} , \| D^{2}  b \|_{\infty} ) \| X-Y \|_{\gamma} . \]
\end{lemma}

\begin{proof}
  First, we give an estimate of the difference in supremum norm.
  \begin{eqnarray*}
    | \Phi^{Y}_{t} ( x ) - \Phi^{X}_{t} ( x ) | & \leqslant & | X_{t} -Y_{t} |
    + \int_{0}^{t} | b_{q} ( \Phi^{Y}_{q} ( x ) ) -b_{q} ( \Phi^{X}_{q} ( x )
    ) | \dd q\\
    & \leqslant & t^{\gamma} \| X-Y \|_{\gamma} + \| D b \|_{\infty}
    \int_{0}^{t} | \Phi^{Y}_{q} ( x ) - \Phi_{q}^{X} ( x ) | \dd q
  \end{eqnarray*}
  By Gronwall's lemma,
  \[ | \Phi^{Y}_{t} ( x ) - \Phi_{t}^{X} ( x ) | \leqslant T^{\gamma} \| X-Y
     \|_{\gamma} e^{T \| D b \|_{\infty}} . \]
  Furthermore, we have
  \begin{eqnarray*}
    | \delta ( \Phi^{X} - \Phi^{Y} )_{s,t} ( x ) | & \leqslant & | \delta (
    X-Y )_{s,t} | + \| D b \|_{\infty} \int_{s}^{t} | ( \Phi^{X} - \Phi^{Y}
    )_{s,q} ( x ) | + | ( \Phi^{X} - \Phi^{Y} )_{s} ( x ) | \dd q.
  \end{eqnarray*}
  Again by Gronwall's lemma, we get
  \[ | \delta ( \Phi^{X} - \Phi^{Y} )_{s,t} ( x ) | \lesssim_{T, \| D b
     \|_{\infty}} | t-s |^{\gamma} \| X-Y \|_{\gamma} , \]
  which is the first part of the result. Furthermore, as
  \[ D  \Phi^{X}_{t} ( x ) = \tmop{id} + \int_{0}^{t} D b_{q} ( \Phi^{X}_{q}
     ( x ) ) .D  \Phi^{X}_{q} ( x ) \dd q, \]
  we have
  \[ \| D \Phi^{X} ( x ) \|_{\infty} \leqslant e^{T \| D b \|_{\infty}} \]
  and
  \begin{eqnarray*}
    | D( \Phi^{X} - \Phi^{Y} )_{t} ( x ) | & \leqslant & \left| \int_{0}^{t} (
    D b_{q} ( \Phi^{X}_{q} ( x ) ) -D b ( \Phi^{Y}_{q} ( x ) ) ) .D
    \Phi^{X}_{t} ( x ) \dd q \right|\\
    &  & + \left| \int_{0}^{t} D b ( \Phi^{Y}_{q} ( x ) ) .D ( \Phi^{X} -
    \Phi^{Y} )_{q} ( x ) \dd q \right|\\
    & \lesssim & \| D ^{2} b \|   \| \Phi^{X} ( x ) - \Phi^{Y} ( x )
    \|_{\infty} e^{T \| D b \|_{\infty}} + \| D b \|_{\infty} \int_{0}^{t} | D
    ( \Phi^{X} - \Phi^{Y} )_{q} ( x ) | \dd q.
  \end{eqnarray*}
  Hence
  \[ \| D ( \Phi^{X} - \Phi^{Y} )_{.} ( x ) \|_{\infty} \lesssim_{T, \| D b
     \|_{\infty} , \| D ^{2} b \|_{\infty}} \| X-Y \|_{\gamma} . \]
  Finally,
  \begin{eqnarray*}
    | \delta D ( \Phi^{X} - \Phi^{Y} )_{s,t} ( x ) | & \lesssim & \| X-Y
    \|_{\gamma} | t-s | + \| D b \|_{\infty} \int_{s}^{t} | \delta D (
    \Phi^{X} - \Phi^{Y} )_{s,q} ( x ) | \dd q,
  \end{eqnarray*}
  and the result follows by Gronwall's Lemma.
\end{proof}
\section{The method of characteristics}\label{appendix:characteristics}
We give a proof of the usual method of characteristics, and a proof of weak
uniqueness thanks to the dual equation.

\begin{theorem}
  Let $b \in L^{\infty} ( [0,T];C^{1}_{b} ( \mathbbm{R}^{d} ) )$, $c \in
  L^{\infty} ( [0,T],C^{1}_{b} ( \mathbbm{R}^{d} ) )$ and $X \in C^{1} ( [0,T]
  )$. Let $\Phi$ the flow of the equation
  \[ \Phi_{t} ( x ) =x+ \int_{0}^{t} b ( \Phi_{u} ( x ) ) \dd u+X_{t} . \]
  Let $u_{0} \in C^{1} ( \mathbbm{R}^{d} )$. There exists a unique strong
  solution $u$ to equation~{\eqref{eq:equation_smooth}} with initial
  condition $u_{0}$.
  \begin{equation}
    \partial_{t} u+ (b+ \dot{X} ) . \nabla u+c u=0. \label{eq:equation_smooth}
  \end{equation}
  Furthermore the solution has the explicit form
  \[ u_{t} ( x ) =u_{0} ( \Phi_{t}^{-1} ( x ) ) \exp \left( - \int_{0}^{t}
     c_{t-q} ( \Phi^{-1}_{q} ( x ) ) \dd q \right) . \]
\end{theorem}

\begin{proof}
  We first prove the existence of such a solution. As $b \in \tmop{Lip} (
  \mathbbm{R}^{d} ) \cap C^{1} ( \mathbbm{R}^{d} )$, the flow $\Phi$ and its
  inverse are well-defined. Let us define
  \[ u_{t} ( x ) =u_{0} ( \Phi_{t}^{-1} ( x ) ) \exp \left( - \int_{0}^{t}
     c_{t-q} ( \Phi^{-1}_{q} ( x ) ) \dd q \right) . \]
  Then $u_{t} ( \Phi_{t} ( x ) ) =u_{0} ( x ) \exp \left( - \int_{0}^{t} c_{q}
  ( \Phi_{q} ( x ) ) \dd q \right)$ and
  \begin{eqnarray*}
    \partial_{t} ( u_{t} ( \Phi_{t} ( x ) ) ) & = & -u_{0} ( x ) c_{t} (
    \Phi_{t} ( x ) ) \exp \left( - \int_{0}^{t} c_{t-q} ( \Phi^{-1}_{q} ( x )
    ) \dd q \right) =-c_{t} ( \Phi_{t} ( x ) ) u_{t} ( \Phi_{t} ( x ) ) .
  \end{eqnarray*}
  On the other hand, as $b,c \in L^{\infty} ( [0,T],C^{1} ( \mathbbm{R}^{d } )
  )$, $\Phi^{-1} \in C^{1} ( [0,T] \times \mathbbm{R}^{d} )$ and $u \in C^{1}
  ( [0,T] \times \mathbbm{R}^{d} )$. Hence
  \begin{eqnarray*}
    \partial_{t} ( u_{t} ( \Phi_{t} ( x ) ) ) & = & \partial_{t} u_{t} (
    \Phi_{t} ( x ) ) + \nabla u_{t} ( \Phi_{t} ( x ) ) . \Phi'_{t} ( x )\\
    & = & \partial_{t} u_{t} ( \Phi_{t} ( x ) ) + \nabla u_{t} ( \Phi_{t} ( x
    ) ) . ( b ( \Phi_{t} ( x ) ) + \dot{X}_{t} ) .
  \end{eqnarray*}
  and therefore
  \[ \partial_{t} u_{t} ( \Phi_{t} ( x ) ) + \nabla u_{t} ( \Phi_{t} ( x ) ) .
     ( b_{t} ( \Phi_{t} ( x ) ) + \dot{X_{t}} ) =-c_{t} ( \Phi_{t} ( x ) )
     u_{t} ( \Phi_{t} ( x ) ) . \]
  We take $x= \Phi_{t}^{-1} ( y )$, and we have
  \[ u_{t} ( y ) +b_{t} ( y ) . \nabla u_{t} ( y ) +c_{t} ( y ) u_{t} ( y ) +
     \nabla u_{t} ( y ) \dot{X}_{t} =0. \]
  Hence, $u$ is a solution of equation {\eqref{eq:equation_smooth}} with
  initial condition $u_{0}$.
  
  For the uniqueness, let $\varphi$ a solution of equation $\eqref{eq:equation_smooth}$ with $\varphi_{0} =u_{0}$ and let $v_{t}
  ( x ) = \varphi_{t} ( \Phi_{t} ( x ) )$. Then
  \[ \partial_{t} ( v_{t} ( x ) ) =-c_{t} ( \Phi_{t} ( x ) ) v_{t} ( x ) . \]
  We thus have
  \[ \varphi_{t} ( \Phi_{t} ( x ) ) =v_{t} ( x ) =v_{0} ( x ) \exp \left( -
     \int_{0}^{t} c_{q} ( \Phi_{q} ( x ) ) \dd q \right) =u_{0} ( x ) \exp
     \left( - \int_{0}^{t} c_{q} ( \Phi_{q} ( x ) ) \dd q \right) , \]
  and then $\varphi_{t} ( x ) =u_{t} ( x )$ which ends the proof.
\end{proof}

\begin{remark}
  For $c=0$, we have the transport equation. When $c= \tmop{div}  b$, we have
  the continuity equation, the dual equation of the transport equation.
\end{remark}

\begin{remark}
  Thanks to the use of Corollary \ref{corollary:flot_schwarz}, we see that
  when $b$ and $c$ are as in the previous theorem, and $u_{0} \in
  C^{\infty}_{c} ( \mathbbm{R}^{d} )$, and $u \in C^{1} ( \mathbbm{R}^{d} )$
  is the unique solution of equation~{\eqref{eq:equation_smooth}}, then for
  every multi-index $\alpha = ( \alpha_{1} , \ldots , \alpha_{d} ) \in
  \mathbbm{N}^{d}$ and all $t \in [0,T]$,
  \[ x \mapsto x^{\alpha} \nabla u_{t} ( x )   \hspace{1em} \tmop{and}
     \hspace{1em}  x \mapsto x. \nabla u_{t} ( x ) \]
  are in $L^{1} ( \mathbbm{R}^{d} )$ where $x^{\alpha} =x_{1}^{\alpha_{1}}
  \ldots x_{d}^{\alpha_{d}}$.
\end{remark}

We can now prove the existence and the uniqueness of $L^{\infty}$ weak
solutions.

\begin{theorem}
  Let $b \in L^{\infty} ( [0,T];C^{1}_{b} ( \mathbbm{R}^{d} ) )$, $\tmop{div} 
  b \nocomma ,c \in L^{\infty} ( [0,T],C^{1}_{b} ( \mathbbm{R}^{d} ) )$ and $X
  \in C^{1} ( [0,T] )$. Let $\Phi$ the flow of the equation
  \[ \Phi_{t} ( x ) =x+ \int_{0}^{t} b ( \Phi_{u} ( x ) ) \dd u+X_{t} . \]
  Let $u_{0} \in L^{\infty} ( \mathbbm{R}^{d} )$. Then there exists a unique
  weak solution $u \in L^{\infty} ( [ 0,T ] \times \mathbbm{R}^{d} )$ of
  equation {\eqref{eq:weak_equation_smooth}} with initial condition $u_{0}$
  \begin{equation}
    \partial_{t} u+ (b+ \dot{X} ) . \nabla u+c u=0,
    \label{eq:weak_equation_smooth}
  \end{equation}
  furthermore for almost every $t$ and $x$,
  \[ u_{t} ( x ) =u_{0} ( \Phi_{t}^{-1} ( x ) ) \exp \left( - \int_{0}^{t}
     c_{T-q} ( \Phi^{-1}_{q} ( x ) ) \dd q \right) . \]
\end{theorem}

\begin{proof}
  Let $\varphi \in C^{\infty}_{c} ( \mathbbm{R}^{d} )$, and let us define
  \[ \varphi_{t}^{t_{0}} ( x ) = \varphi ( \Phi_{t_{0} -t} ( x ) ) \exp
     \left( \int_{t}^{t_{0}} ( \tmop{div}  b_{q} -c_{q} ) ( \Phi_{q-t} ( x ) )
     \dd q \right) . \]
  Hence, $\varphi^{t_{0}}$ is the unique strong solution of the equation
  \[ \partial_{t} u+ (b+ \dot{X} ) . \nabla u+ ( \tmop{div}  b-c)  u=0 \]
  such that $\varphi_{t_{0}}^{t_{0}} ( x ) = \varphi ( x )$. Furthermore,
  thanks to the previous remark and since $\varphi \in C^{\infty}_{c}$, all
  the following computations are allowed. Let us define as before
  \[ u_{t} ( x ) =u_{0} ( \Phi_{t}^{-1} ( x ) ) \exp \left( - \int_{0}^{t}
     c_{T-q} ( \Phi^{-1}_{q} ( x ) ) \dd q \right) . \]
  We have
  \begin{eqnarray*}
    u_{t_{0}} ( \varphi ) & = & \int_{\mathbbm{R}^{d }} u_{0} (
    \Phi_{t_{0}}^{-1} ( x ) ) \exp \left( - \int_{0}^{t_{0}} c_{t_{0} -q} (
    \Phi^{-1}_{q} ( x ) ) \dd q \right) \varphi ( x )\\
    & = & \int_{\mathbbm{R}^{d}} u_{0} ( x ) \varphi_{0}^{t_{0}} ( x ) \dd
    x\\
    & = & \int_{\mathbbm{R}^{d}} \dd x u_{0} ( x ) \int_{0}^{t_{0}}
    \partial \varphi_{q}^{t_{0}} ( x ) \dd q+u_{0} ( \varphi )\\
    & = & - \int_{\mathbbm{R}^{d}} \dd x u_{0} ( x ) \int_{0}^{t_{0}} (
    b_{q} ( x ) + \dot{X}_{q} ) . \nabla \varphi_{q} ( x ) + ( \tmop{div} 
    b_{q} ( x ) -c_{q} ( x ) ) \varphi_{q} ( x ) \dd q+u_{0} ( \varphi )\\
    & = & \int_{0}^{t_{0}} \dd q \int_{\mathbbm{R}^{d}} \dd x u_{q} ( (
    \nobracket \nobracket b_{q} + \dot{X}_{q} ) . \nabla \varphi_{q} + (
    \tmop{div}  b_{q} -c_{q} ) \varphi ) +u_{0} ( \varphi ) .
  \end{eqnarray*}
  Hence, by definition, $u$ is a weak solution of Equation
  {\eqref{eq:weak_equation_smooth}}. Furthermore, for a weak solution $v$
  with initial condition $0$, by testing against $\varphi$ we have
  \begin{eqnarray*}
    v_{t_{0}} ( \varphi ) & = & v_{t_{0}} ( \varphi_{t_{0}} ) -v_{0} (
    \varphi_{0} )\\
    & = & \int_{0}^{t_{0}} \dd q v_{q} ( ( b_{q} + \dot{X}_{q} ) . \nabla
    \varphi_{t_{0}} + ( \tmop{div}  b_{q} -c_{q} ) \varphi_{t_{0}} ) +v_{q} (
    \varphi_{q} ) -v_{0} ( \varphi_{q} )\\
    & = & 0.
  \end{eqnarray*}
  And this is true for all $\varphi \in C^{\infty}_{c} ( \mathbbm{R}^{d} )$,
  hence $u_{t_{0}} ( x ) =0$ for almost all $x$. As the equation is linear,
  the result is proved.
\end{proof}

\bibliographystyle{plain}
\bibliography{bibli_intro}

\end{document}